\documentclass[a4paper]{amsart}
\usepackage{a4wide,amsfonts,graphicx,
amssymb,stmaryrd,amscd,amsmath,latexsym,amsbsy}
\usepackage{pst-node}
\usepackage{tikz-cd} 

\usepackage{pst-node}
\usepackage{pst-intersect}

\usepackage{etex}
\usepackage{amsfonts}
\usepackage{amsmath}
\usepackage[latin1]{inputenc}
\usepackage{setspace}
\usepackage{mathrsfs}
\usepackage{graphicx}
\graphicspath{ {./pics/} }
\usepackage{caption}
\usepackage{subcaption}
\usepackage{stmaryrd}
\usepackage[english]{babel}
\usepackage{enumitem}
\usepackage{amssymb}
\usepackage{amsthm}
\usepackage{bbding}
\usepackage[all]{xy}
\usepackage[a4paper]{geometry} 

\usepackage{tikz-cd}
\usepackage{pstricks}
\usepackage{pst-slpe}
\usepackage{multido}							
\usepackage{pstricks-add}
\usepackage{pst-plot}

\usepackage{marginnote}

\numberwithin{equation}{section}
\theoremstyle{plain}

\newtheorem{thm}{Theorem}[section]

\newtheorem{prop}[thm]{Proposition}

\newtheorem{defi}[thm]{Definition}

\newtheorem{lem}[thm]{Lemma}
\newtheorem{cor}[thm]{Corollary}

\newtheorem{rema}[thm]{Remark}

\newcommand{\rr}{n}
\newcommand{\NN}{N}

\title{$N$-point spherical functions and asymptotic boundary KZB equations}
\author{J.V. Stokman \& N. Reshetikhin}
\address{J.S.: KdV Institute for Mathematics, University of Amsterdam,
Science Park 105-107, 1098 XG Amsterdam, The Netherlands.}
\email{j.v.stokman@uva.nl}
\address{N.R.: Department of Mathematics, University of California, Berkeley,
CA 94720, USA \& ITMO University, Kronverskii Ave. 49, Saint Petersburg, 197101, Russia \& KdV Institute for Mathematics, University of Amsterdam, Science Park 105-107, 1098 XG Amsterdam, The Netherlands}
\email{reshetik@math.berkeley.edu}
\begin{document}
\begin{abstract}
Let $G$ be a split real connected Lie group with finite center.
In the first part of the paper we define and study formal elementary spherical functions. They are formal power series analogues of elementary spherical functions on $G$ in which the role of the quasi-simple admissible $G$-representations is 
replaced by Verma modules.
For generic highest weight we express the formal elementary spherical functions
in terms of Harish-Chandra series and integrate them to spherical functions on the regular part of $G$.  We show that they produce eigenstates for spin versions of quantum hyperbolic Calogero-Moser systems.

In the second part of the paper we define and study special subclasses of global and formal elementary spherical functions, which we call global and formal $N$-point spherical functions. Formal $N$-point spherical functions arise as limits of correlation functions for boundary Wess-Zumino-Witten conformal field theory on the cylinder when the position variables tend to infinity. We construct global $N$-point spherical functions in terms of compositions of equivariant differential intertwiners associated with principal series representations,
and express them in terms of Eisenstein integrals.
We show that the eigenstates of the quantum spin Calogero-Moser system associated to $N$-point spherical functions are also common eigenfunctions of a commuting family of first-order differential operators, which we call asymptotic boundary Knizhnik-Zamolodchikov-Bernard operators. These operators are explicitly given in terms of folded classical dynamical $r$-matrices and associated dynamical $k$-matrices. 
\end{abstract}
\maketitle
\pagestyle{myheadings}
\section{Introduction}
Results of this paper lie at the interface of representation theory and quantum integrable systems.
The motivation comes from the theory of spherical functions in harmonic analysis on real reductive groups, from the theory of quantum integrable systems of Calogero-Moser type and from conformal field theory with conformal boundary conditions. 

We show that vector-valued elementary spherical functions provide joint eigenfunctions of the 
commuting quantum Hamiltonians of quantum spin\footnote{Throughout this paper we use "spin" in the sense how this term is used in physics as the description of internal degrees of freedom of one-dimensional quantum particles.}
Calogero-Moser type systems. 
We introduce a special class of vector-valued elementary spherical functions, which we call $N$-point spherical functions.
We show that the associated joint eigenfunctions of the quantum Hamiltonians are also joint eigenfunctions of 
a commuting family of
first order differential Knizhnik-Zamolodchikov-Bernard (KZB) type operators,
which originate in conformal field theory with conformal boundary conditions.

We also develop the theory of formal elementary spherical functions and formal $N$-point spherical functions. We show that formal spherical functions provide a representation theoretic interpretation of the Harish-Chandra series, and we use formal $N$-point spherical functions to establish the consistency of the differential KZB type equations. 

In the next four sections of the introduction we describe the main results in more detail.

\subsection{$\NN$-point spherical functions}
Let $G$ be a split real connected semisimple Lie group with finite center, and $K$ be a maximal compact subgroup of $G$. 
For two finite dimensional complex $K$-representations $(\sigma_\ell,V_\ell)$ and $(\sigma_r,V_r)$, write $\sigma:=\sigma_\ell\otimes\sigma_r^*$ for the resulting $K\times K$-representation on $V_\ell\otimes V_r^*\simeq\textup{Hom}(V_r,V_\ell)$.

The space $C_\sigma^\infty(G)$ of {\it $\sigma$-spherical functions on $G$} consists of the smooth functions $f: G\rightarrow V_\ell\otimes V_r^*$ satisfying 
\begin{equation}\label{transfospherical}
f(k_\ell gk_r^{-1})=\sigma(k_\ell,k_r)f(g)\qquad (k_\ell,k_r\in K,\, g\in G). 
\end{equation}
We say that a $\sigma$-spherical function $f\in C_\sigma^\infty(G)$ is
{\it elementary} if it is 
of the form
\[
f_{\mathcal{H}}^{\phi_\ell,\phi_r}(g):=\phi_\ell\circ\pi(g)\circ\phi_r
\]
for some quasi-simple admissible $G$-representation $(\pi,\mathcal{H})$ and $K$-intertwiners
$\phi_\ell\in\textup{Hom}_K(\mathcal{H},V_\ell)$ and $\phi_r\in\textup{Hom}_K(V_r,\mathcal{H})$. 

For special choices of $\sigma$ the theory of $\sigma$-spherical functions leads to representation theoretic constructions of integrable quantum one-dimensional many body systems and their eigenstates (see, e.g., \cite{OP,Op,HS,EK,O}). The commuting Hamiltonians arise from the action of the $G$-biinvariant differential operators on $C_\sigma(G)$,
while elementary $\sigma$-spherical functions produce the eigenstates. 
We extend these results to an arbitrary $K\times K$-representation $\sigma$. The corresponding quantum integrable system is called the
{\it quantum $\sigma$-spin Calogero-Moser system}. We will describe this integrable system in more detail in 
Subsection \ref{S13} of the introduction.

In this paper we also study elementary spherical functions when
the $K\times K$-representation is the state space $V_\ell\otimes\mathbf{U}\otimes V_r^*$
of a quantum spin chain of length $N\in\mathbb{Z}_{\geq 0}$ with reflecting boundaries. The bulk part
\begin{equation}\label{space}
\mathbf{U}:=U_1\otimes\cdots\otimes U_\NN
\end{equation}
of the state space $V_\ell\otimes\mathbf{U}\otimes V_r^*$
is the tensor product of $N$ finite dimensional $G$-modules $(\tau_i,U_i)$, and $V_\ell\otimes\mathbf{U}\otimes V_r^*$ is regarded as $K\times K$-module
with the subgroup $K\times 1$ acting diagonally on the first $N+1$ tensor factors and $1\times K$ acting on the last tensor factor. We denote its representation map by
$\sigma^{(N)}$.
We define {\it $N$-point $\sigma^{(N)}$-spherical functions}, or simply $N$-point spherical functions, as the special subclass\footnote{
In our follow-up paper \cite{RS}, we consider the space $C_{\sigma_\ell,\underline{\tau},\sigma_r}^\infty(G^{\times (N+1)})$ of $V_\ell\otimes\mathbf{U}\otimes V_r^*$-valued functions
$\widetilde{f}$ on $G^{\times (N+1)}$ satisfying the transformation behaviour
\[
\widetilde{f}(k_\ell g_0h_1^{-1},h_1g_1h_2^{-1},\ldots,h_Ng_Nk_r^{-1})=(\sigma_\ell(k_\ell)\otimes\tau_1(h_1)\otimes\cdots\otimes\tau_N(h_N)\otimes\sigma_r^*(k_r))
\widetilde{f}(g_0,\ldots,g_N)
\]
for $(k_\ell,h_1,\ldots,h_N,k_r)\in K\times G^{\times N}\times K$. This space is  
preserved by the action of the commutative algebra of biinvariant differential operators on $G^{\times (N+1)}$, and $N$-point $\sigma^{(N)}$-spherical functions $f$ produce
simultaneous eigenfunctions $\widetilde{f}\in C^{\infty}_{\sigma,\underline{\tau},\sigma_r}(G^{\times (N+1)})$ of the biinvariant differential operators on $G^{\times (N+1)}$ via the formula
\[
\widetilde{f}(g_0,\ldots,g_N):=\bigl(\textup{id}_{V_\ell}\otimes\tau_1(g_0^{-1})\otimes
\tau_2(g_1^{-1}g_0^{-1})\otimes\cdots\otimes\tau_N(g_{N-1}^{-1}\cdots g_1^{-1}g_0^{-1})\otimes\textup{id}_{V_r^*}\bigr)
f(g_0\cdots g_N).
\]
In this paper we do not use to full extent the $G^{\times N}$-action on $\mathbf{U}$. This will be done in the followup paper \cite{RS}, where we will focus on superintegrability.}
of elementary $\sigma^{(\NN)}$-spherical functions of the form
\begin{equation}\label{f-sigma}
f_{\underline{\mathcal{H}}}^{\phi_\ell,\mathbf{D},\phi_r}(g) :=
(\phi_\ell\otimes id_{\mathbf{U}}){\mathbf D}\circ\pi_{\mathcal{H}_N}(g)\circ \phi_r,
\end{equation}
where 
 \begin{enumerate}
\item $\underline{\mathcal{H}}:=(\mathcal{H}_{0},\ldots,
\mathcal{H}_{\NN})$ is an $(\NN+1)$-tuple of quasi-simple admissible $G$-representations,
\item $\mathbf{D}: \mathcal{H}_{\NN}^\infty\rightarrow
\mathcal{H}_{0}^\infty\otimes\mathbf{U}$ is a $G$-intertwiner given as the composition
\[
\mathbf{D}=(D_1\otimes\textup{id}_{U_2\otimes\cdots\otimes U_\NN})\cdots
(D_{\NN-1}\otimes\textup{id}_{U_\NN})D_\NN,
\]
of $G$-intertwiners $D_i: \mathcal{H}_{i}^{\infty}
\rightarrow \mathcal{H}_{{i-1}}^{\infty}\otimes U_i$, where $\mathcal{H}_i^\infty\subseteq\mathcal{H}_i$ is the space of smooth vectors,
\item $\phi_\ell\in\textup{Hom}_K(\mathcal{H}_{0},V_\ell)$ and
$\phi_r\in\textup{Hom}_K(V_r,\mathcal{H}_\NN)$ are $K$-intertwiners.
\end{enumerate}
Note that indeed $f_{\underline{\mathcal{H}}}^{\phi_\ell,\mathbf{D},\phi_r}$ is an elementary $\sigma^{(N)}$-spherical function because
$(\phi_\ell\otimes id_{\mathbf{U}}){\mathbf D}$ extends by continuity to a $K$-intertwiner $\mathcal{H}_N\rightarrow V_\ell\otimes\mathbf{U}$ and 
\[
f_{\underline{\mathcal{H}}}^{\phi_\ell,\mathbf{D},\phi_r}=f_{\mathcal{H}_{N}}^{(\phi_\ell\otimes id_{\mathbf{U}}){\mathbf D},\phi_r}.
\]
Moreover, elementary $\sigma$-spherical functions may be viewed as the $0$-point spherical functions.

Let $G=KAN_+$ be an Iwasawa decomposition of $G$, and denote by $\mathfrak{h}$ the complexified Lie algebra of $A$. Because $G$ is split, $\mathfrak{h}$ is a Cartan subalgebra of the complexified Lie algebra $\mathfrak{g}$ of $G$.
A linear functional $\lambda\in\mathfrak{h}^*$ defines
a multiplicative character $\eta_\lambda$ of $AN_+$ which acts trivially on $N_+$. For $\lambda\in\mathfrak{h}^*$ let $(\pi_\lambda,\mathcal{H}_\lambda)$ be the quasi-simple admissible $G$-representation obtained by 
normalized induction from $\eta_\lambda$.
The representation $(\pi_\lambda, \mathcal{H}_\lambda)$ is a finite direct sum of principal series representations. 
In Subsection \ref{S63} we provide a nontrivial family of $\NN$-point spherical functions $f_{\mathcal{H}_{\underline{\lambda}}}^{\phi_\ell,\mathbf{D},\phi_r}$ with the $(N+1)$-tuple of quasi-simple admissible representations given by
\[
\mathcal{H}_{\underline{\lambda}}=(\mathcal{H}_{\lambda_0},\ldots,\mathcal{H}_{\lambda_N})
\]
where $\lambda_i\in\mathfrak{h}^*$ are such that $\lambda_i-\lambda_{i-1}$ are weight of $U_i$, and  
$D_i: \mathcal{H}_{\lambda_i}^\infty\rightarrow\mathcal{H}_{\lambda_{i-1}}^\infty\otimes U_i$ are $G$-intertwiners constructed as $G$-equivariant differential operators.
They admit an integral representation
\begin{equation}\label{link1}
f_{\mathcal{H}_{\underline{\lambda}}}^{\phi_\ell,\mathbf{D},\phi_r}(g)=E_{\lambda_\NN}^{\sigma^{(\NN)}}(g)
T_{\lambda_\NN}^{(\phi_\ell\otimes\textup{id}_{\mathbf{U}})\mathbf{D},\phi_r}
\end{equation}
where $T_\lambda^{(\phi_\ell\otimes\textup{id}_{\mathbf{U}})\mathbf{D},\phi_r}\in V_\ell\otimes V_r^*\simeq\textup{Hom}(V_r,V_\ell)$ is an explicit rank one operator depending on the two $K$-intertwiners 
$(\phi_\ell\otimes\textup{id}_{\mathbf{U}})\mathbf{D}$ and $\phi_r$, and $E_\lambda^{\sigma}(g)$ ($\lambda\in\mathfrak{h}^*$) is the Eisenstein integral \eqref{Eisenstein}.\footnote{
For more background on Eisenstein integrals and their role in harmonic analysis see, e.g.,  \cite{HCe,HCI, HCII,W}.}

One can naturally speculate that affine analogues of $N$-point spherical functions 
should give $N$-point correlation functions for boundary Wess-Zumino-Witten-Novikov (WZWN) conformal field theory on an elliptic curve with conformally invariant boundary conditions. From this perspective the $G$-intertwiners $D_i$ are asymptotic remnants of affine vertex operators, and the $K$-intertwiners $\phi_\ell$ and $\phi_r$ are limits of boundary vertex operators. 

This perspective predicts that the restrictions of $N$-point spherical functions to $A\subset G$ provide joint eigenfunctions
of a commuting family of $N$ first-order differential operators, obtained as ``topological limit'' of trigonometric KZB operators.
The pertinent trigonometric KZB operators are first order differential operators in variables describing points on an infinite cylinder with reflecting conformal boundary conditions
and in dynamical variables, which can be identified with the subgroup $A\subset G$ (see \cite[\S2.3]{St2}). In the topological limit the 
dependence on the points disappears. 

In Section \ref{SectionbKZB} we directly construct $N$ first-order differential operators 
on $A$, called {\it asymptotic boundary KZB operators}, and we show that the restrictions of the $N$-point spherical functions $f_{\mathcal{H}_{\underline{\lambda}}}^{\phi_\ell,\mathbf{D},\phi_r}$ to $A\subset G$ provide joint eigenfunctions of the  Hamiltonians of the quantum $\sigma^{(N)}$-spin Calogero-Moser system as well as of the asymptotic boundary KZB operators (see Theorem \ref{mainthmbKZBEisenstein})\footnote{From the perspective of footnote 2, the eigenvalue equations with respect to the asymptotic boundary KZB operators arise from the action of the biinvariant differential operator $\Omega_i-\Omega_{i-1}$ on $\widetilde{f}$, where $\Omega$ is the quadratic Casimir of $G$ and $\Omega_i$ is its interpretation as biinvariant differential operator acting on the $i^{\textup{th}}$-coordinate of $G^{\times (N+1)}$.}.
We describe the asymptotic boundary KZB operators in more detail in Subsection \ref{S14} of the introduction.

\subsection{Formal $\NN$-point spherical functions}
In this paper we also develop the theory of formal elementary $\sigma$-spherical functions and formal $\NN$-point $\sigma^{(N)}$-spherical functions. A formal elementary $\sigma$-spherical function is a formal power series analogue of the restriction of the elementary $\sigma$-spherical function
$f_{\mathcal{H}_\lambda}^{\phi_\ell,\phi_r}$ to the  positive Weyl chamber $A_+$ in $A$, constructed as follows.

The complexified Lie algebra 
$\mathfrak{b}$ of $AN_+$ is a Borel subalgebra containing the Cartan subalgebra $\mathfrak{h}$. Let $R$ be the associated root system of 
$\mathfrak{g}$, $R^+$ be the set of positive roots, and $M_\lambda$ the Verma module of highest weight $\lambda\in\mathfrak{h}^*$. We denote by $M_\lambda[\mu]$ the weight space of $M_\lambda$ of weight $\mu\in\mathfrak{h}^*$.

Let $\mathfrak{n}_-$ be the nilpotent subalgebra of $\mathfrak{g}$ opposite to $\mathfrak{b}$, and $\overline{M}_\lambda$ be the $\mathfrak{n}_-$-completion of $M_\lambda$.
Fix $\mathfrak{k}$-intertwiners $\phi_\ell\in\textup{Hom}_{\mathfrak{k}}(M_\lambda,V_\ell)$ and
$\phi_r\in\textup{Hom}_{\mathfrak{k}}(V_r,\overline{M}_\lambda)$, where $\mathfrak{k}$ is the complexified Lie algebra of $K$. 
We denote by 
$\phi_\ell^\mu\in\textup{Hom}_{\mathbb{C}}(M_\lambda[\mu],V_\ell)$ and 
$\phi_r^\mu\in\textup{Hom}_{\mathbb{C}}(V_r,M_\lambda[\mu])$ the weight components of  $\phi_\ell$ and $\phi_r$ of weight $\mu$. 

The {\it formal elementary $\sigma$-spherical function associated with $M_\lambda$, $\phi_\ell$ and $\phi_r$} is the formal series 
\begin{equation}\label{fssf}
F_{M_\lambda}^{\phi_\ell,\phi_r}:=\sum_{\mu\leq\lambda}\bigl(\phi_\ell^\mu\circ\phi_r^\mu)\xi_\mu
\end{equation}
where $\leq$ is the dominance order on $\mathfrak{h}^*$ and $\xi_\mu$ is 
the multiplicative character $\xi_\mu(a):=e^{\mu(\log(a))}$ on $A$.

Let $U(\mathfrak{g})$ and $U(\mathfrak{k})$ be the universal enveloping algebra of $\mathfrak{g}$ and $\mathfrak{k}$ respectively, and denote by $Z(\mathfrak{g})$ the center of $U(\mathfrak{g})$.
Harish-Chandra's radial component 
$\widehat{\Pi}(z)$ of $z\in U(\mathfrak{g})$ is the $U(\mathfrak{k})^{\otimes 2}$-valued differential operator on the regular part $A_{\textup{reg}}$ of $A$ such that 
\[
\widehat{\Pi}(z)(f\vert_{A_{\textup{reg}}})=(r^*(z)f)\vert_{A_{\textup{reg}}}
\]
for all spherical functions $f$, where $r^*(z)$ denotes the left $G$-invariant differential operator on $G$ associated to $z$. 
The radial components $\widehat{\Pi}(z)$ of the $G$-biinvariant differential operators $r^*(z)$ ($z\in Z(\mathfrak{g})$) pairwise commute.

We show in Theorem \ref{mainTHMF}{\bf a} that 
$F_{M_\lambda}^{\phi_\ell,\phi_r}$, as formal power series, is a simultaneous eigenfunction of the differential operators $\widehat{\Pi}(z)$ ($z\in Z(\mathfrak{g})$) with
eigenvalues given by the central character $\zeta_\lambda$ of $M_\lambda$. As a consequence, we are able to relate 
the formal elementary $\sigma$-spherical function $F_{M_\lambda}^{\phi_\ell,\phi_r}$ to the $\sigma$-Harish-Chandra series, when $\lambda$ is in an appropriate subset 
of generic highest weights\footnote{See \eqref{HCgeneric} for details.}.

The $\sigma$-Harish-Chandra series is defined as follows. Let $\Omega\in Z(\mathfrak{g})$ be the quadratic Casimir element.
For generic $\lambda\in\mathfrak{h}^*$ the {\it $\sigma$-Harish-Chandra series}  $\Phi^\sigma_\lambda$ 
is the unique $\textup{End}(V_\ell\otimes V_r^*)$-valued formal eigenfunction of $\widehat{\Pi}(\Omega)$ with eigenvalue $\zeta_\lambda(\Omega)$
of the form\footnote{Note here the remarkable fact, well known to specialists in harmonic analysis, that for generic $z\in Z(\mathfrak{g})$ and $\lambda\in\mathfrak{h}^*$ the requirement that the formal $\textup{End}(V_\ell\otimes V_r^*)$-valued power series $f=\sum_{\mu\leq\lambda}f_{\lambda-\mu}\xi_\mu$ is an eigenfunction of the radial component of $z$ with eigenvalue $\zeta_\lambda(z)$
will uniquely define the coefficients $f_\gamma\in\textup{End}(V_\ell\otimes V_r^*)$ in terms 
of $f_0\in\textup{End}(V_\ell\otimes V_r^*)$.
This in particular holds true for $z=\Omega$.
The quadratic Casimir $\Omega$ is a natural choice since its radial component
is an explicit second-order differential operator that produces the Hamiltonian of the $\sigma$-spin quantum Calogero-Moser system, solvable by asymptotic Bethe ansatz, see Subsection \ref{S13}.}
\[
\Phi^\sigma_\lambda=\sum_{\mu\leq\lambda}\Gamma_{\lambda-\mu}^\sigma(\lambda)\xi_\mu \qquad\quad (\Gamma_0^\sigma(\lambda)=\textup{id}_{V_\ell\otimes V_r^*}).
\]
The $\sigma$-Harish-Chandra series converges on $A_+$, and thus defines an analytic $\textup{End}(V_\ell\otimes V_r^*)$-valued analytic function on $A_+$. The $\sigma$-Harish-Chandra function plays an important role in the asymptotic analysis of $\sigma$-spherical functions through the explicit expansion of the Eisenstein integral in Harish-Chandra series, see, e.g., \cite{HC,HCI,HCII,W}. Another interesting recent application of $\sigma$-Harish-Chandra series is its appearance in the description of four-point spin conformal blocks in Euclidean conformal field theories within the conformal bootstrap program (see \cite{SSI,IS,ILLS} and references therein).

We show in Theorem \ref{mainTHMF}{\bf c} that for generic $\lambda\in\mathfrak{h}^*$,
\begin{equation}\label{link2}
F_{M_\lambda}^{\phi_\ell,\phi_r}=\Phi_\lambda^\sigma(\cdot)(\phi_\ell^\lambda\circ\phi_r^\lambda).
\end{equation}
In this case the formal $\sigma$-spherical function $F_{M_\lambda}^{\phi_\ell,\phi_r}$ is a $V_\ell\otimes V_r^*$-valued analytic function on $A_+$ 
which extends to a smooth $V_\ell\otimes V_r^*$-valued function on the dense open subset $G_{\textup{reg}}:=KA_+K$ of regular elements in $G$ satisfying the equivariance poperty \eqref{transfospherical}, where $M$ is the centraliser of $A$ in $K$.
Conversely, \eqref{link2} provides a representation theoretic interpretation of the expansion coefficients $\Gamma_{\lambda-\mu}^\sigma(\lambda)$ of the Harish-Chandra series in terms of matrix coefficients of Verma modules.

For $\mathfrak{g}=\mathfrak{sl}_2(\mathbb{C})$ the weight components of $\mathfrak{k}$-intertwiners
$\phi_\ell$ and $\phi_r$ are Meixner-Pollaczek polynomials. On the other hand, the $\sigma$-Harish-Chandra series can be expressed in terms of Gauss' hypergeometric series ${}_2F_1$.
Formula \eqref{link2} then provides a representation theoretic proof of the formula \cite{Er,Ra,KJ} expressing the Poisson kernel of Meixner-Pollaczek polynomials as a ${}_2F_1$. This is detailed in Subsection \ref{rankoneSection}.

We define {\it formal $N$-point $\sigma^{(N)}$-spherical functions} to be the special subclass of 
formal elementary $\sigma^{(N)}$-spherical functions of the form
\[
F_{M_{\underline{\lambda}}}^{\phi_\ell,\mathbf{\Psi},\phi_r}:=F_{M_{\lambda_N}}^{(\phi_\ell\otimes\textup{id}_{\mathbb{U}})\mathbf{\Psi},\phi_r}
\]
where 
\begin{enumerate}
\item $M_{\underline{\lambda}}=(M_{\lambda_0},\ldots,M_{\lambda_N})$ is an $(N+1)$-tupe of Verma modules with highest weights $\lambda_i$ such that $\lambda_i-\lambda_{i-1}$ is a weight of $U_i$ for each $i=1,\dots, N$,
\item $\mathbf{\Psi}: M_{\lambda_N}\rightarrow
M_{\lambda_0}\otimes\mathbf{U}$ is a $\mathfrak{g}$-intertwiner given as the composition
 \[
 \mathbf{\Psi}=(\Psi_1\otimes\textup{id}_{U_2\otimes\cdots\otimes U_N})\cdots
(\Psi_{N-1}\otimes\textup{id}_{U_N})\Psi_N
\]
of $\mathfrak{g}$-intertwiners $\Psi_i: M_{\lambda_i}
\rightarrow M_{\lambda_{i-1}}\otimes U_i$,
\item $\phi_\ell\in\textup{Hom}_{\mathfrak{k}}(M_{\lambda_0},V_\ell)$ and
$\phi_r\in\textup{Hom}_{\mathfrak{k}}(V_r,\overline{M}_{\lambda_N})$  are $\mathfrak{k}$-intertwiners.
\end{enumerate}
For generic $\lambda_N\in\mathfrak{h}^*$ formula \eqref{link2} provides an explicit expression of the formal $N$-point spherical function $F_{M_{\underline{\lambda}}}^{\phi_\ell,\mathbf{\Psi},\phi_r}$ in terms of the $\sigma^{(N)}$-Harish-Chandra series $\Phi_{\lambda_N}^{\sigma^{(N)}}$ and the highest weight components of the $\mathfrak{k}$-intertwiners
$(\phi_\ell\otimes\textup{id}_{\mathbf{U}})\mathbf{\Psi}$ and $\phi_r$. It is the analogue of  
formula \eqref{link1} expressing the $N$-point spherical function $f_{\mathcal{H}_{\underline{\lambda}}}^{\phi_\ell,\mathbf{D},\phi_r}$ as Eisenstein integral. 

We show that
formal elementary $N$-point spherical functions $F_{M_{\underline{\lambda}}}^{\phi_\ell,\mathbf{\Psi},\phi_r}$ give rise to joint eigenfunctions of the quantum Hamiltonians of the $\sigma^{(N)}$-spin Calogero-Moser system and, in addition, are joint eigenfunctions of asymptotic boundary KZB operators (Theorem \ref{mainthmbKZB} and Corollary \ref{corMAINkzb}). Using a boundary version of the fusion operator for $\mathfrak{g}$-intertwiners from \cite{E,EV}, we obtain
a topologically complete set of joint formal eigenfunctions consisting of formal $N$-point spherical functions. This result implies
that the boundary asymptotic KZB operators commute (Theorem \ref{consistentoperators}).
It also suggests that the quantum $\sigma^{(N)}$-spin Calogero-Moser system is super-integrable, which we prove in our follow-up paper \cite{RS}.

\subsection{The quantum spin Calogero-Moser systems}\label{S13}

The commuting quantum spin Calogero-Moser Hamiltonians corresponding to the spherical functions on $G$ 
are the $U(\mathfrak{k})^{\otimes 2}$-valued differential operators 
\[
H_z:=\delta\circ\widehat{\Pi}(z)\circ\delta^{-1}
\qquad (z\in Z(\mathfrak{g})),
\]
on $A_{\textup{reg}}$, where $\delta:=\xi_\rho\prod_{\alpha\in R^+}(1-\xi_{-2\alpha})^{\frac{1}{2}}$ and $\rho$ is the half sum of the positive roots. The $\textup{End}(V_\ell\otimes V_r^*$)-valued differential operators $H_z^\sigma:=\sigma(H_z)$ ($z\in Z(\mathfrak{g})$) are the Hamiltonians for the $\sigma$-spin Calogero-Moser system. 

The quadratic Hamiltonian $H_\Omega$ admits the following explicit expression.
Denote by $\langle\cdot,\cdot\rangle_{\mathfrak{g}_0}$ the Killing form  of the Lie algebra $\mathfrak{g}_0$ of $G$. It restricts to a scalar product on the Lie algebra $\mathfrak{h}_0$ of $A$, giving $A=\exp(\mathfrak{h}_0)$ the structure of a Riemannian manifold. 
Denote by $\mathfrak{g}_{0,\alpha}$ the root subspace in $\mathfrak{g}_0$ associated to $\alpha\in R$, and by $\theta\in\textup{Aut}(\mathfrak{g})$ the complex linear extension of the Cartan involution of $\mathfrak{g}_0$ relative to the Iwasawa decomposition $G=KAN_+$. 
Choose $e_\alpha\in\mathfrak{g}_{0,\alpha}$ ($\alpha\in R$) 
such that $\theta(e_\alpha)=-e_{-\alpha}$ and $[e_\alpha,e_{-\alpha}]=t_\alpha$, where 
$t_\alpha\in\mathfrak{h}_0$ is the unique element such that $\langle t_\alpha,h\rangle_{\mathfrak{g}_0}=\alpha(h)$ for all $h\in\mathfrak{h}_0$. Then
\[
\mathfrak{k}=\bigoplus_{\alpha\in R^+}\mathbb{C}y_\alpha
\] 
with $y_\alpha:=e_\alpha-e_{-\alpha}$ ($\alpha\in R$). 

The quadratic Hamiltonian $\mathbf{H}:=-\frac{1}{2}(H_\Omega+\|\rho\|^2)$ of the spin Calogero-Moser model is 
given by
\[
\mathbf{H}=-\frac{1}{2}\Delta+V
\]
with $\Delta$ the Laplace-Beltrami operator on $A$, and $V$
the $U(\mathfrak{k})^{\otimes 2}$-valued potential 
\[
V=-\frac{1}{2}\sum_{\alpha\in R}\frac{1}{(\xi_\alpha-\xi_{-\alpha})^2}
\Bigl(\frac{\|\alpha\|^2}{2}+\prod_{\epsilon\in\{\pm 1\}}(y_\alpha\otimes 1+\xi_{\epsilon\alpha}
(1\otimes y_\alpha)\Bigr),
\]
see Proposition \ref{qH}. The extension of this result to arbitrary real semisimple Lie group $G$ is given in \cite{RS}.

Special cases 
of the representation theoretic construction of quantum $\sigma$-spin Calogero-Moser systems and their eigenstates are known. For example, the case 
when $\sigma_\ell$ and $\sigma_r$ are the trivial representation was studied in \cite{OP,Op}, and the case when
$\mathfrak{g}=\mathfrak{sp}_{r}(\mathbb{C})$ and
$\sigma_\ell=\sigma_r$ is one-dimensional was analysed in \cite[Chpt. 5]{HS}. Other natural special cases will be discussed in
Subsection \ref{vvCM}. 

The theory developed in this paper can also be applied to compact symmetric spaces. In this case it
yields a trigonometric version of quantum spin Calogero-Moser systems, with eigenstates described by vector-valued multivariable orthogonal polynomials. For certain compact symmetric spaces and special choices of $\sigma$, this relates to the theory of Etingof, Kirillov Jr. and Schiffmann \cite{EK,ES} on generalised weighted trace functions and Oblomkov's \cite{O} version for Grassmannians. In these two cases 
the eigenfunctions can be expressed in terms of {\it scalar-valued} Jack polynomials and $BC$-type Heckman-Opdam polynomials, respectively.

The classical integrable systems underlying the quantum $\sigma$-spin trigonometric Calogero-Moser systems were considered in \cite{FP, FP1, FP2, Re}. 

\subsection{Asymptotic boundary KZB operators}\label{S14}
For $N$-point $\sigma^{(N)}$-spherical functions on $G$ the related quantum $\sigma^{(N)}$-spin Calogero-Moser system turns out to be a super-integrable quantum Calogero-Moser spin chain with associated spin space $V_\ell\otimes\mathbf{U}\otimes V_r^*$. In its universal form the quantum Hamiltonians 
are obtained by a coordinate
radial component map from $Z(\mathfrak{g})^{\otimes (N+1)}$ to $U(\mathfrak{k})\otimes  U(\mathfrak{g})^{\otimes N}\otimes U(\mathfrak{k})$-valued differential operators on $A_{\textup{reg}}$, cf. footnote 2. The quantum Hamiltonians described in the previous subsection arise as the coordinate radial components of $1^{\otimes N}\otimes z$ ($z\in Z(\mathfrak{g})$) and are given by 
\[
H_z^{(N)}:=\delta\circ (\Delta^{(N)}\otimes\textup{id})(\widehat{\Pi}(z))\circ\delta^{-1}\qquad\quad (z\in Z(\mathfrak{g}))
\]
with $\Delta^{(N)}: U(\mathfrak{k})\rightarrow U(\mathfrak{k})^{\otimes (N+1)}$ the $N$-fold iterated
comultiplication of the universal enveloping algebra $U(\mathfrak{k})$. In particular, these quantum Hamiltonians are $U(\mathfrak{k})^{\otimes (N+2)}$-valued.
The asymptotic KZB operators are part of the algebra of quantum Hamiltonians of the quantum Calogero-Moser spin chain, see footnote 4.
The super-integrable perspective is discussed in detail in our follow-up paper \cite{RS}. In this paper we obtain the asymptotic KZB operators by deriving the asymptotic KZB equations for $N$-point spherical functions using quantum field theoretic methods.

Let $\{x_s\}_{s=1}^\rr$ be an orthonormal basis of $\mathfrak{h}_0$ and $\partial_{x_s}$ the associated first order differential operator on $A$.
Write $E$ for the $U(\mathfrak{g})$-valued first order differential operator 
\[
E:=\sum_{s=1}^\rr\partial_{x_s}\otimes x_s
\]
on $A$. 
Consider the $\mathfrak{g}\otimes\mathfrak{g}$-valued functions
\begin{equation}\label{foldedexpressions}
\begin{split}
r^+&=\sum_{\alpha\in R}\frac{y_\alpha\otimes e_\alpha}{1-\xi_{-2\alpha}},\\
r^-&=\sum_{s=1}^\rr x_s\otimes x_s+\sum_{\alpha\in R}\frac{(e_\alpha+e_{-\alpha})\otimes e_\alpha}
{1-\xi_{-2\alpha}}
\end{split}
\end{equation}
and the $U(\mathfrak{k})\otimes U(\mathfrak{g})\otimes U(\mathfrak{k})$-valued function 
\begin{equation*}
\kappa:=\sum_{\alpha\in R}\frac{y_\alpha\otimes e_\alpha\otimes 1}{1-\xi_{-2\alpha}}
+1\otimes \kappa^{\textup{core}}\otimes 1+\sum_{\alpha\in R}\frac{1\otimes e_\alpha\otimes y_\alpha}{\xi_\alpha-\xi_{-\alpha}},
\end{equation*}
with its core $\kappa^{\textup{core}}$ the $U(\mathfrak{g})$-valued function 
\[
\kappa^{\textup{core}}:=\frac{1}{2}\sum_{s=1}^\rr x_s^2+\sum_{\alpha\in R}\frac{e_\alpha^2}{1-\xi_{-2\alpha}}.
\]
The {\it asymptotic boundary KZB operators} 
are the first-order $U(\mathfrak{k})\otimes U(\mathfrak{g})^{\otimes N}\otimes U(\mathfrak{k})$-valued differential operators
\begin{equation}\label{introbKZBoper}
\mathcal{D}_i:=E_i-\sum_{j=1}^{i-1}r_{ji}^+-\kappa_i-\sum_{j=i+1}^Nr_{ij}^-
\end{equation}
on $A_{\textup{reg}}$ for $i=1,\ldots,N$.
Here the indices $i,j$ on the right hand side of \eqref{introbKZBoper}
indicate in which tensor components of $U(\mathfrak{g})^{\otimes N}$ the $U(\mathfrak{g})$-components of $E$, $r^{\pm}$ and $\kappa$ are placed. Note that the only nontrivial contributions to the left and right $U(\mathfrak{k})$-tensor components of $U(\mathfrak{k})\otimes U(\mathfrak{g})^{\otimes N}\otimes U(\mathfrak{k})$ arise from $\kappa_i-\kappa_i^{\textup{core}}$.

The local terms of the asymptotic KZB operators are folded and contracted versions of Felder's  \cite{F}, \cite[\S 2]{ES} classical trigonometric dynamical $r$-matrix
\[
r:=-\frac{1}{2}\sum_{s=1}^\rr x_s\otimes x_s-\sum_{\alpha\in R}\frac{e_{-\alpha}\otimes e_\alpha}{1-\xi_{-2\alpha}}
\]
since
\begin{equation}\label{rpm}
r^{\pm}=\pm r+(1\otimes \theta)r_{21},\qquad \kappa^{\textup{core}}=m((1\otimes \theta)r_{21}),
\end{equation}
with $m$ the multiplication map of $U(\mathfrak{g})$.
More generally, for $a\in A_{\textup{reg}}$, 
\[
\kappa(a)=r^+(a)\otimes 1+1\otimes\kappa^{\textup{core}}(a)\otimes 1+
1\otimes ((\textup{Ad}_{a^{-1}}\otimes\textup{id})r_{21}^+(a)).
\]
 An algebraic analysis of folding and contraction of classical dynamical $r$-matrices is in the follow-up paper \cite{St2}.

By the results as explained in Subsection \ref{S13}, the $V_\ell\otimes\mathbf{U}\otimes V_r^*$-valued analytic functions 
\[
\mathbf{f}_{\underline{\lambda}}^{\phi_\ell,\mathbf{D},\phi_r}(a):=\delta(a)f_{\mathcal{H}_{\underline{\lambda}}}^{\phi_\ell,\mathbf{D},\phi_r}(a)\qquad\quad (a\in A_+)
\]
satisfy 
\[
H^{(N)}_z\bigl(\mathbf{f}_{\underline{\lambda}}^{\phi_\ell,\mathbf{D},\phi_r}\bigr)=\zeta_{\lambda_N-\rho}(z)\mathbf{f}_{\underline{\lambda}}^{\phi_\ell,\mathbf{D},\phi_r}
\qquad\quad (z\in Z(\mathfrak{g})),
\]
see Theorem \ref{mainthmbKZBEisenstein}. In this paper we will present two different proofs for the fact that they also satisfy the
first order differential equations
\begin{equation}\label{bKZBeqintro}
\mathcal{D}_i\bigl(\mathbf{f}_{\underline{\lambda}}^{\phi_\ell,\mathbf{D},\phi_r}\bigr)=\Bigl(\frac{(\lambda_i,\lambda_i)}{2}-
\frac{(\lambda_{i-1},\lambda_{i-1})}{2}\Bigr)\mathbf{f}_{\underline{\lambda}}^{\phi_\ell,\mathbf{D},
\phi_r},\qquad i=1,\ldots,N.
\end{equation}

The starting point of the proofs is rewriting the right hand side of \eqref{bKZBeqintro} in terms of the action of the Casimir $\Omega$ on 
$\mathcal{H}_{\lambda_i}^{\infty}$ and $\mathcal{H}_{\lambda_{i-1}}^\infty$ on both sides of the $i$th intertwiner $D_i$ in 
$\mathbf{f}_{\underline{\lambda}}^{\phi_\ell,\mathbf{D},\phi_r}$.

For the first proof we use an explicit Cartan-type factorisation of the Casimir $\Omega$ in $U(\mathfrak{g})$, see \eqref{KAKomega}. This factorisation is the algebraic reflection of the explicit formula for the differential operator $\widehat{\Pi}(\Omega)$ on $A_{\textup{reg}}$.
Pushing the factors from this factorisation through the intertwiners 
$D_j$ in $\mathbf{f}_{\underline{\lambda}}^{\phi_\ell,\mathbf{D},\phi_r}$ to the far left and right 
is creating the $r^{\pm}$ contributions to the asymptotic boundary KZB equations. 
The remaining factors are then absorbed by the $K$-intertwiners $\phi_\ell$ and $\phi_r$, producing the contribution $\kappa_i-\kappa_i^{\textup{core}}$ to $\mathcal{D}_i$. In this proof the core $\kappa^{\textup{core}}_i$ of $\kappa_i$ is already part of the initial factorisation of the Casimir element, and stays put at its initial spot throughout this procedure. 
In this proof
the terms $r^+_{ji}$ ($j<i$) and $r^-_{ij}$ ($j>i$) appear as the expressions \eqref{foldedexpressions}, not as folded and contracted versions of Felder's classical dynamical $r$-matrix.

In the second proof we substitute the factorisation\footnote{This factorisation can be used to 
derive the asymptotic KZB equations for Etingof's and Schiffmann's \cite{ES} generalised weighted trace functions in a manner similar to the one as described above for $N$-point spherical functions, see \cite{St2} (weighted traces are naturally
associated to the symmetric space $G\times G/\textup{diag}(G)$, with $\textup{diag}(G)$
the group $G$ diagonally embedded into $G\times G$).}
\begin{equation}\label{OmegaAAA}
\Omega=\sum_{k=1}^\rr x_k^2+\frac{1}{2}\sum_{\alpha\in R}\left(\frac{1+a^{-2\alpha}}{1-a^{-2\alpha}}\right)t_\alpha+2\sum_{\alpha\in R}\frac{e_{-\alpha}e_\alpha}{1-a^{-2\alpha}}
\end{equation}
of the quadratic Casimir element $\Omega$ for regular $a\in A_{\textup{reg}}$,
push the left and right root vectors through the intertwiners $D_j$ 
to the far left and right, reflect against the $K$-intertwiners $\phi_\ell$ and $\phi_r$,
and push the reflected factors back to their original position, where they merge and create the core $\kappa^{\textup{core}}_i$ of $\kappa_i$.  When we initially move components of $\Omega$ to the boundaries the 
terms $r_{ji}$ or $r_{ij}$ are created.
On the way back they are producing similar terms, but now involving the $\theta$-twisted $r$-matrix
$(1\otimes \theta)r_{21}$. This proof
naturally leads to the folded and contracted expressions \eqref{rpm} for $r^{\pm}$ and $\kappa^{\textup{core}}$ in terms of Felder's $r$-matrix. 

A separate proof is needed to show that formal $N$-point $\sigma^{(N)}$-spherical functions are joint eigenfunctions of the asymptotic KZB operators
(Theorem \ref{mainthmbKZB}). 
It leads to the proof of the commutativity of the asymptotic 
boundary KZB operators (Theorem \ref{consistentoperators}). This in turn implies that 
the $r^{\pm}$ satisfy three coupled classical dynamical Yang-Baxter equations,
and that $\kappa$ solves an associated classical dynamical reflection equation (Theorem \ref{thmDEF}). An algebraic proof of this fact is given in the follow-up paper
\cite{St2}.
\subsection{Outlook}
Harish-Chandra's theory of harmonic analysis on $G$ has been developed for arbitrary real connected semisimple Lie groups $G$ with finite center (more generally, for reductive $G$ in Harish-Chandra's class). We expect that the theory of global and formal $N$-point spherical functions extends to this more general setup as well. The role of the Cartan subalgebra $\mathfrak{h}_0$ will be taken over by a maximal abelian subalgebra $\mathfrak{a}_0$ of the $(-1)$-eigenspace of the Cartan involution $\theta_0$, and the role of the root system $R$ by the associated restricted root system in $\mathfrak{a}_0^*$. In our follow-up paper \cite{RS} we derive the asymptotic boundary KZB equations in this more general context. The compatibility condition for asymptotical boundary KZB equations in the non-split cases also give rise to consistency conditions on their building blocks, but these conditions no longer imply separate dynamical Yang-Baxter and reflection equations, see \cite[\S 6.2]{RS}.

Boundary KZB equations with spectral parameters will be discussed in a separate paper (for affine $\mathfrak{sl}_2$ Kolb has already derived the associated KZB-heat equation in \cite{K}). A short discussion of the boundary KZB equations and their degeneration to asymptotic boundary KZB equations and type $C$ (asymptotic) Gaudin Hamiltonians can be found in \cite[\S 2.3]{St2}.

It is natural to generalise the theory to quantum groups using the Letzter-Kolb \cite{Le,K2} theory of quantum (affine) symmetric pairs. We expect that the role of $\kappa$ with trivialised right boundary component 
will be taken over by a dynamical universal $K$-matrix $\mathcal{K}$, whose action on the parametrising spaces of quantum boundary vertex operators describes the action of the Balagovic-Kolb \cite{BK} universal $K$-matrix \cite{BK} on the spin spaces of the quantum boundary vertex operators. This should be compared with the way that dynamical $R$-matrices appear in Etingof's and Varchenko's \cite{EV2} theory of generalised trace functions  and quantum KZB equations. 
This direction has many promising connections to integrable models in statistical mechanics and quantum field theory with integrable boundary conditions, see, e.g., \cite{DM,GZ,JKKMW} and references therein.
\subsection{Contents of the paper}
In Section \ref{sectionII} we recall basic facts on irreducible split Riemannian pairs and establish the relevant notations. In Section \ref{S3} we recall, following \cite{CM,W}, Harish-Chandra's radial component map and the explicit expression of the radial component of the quadratic Casimir element. We furthermore establish the link to quantum spin hyperbolic Calogero-Moser systems (Subsection \ref{vvCM}) and highlight various important special cases.
We recall the construction of the Harish-Chandra series in Subsection \ref{S5}, and discuss how they give rise to eigenstates for the quantum spin hyperbolic Calogero-Moser systems.  In the first two subsections of Section \ref{SectionRepTh} we recall fundamental results of 
Harish-Chandra \cite{HCe,HCI, HCII} on the principal series representations of $G$ and its associated matrix coefficients. In Subsection \ref{ssss3} we discuss the algebraic principal series representations, and the description of the associated spaces of 
$\mathfrak{k}$-intertwiners. Section \ref{RTHC} first discusses how the algebraic principal series representations can be identified with $\mathfrak{k}$-finite parts of weight completions of Verma modules, which
leads to a detailed description of the $\mathfrak{k}$-intertwining spaces $\textup{Hom}_{\mathfrak{k}}(M_\lambda,V_\ell)$
and $\textup{Hom}_{\mathfrak{k}}(V_r,\overline{M}_\lambda)$.
In the second half of the section we introduce formal elementary $\sigma$-spherical functions and prove their key properties (differential equations and relation to $\sigma$-Harish-Chandra series). In Section \ref{SectionbKZB} we first derive asymptotic operator KZB equations for $\mathfrak{g}$-intertwiners and relate them to factorisations of the quadratic Casimir element $\Omega$. 
In Subsections \ref{S62} \& \ref{S63} we describe the spaces of $G$-equivariant
differential operators $\mathcal{H}_\lambda^\infty\rightarrow\mathcal{H}_\mu^\infty\otimes U$ for a finite dimensional $G$-representation $U$, and derive the asymptotic boundary KZB equations for the associated $N$-point spherical functions.
In Subsection \ref{intertwinersection}
we derive the asymptotic boundary KZB equations for the formal $N$-point spherical functions.
Subsection \ref{sectionBFO} and subsection \ref{S66} introduce the boundary fusion operator and establishes the integrability of the asymptotic boundary KZB operators. Finally, in Subsection \ref{inteq} we establish the resulting coupled classical dynamical Yang-Baxter equations and the associated dynamical reflection equations for the building blocks $r^{\pm}$ and $\kappa$ of the asymptotic boundary KZB operators.
\vspace{.5cm}\\
\noindent
{\bf Acknowledgments.} We thank Ivan Cherednik, Pavel Etingof, Giovanni Felder, Gert Heckman, Erik Koelink, Christian Korff, Tom Koornwinder, Eric Opdam, Maarten van Pruijssen, Taras Skrypnyk and Bart Vlaar for discussions and comments. We thank
Sam van den Brink for carefully reading the first part of the paper and pointing out a number of typos.
The work of J.S. and N.R. was supported by the Dutch Research Council (NWO). 
In addition the work of N.R. was partially supported by NSF DMS-1601947 and by 
RSF 21-11-00141. He also would like to thank ETH-ITS for the hospitality during the final stages of the work.\\
\vspace{.3cm}\\
{\bf Notations and conventions.} 
We write $\textup{ad}_L: L\rightarrow \mathfrak{gl}(L)$ for the adjoint representation of a Lie algebra $L$, and $\langle\cdot,\cdot\rangle_L$ for its Killing form.  Real Lie algebras will be denoted with a subscript zero; The complexification of a real Lie algebra
$\mathfrak{g}_0$ with be denoted by $\mathfrak{g}:=\mathfrak{g}_0\otimes_{\mathbb{R}}\mathbb{C}$. The tensor product $\otimes_F$ of $F$-vector spaces is denoted by $\otimes$ in case $F=\mathbb{C}$. For complex vector spaces $U$ and $V$ we write $\textup{Hom}(U,V)$ for the vector space of complex linear maps $U\rightarrow V$. Representations of Lie groups are complex, strongly continuous Hilbert space representations. If $U$ and $V$ are the representation spaces of
two representations of a Lie group $G$, then $\textup{Hom}_G(U,V)$ denotes the subspace of bounded linear $G$-intertwiners $U\rightarrow V$.  If $U$ and $V$ are two $\mathfrak{g}$-modules for a complex Lie algebra $\mathfrak{g}$, then $\textup{Hom}_{\mathfrak{g}}(U,V)$ denotes
the subspace of $\mathfrak{g}$-intertwiners $U\rightarrow V$. The representation map of the infinitesimal $\mathfrak{g}$-representation associated to a smooth $G$-representation $(\tau,U)$ will be denoted by $\tau$ again, if no confusion can arise.


\section{Split real semisimple Lie algebras}\label{sectionII}
This short section is to fix the basic notations for split real semisimple Lie algebras and Lie groups. For further reading consult, e.g., \cite{Kn}.
\subsection{Root space and Cartan decomposition}

Let $\mathfrak{g}_0$ be a split real semisimple Lie algebra with Cartan involution $\theta_0\in\textup{Aut}(\mathfrak{g}_0)$ and corresponding Cartan decomposition
\[
\mathfrak{g}_0=\mathfrak{k}_0\oplus\mathfrak{p}_0.
\]
The $+1$-eigenspace $\mathfrak{k}_0\subset\mathfrak{g}_0$ is a Lie subalgebra of $\mathfrak{g}_0$, and 
the $-1$-eigenspace $\mathfrak{p}_0$ is an $\textup{ad}_{\mathfrak{g}_0}(\mathfrak{k}_0)$-submodule of $\mathfrak{g}_0$. The complex linear extension of $\theta_0$ will be denoted by $\theta\in\textup{Aut}(\mathfrak{g})$ (it is a Chevalley involution of $\mathfrak{g}$). Then $\mathfrak{g}=\mathfrak{k}\oplus\mathfrak{p}$ is the decomposition of $\mathfrak{g}$ in $+1$ and $-1$-eigenspaces of $\theta$.

The bilinear form $(x,y)\mapsto -\langle x,\theta_0(y)\rangle_{\mathfrak{g}_0}$ on $\mathfrak{g}_0$ is positive definite.
Fix a Cartan subalgebra $\mathfrak{h}_0$ of $\mathfrak{g}_0$ which is contained in $\mathfrak{p}_0$ (this is possible since $\mathfrak{g}_0$ is split), then the restriction of $\langle\cdot,\cdot\rangle_{\mathfrak{g}_0}$ to $\mathfrak{h}_0$ is positive definite. 
We will write $(\cdot,\cdot)$ for the resulting inner product on $\mathfrak{h}_0$, and 
$\|\cdot\|$ for the norm. We use the same notations for the induced scalar product and norm on $\mathfrak{h}_0^*$. 
The complexification $\mathfrak{h}$ of $\mathfrak{h}_0$ is a Cartan subalgebra of $\mathfrak{g}$. We also write $(\cdot,\cdot)$ for the complex bilinear extensions of $(\cdot,\cdot)$ to bilinear forms on $\mathfrak{h}$ and $\mathfrak{h}^*$. 

Let
\begin{equation}\label{rootspacecomplex}
\mathfrak{g}=\mathfrak{h}\oplus\bigoplus_{\alpha\in R}\mathfrak{g}_\alpha
\end{equation}
be the root space decomposition of $\mathfrak{g}$,
with root system $R=R(\mathfrak{g},\mathfrak{h})\subset\mathfrak{h}^*$ and associated root spaces
\[
\mathfrak{g}_\alpha:=\{x\in\mathfrak{g} \,\, | \,\, \textup{ad}_{\mathfrak{g}}(h)x=\alpha(h)x\quad \forall\, h\in\mathfrak{h}\}.
\]
Fix a set $\{\alpha_1,\ldots,\alpha_\rr\}$ of simple roots of $R$. Write $R^+$ for the associated set of positive roots. 
Let $t_\lambda\in\mathfrak{h}$ be the unique element satisfying
\[
\langle h,t_\lambda\rangle_{\mathfrak{g}}=\lambda(h)\qquad \forall h\in\mathfrak{h}.
\]
Then $[x,y]=\langle x,y\rangle_{\mathfrak{g}}t_\alpha$ for root vectors $x\in\mathfrak{g}_\alpha$ and 
$y\in\mathfrak{g}_{-\alpha}$, see \cite[Prop. 8.3]{Hu}.

The root space decomposition \eqref{rootspacecomplex} refines to
\begin{equation}\label{rootspacereal}
\mathfrak{g}_0=\mathfrak{h}_0\oplus\bigoplus_{\alpha\in R}\mathfrak{g}_{0,\alpha}
\end{equation}
with $\mathfrak{g}_{0,\alpha}:=\mathfrak{g}_0\cap\mathfrak{g}_\alpha$ a one-dimensional real vector space for all $\alpha\in R$. In particular, all roots $\alpha\in R$ are real-valued on $\mathfrak{h}_0$, and $t_\alpha\in\mathfrak{h}_0$ ($\alpha\in R$).

We fix $e_\alpha\in\mathfrak{g}_{0,\alpha}$ ($\alpha\in R$)
such that 
\begin{equation}\label{mapSerre}
[e_\alpha,e_{-\alpha}]=t_\alpha,\qquad\quad \theta_0(e_\alpha)=-e_{-\alpha}
\end{equation}
for all $\alpha\in R$ (the fact that this is possible follows from, e.g., \cite[\S 25.2]{Hu}). Then $\langle e_\alpha,e_{-\alpha}\rangle_{\mathfrak{g}_0}=1$ for $\alpha\in R$. Set 
\[
y_\alpha:=e_\alpha-e_{-\alpha}\in\mathfrak{k}_0,\qquad \alpha\in R,
\]
then $y_{-\alpha}=-y_{\alpha}$ ($\alpha\in R$) and 
\begin{equation*}
\begin{split}
\mathfrak{k}_0&=\bigoplus_{\alpha\in R^+}\mathbb{R}y_\alpha,\\
\mathfrak{p}_0&=\mathfrak{h}_0\oplus\bigoplus_{\alpha\in R^+}\mathbb{R}(e_\alpha+e_{-\alpha}).
\end{split}
\end{equation*}

Let $G$ be a connected real Lie group with Lie algebra $\mathfrak{g}_0$ and finite center. Denote by $K\subset G$ the connected Lie subgroup with Lie algebra $\mathfrak{k}_0$, which is maximal compact in $G$. The Cartan involution $\theta_0$ integrates to a global Cartan involution $\Theta_0\in\textup{Aut}(G)$, and $K$ is the subgroup of elements $g\in G$ fixed by $\Theta_0$ (i.e., $(G,K)$ is a Riemannian symmetric pair).  The map 
\[
K\times \mathfrak{p}_0\rightarrow G,\qquad (k,x)\mapsto k\exp(x)
\] 
is a diffeomorphism, called the global Cartan decomposition of $G$.

\subsection{One-dimensional $\mathfrak{k}_0$-representations}\label{SSonedim}

Let $\textup{ch}(\mathfrak{k}_0)$ be the space of one-dimensional real representations of $\mathfrak{k}_0$. If $\mathfrak{g}$ is simple but not of type $C_\rr$ ($\rr\geq 1$) then $\mathfrak{k}_0$ is semisimple (see, for instance, \cite[\S 3.1]{St}), hence $\textup{ch}(\mathfrak{k}_0)=\{\chi_0\}$ with $\chi_0$ the trivial representation. 
If $\mathfrak{g}_0\simeq\mathfrak{sp}(\rr;\mathbb{R})$ ($\rr\geq 1$) then 
$\mathfrak{k}_0\simeq\mathfrak{gl}_\rr(\mathbb{R})$, hence $\textup{ch}(\mathfrak{k}_0)=\mathbb{R}\chi$ is one-dimensional.
Write in this case $R_s$ and $R_\ell$ for the set of short and long roots in $R$ with respect to the norm $\|\cdot\|$ (by convention, $R_s=\emptyset$ and $R_\ell=R$ for $\rr=1$).  Set $R_s^+:=R_s\cap R^+$
and $R_l^+:=R_l\cap R^+$. 

\begin{lem}\label{SSonedimlem}
Let $\mathfrak{g}_0=\mathfrak{sp}(\rr;\mathbb{R})$ ($\rr\geq 1$). Denote by $\chi_{\mathfrak{sp}}\in\mathfrak{k}_0^*$ the linear functional satisfying $\chi_{\mathfrak{sp}}(y_\alpha)=0$ for $\alpha\in R_s^+$ and $\chi_{\mathfrak{sp}}(y_\alpha)=1$ for $\alpha\in R_\ell^+$.
Then 
\[
\textup{ch}(\mathfrak{k}_0)=\mathbb{R}\chi_{\mathfrak{sp}}
\]
\end{lem}
\begin{proof}
See \cite[Lemma 4.3]{St}.
\end{proof}

\subsection{The Iwasawa decomposition}
Let $A\subset G$ be the connected Lie subgroup with Lie algebra $\mathfrak{h}_0$. It is a closed commutative Lie subgroup of $G$, isomorphic to $\mathfrak{h}_0$ through the restriction of the exponential map $\exp: \mathfrak{g}\rightarrow G$ to $\mathfrak{h}_0$. We write $\log: A\rightarrow \mathfrak{h}_0$ for its inverse. 

Consider the nilpotent Lie subalgebra 
\[
\mathfrak{n}_{0,+}:=\bigoplus_{\alpha\in R^+}\mathfrak{g}_{0,\alpha}
\]
of $\mathfrak{g}_0$. The vector space decomposition
\[
\mathfrak{g}_0=\mathfrak{k}_0\oplus\mathfrak{h}_0\oplus\mathfrak{n}_{0,+}
\]
is the Iwasawa decomposition of $\mathfrak{g}_0$.
Let $N_+\subset G$ be the connected Lie subgroup with Lie algebra $\mathfrak{n}_{0,+}$. Then $N_+$ is simply connected and closed in $G$, and the exponential map $\exp: \mathfrak{n}_{0,+}\rightarrow N_+$ is a diffeomorphism. 
The multiplication map 
\begin{equation}\label{mm}
K\times A\times N_+\rightarrow G, \qquad (k,a,n)\mapsto kan
\end{equation}
is a diffeomorphism onto $G$ (the global Iwasawa decomposition). 
We write 
\[
g=k(g)a(g)n(g)
\]
for the Iwasawa decomposition of $g\in G$, with $k(g)\in K$, $a(g)\in A$ and
$n(g)\in N_+$. 

Since $G$ is split with finite center,  
the centralizer $M:=Z_{K}(\mathfrak{h}_0)$ of $\mathfrak{h}_0$ in $K$ is a finite group. The minimal parabolic subgroup $P=MAN_+$ of $G$
is a closed Lie subgroup of $G$ with Lie algebra $\mathfrak{b}_0:=\mathfrak{h}_0\oplus\mathfrak{n}_{0,+}$. Note that the complexification $\mathfrak{b}$ of $\mathfrak{b}_0$
is the Borel subalgebra of $\mathfrak{g}$ containing $\mathfrak{h}$.

\section{Radial components of invariant differential operators}\label{S3}

Throughout this section we fix a triple $(\mathfrak{g}_0,\mathfrak{h}_0,\theta_0)$ with $\mathfrak{g}_0$ a split real semisimple Lie algebra, $\mathfrak{h}_0$ a split Cartan subalgebra and $\theta_0$ a Cartan involution such 
that $\theta_0|_{\mathfrak{h}_0}=-\textup{id}_{\mathfrak{h}_0}$. We write 
$\mathfrak{h}_0^*\subset\mathfrak{h}^*$ for the real span of the roots, 
$G$ for a connected Lie group with Lie algebra $\mathfrak{g}_0$ and finite center, $K\subset G$ for the connected Lie subgroup with Lie algebra $\mathfrak{k}_0$, and $A\subset G$ for the connected Lie subgroup with Lie algebra $\mathfrak{h}_0$.

\subsection{The radial component map}
The radial component map describes the factorisation of elements $x\in U(\mathfrak{g})$ 
along algebraic counterparts of the Cartan decomposition $G=KAK$. 
We first introduce some preliminary notations.

For $\lambda\in\mathfrak{h}^*$ the map 
\[
\xi_\lambda: A\rightarrow\mathbb{C}^*,\quad a\mapsto a^\lambda:=e^{\lambda(\log(a))}
\]
defines a complex-valued multiplicative character of $A$, which is real-valued for $\lambda\in\mathfrak{h}_0^*$. It satisfies $\xi_\lambda\xi_\mu=\xi_{\lambda+\mu}$ 
($\lambda,\mu\in\mathfrak{h}^*$) and $\xi_0\equiv 1$. 

The adjoint representation $\textup{Ad}: G\rightarrow \textup{Aut}(\mathfrak{g}_0)$ extends naturally to an action of $G$ on the universal enveloping algebra $U(\mathfrak{g})$ of $\mathfrak{g}$ by complex linear algebra automorphisms. We write $\textup{Ad}_g(x)$ for the adjoint action of $g\in G$ on $x\in U(\mathfrak{g})$. Note that for $a\in A$,
\[
\textup{Ad}_a(e_\alpha)=a^{\alpha}e_\alpha\qquad \forall\, \alpha\in R
\]
and $\textup{Ad}_a$ fixes $\mathfrak{h}$ pointwise.

Each $g\in G$ admits a decomposition $g=kak^\prime$ with $k,k^\prime\in K$ and $a\in A$. The double cosets $KaK$
and $Ka^\prime K$ ($a,a^\prime\in A$) coincide iff $a^\prime\in Wa$ with 
$W:=N_K(\mathfrak{h}_0)/M$
the analytic Weyl group of $G$, acting on $A$ by conjugation. Note that $W$ is isomorphic to the Weyl group of $R$ since $G$ is split.
Set 
\[
A_{\textup{reg}}:=\{a\in A \,\, | \,\, a^\alpha\not=1\quad \forall \alpha\in R\}.
\]
Then $\exp: \mathfrak{h}_{0,\textup{reg}}\overset{\sim}{\longrightarrow} A_{\textup{reg}}$, with $\mathfrak{h}_{0,\textup{reg}}:=\{h\in\mathfrak{h}_0 \,\, | \,\,
\alpha(h)\not=0\,\,\, \forall\, \alpha\in R\}$ the set of regular elements in $\mathfrak{h}_0$. The Weyl group $W$ acts freely on $A_{\textup{reg}}$. 

Infinitesimal analogues of the Cartan decomposition of $G$ are realized through the  
vector space decompositions
\begin{equation}\label{infinitesimalKAK}
\mathfrak{g}_0=\mathfrak{h}_0\oplus\textup{Ad}_{a^{-1}}\mathfrak{k}_0\oplus \mathfrak{k}_0
\end{equation}
of $\mathfrak{g}_0$ for $a\in A_{\textup{reg}}$. The decomposition \eqref{infinitesimalKAK} follows
from the identity
\begin{equation}\label{elementaryrel}
e_\alpha=\frac{a^{-\alpha}(\textup{Ad}_{a^{-1}}y_\alpha)-y_\alpha}{a^{-2\alpha}-1},
\end{equation}
which shows that $\{\textup{Ad}_{a^{-1}}y_\alpha, y_\alpha\}$ is a linear basis of 
$\mathfrak{g}_{0,\alpha}\oplus\mathfrak{g}_{0,-\alpha}$ for $a\in A_{\textup{reg}}$. Set
\[
\mathcal{V}:=U(\mathfrak{h})\otimes U(\mathfrak{k})\otimes U(\mathfrak{k}).
\]
By the Poincar{\'e}-Birkhoff-Witt-Theorem, for each $a\in A_{\textup{reg}}$ the linear map
\[
\Gamma_a: \mathcal{V}\rightarrow U(\mathfrak{g}),\qquad
\Gamma_a(h\otimes x\otimes y):=\textup{Ad}_{a^{-1}}(x)hy
\]
is a linear isomorphism. 

Extending the scalars of the complex vector space $\mathcal{V}$ to the ring
$C^\infty(A_{\textup{reg}})$ of complex valued smooth functions on $A_{\textup{reg}}$ allows one to give the factorisation $\Gamma_a^{-1}(x)$ for $x\in
U(\mathfrak{g})$ uniformly in $a\in A_{\textup{reg}}$. It suffices to extend the scalars to the
unital subring $\mathcal{R}$ of $C^\infty(A_{\textup{reg}})$ generated by $\xi_{-\alpha}$ and
$(1-\xi_{-2\alpha})^{-1}$ for all $\alpha\in R^+$. For $a\in A_{\textup{reg}}$ the extension
of $\Gamma_a$ is then the complex linear
map $\widetilde{\Gamma}_a: \mathcal{R}\otimes\mathcal{V}\rightarrow U(\mathfrak{g})$ defined by
\[
\widetilde{\Gamma}_a(f\otimes Z):=f(a)\Gamma_a(Z),\qquad f\in\mathcal{R},\,\,\, Z\in\mathcal{V}.
\]
\begin{thm}[\cite{CM}] \label{infKAK}
For $x\in U(\mathfrak{g})$ there exists a unique $\Pi(x)\in\mathcal{R}\otimes\mathcal{V}$
such that 
\[
\widetilde{\Gamma}_a\bigl(\Pi(x)\bigr)=x\qquad \forall\, a\in A_{\textup{reg}}.
\]
\end{thm}
For example, by \eqref{elementaryrel},
\[
\Pi(e_\alpha)=(\xi_{-\alpha}-\xi_\alpha)^{-1}\otimes 1\otimes y_\alpha\otimes 1
-(\xi_{-2\alpha}-1)^{-1}\otimes 1\otimes 1\otimes y_\alpha.
\]
The resulting linear map $\Pi: U(\mathfrak{g})\rightarrow\mathcal{R}\otimes\mathcal{V}$ is called
the radial component map. 
\subsection{$\sigma$-Spherical functions}
The radial component map plays an important role in the study of spherical functions.
Fix a finite dimensional representation $\sigma: K\times K\rightarrow \textup{GL}(V_\sigma)$. 
Denote by
$C^\infty(G;V_\sigma)$ the space of smooth $V_\sigma$-valued functions on $G$.
\begin{defi}\label{defispher}
We say that $f\in C^\infty(G;V_\sigma)$ is a $\sigma$-spherical function on $G$ if 
\[
f(k_1gk_2^{-1})=\sigma(k_1,k_2)f(g)\qquad \forall\, g\in G,\,\,\, \forall k_1,k_2\in K.
\]
We denote by $C^\infty_\sigma(G)$ the subspace of $C^\infty(G;V_\sigma)$ consisting of
$\sigma$-spherical functions on $G$.
\end{defi}
Examples of $\sigma$-spherical functions on $G$ are 
\[
E_\lambda^\sigma(\cdot)v\in C^\infty_\sigma(G),\qquad v\in V_\sigma
\]
where $E_\lambda^\sigma: G\rightarrow\textup{End}(V_\sigma)$ for $\lambda\in\mathfrak{h}^*$ is
the Eisenstein integral
\begin{equation}\label{Eisenstein}
E_\lambda^\sigma(g):=\int_Kdx\, 
\xi_{-\lambda-\rho}(a(g^{-1}x))
\sigma(x,k(g^{-1}x)).
\end{equation}
Here $\rho:=\frac{1}{2}\sum_{\alpha\in R^+}\alpha\in\mathfrak{h}^*$ and $dx$ is the normalised Haar measure on $K$. 
The representation theoretic construction of $\sigma$-spherical functions (see, e.g., \cite[\S 8]{CM}) will be discussed in Section \ref{SectionRepTh}.

Let $V_\sigma^M$ be the subspace of $M$-invariant elements in $V_\sigma$, with $M$ acting diagonally on $V_\sigma$. 
The function space $C^\infty(A;V_\sigma^M)$ is a $W$-module with $w=kM\in W$ for $k\in N_K(\mathfrak{h}_0)$ acting by
\[
(w\cdot f)(a):=\sigma(k,k)f(k^{-1}ak),\qquad\quad a\in A,\, f\in C^\infty(A;V_\sigma^M).
\]
We write $C^\infty(A;V_\sigma^M)^W$ for the subspace
of $W$-invariant $V_\sigma^M$-valued smooth functions on $A$.
By the Cartan decomposition of $G$, we have the following well known result.
\begin{cor}
The map $C^\infty(G;V_\sigma)\rightarrow C^\infty(A;V_\sigma)$, $f\mapsto f\vert_{A}$ restricts
to an injective linear map from $C_\sigma^\infty(G)$
into $C^\infty(A;V_\sigma^M)^W$. Similarly, restriction 
to $A_{\textup{reg}}$ defines an injective linear map
$C_\sigma^\infty(G)\hookrightarrow C^\infty(A_{\textup{reg}};V_\sigma^M)^W$.
\end{cor}

The action of left $G$-invariant differential operators on $C^\infty_\sigma(G)$, pushed through the restriction map $\vert_{A_{\textup{reg}}}$, gives rise to differential operators on $A_{\textup{reg}}$ that can be described explicitly in terms of the radial component map $\Pi$. We describe them in the next subsection.

\subsection{Invariant differential operators}

Denote by $\ell$ and $r$ the left-regular and right-regular representations of $G$ on $C^\infty(G)$ respectively,
\[
(\ell(g)f)(g^\prime):=f(g^{-1}g^\prime),\qquad (r(g)f)(g^\prime):=f(g^\prime g),
\]
with $g,g^\prime\in G$ and $f\in C^\infty(G)$. Let $\mathbb{D}(G)$ be the ring of differential operators on $G$, and $\mathbb{D}(G)^G\subseteq\mathbb{D}(G)$ its subalgebra of left $G$-invariant differential operators.
Differentiating $r$ gives an isomorphism
\[
r_\ast: U(\mathfrak{g})\overset{\sim}{\longrightarrow} \mathbb{D}(G)^G
\]
of algebras. 

Let $U(\mathfrak{g})^M\subseteq U(\mathfrak{g})$ be the subalgebra of $\textup{Ad}(M)$-invariant
elements in $U(\mathfrak{g})$.
Embed $\mathbb{D}(G)$ into $\mathbb{D}(G)\otimes\textup{End}(V_\sigma)$ by
$D\mapsto D\otimes\textup{id}_{V_\sigma}$ ($D\in\mathbb{D}(G)$). 
With respect to the resulting action $r_\ast$ of $U(\mathfrak{g})$ on $C^\infty(G; V_{\sigma})$, the subspace $C^\infty_\sigma(G)$ of $\sigma$-spherical functions is a $U(\mathfrak{g})^K$-invariant subspace of $C^\infty(G; V_\sigma)$.

Let $\mathbb{D}(A)$ be the ring of differential operators on $A$ and $\mathbb{D}(A)^A$ the subalgebra of $A$-invariant differential operators. Let $r^A$ be the right-regular action of $A$ on $C^\infty(A)$. Its differential gives rise to an algebra isomorphism
\begin{equation}\label{rhoA}
r^A_\ast: U(\mathfrak{h})\overset{\sim}{\longrightarrow}\mathbb{D}(A)^A.
\end{equation}
We will write $\partial_h:=r^A_\ast(h)\in\mathbb{D}(A)^A$ for $h\in\mathfrak{h}_0$, which are the derivations 
\[
\bigl(\partial_hf\bigr)(a)=\frac{d}{dt}\biggr\rvert_{t=0}f\bigl(a\exp_A(th)\bigr)
\]
for $f\in C^\infty(A)$ and $a\in A$. 
We also consider $\mathbb{D}(A)^A$ as the subring of $\mathbb{D}(A_{\textup{reg}})$ consisting of
constant coefficient differential operators and write
\[
\mathbb{D}_{\mathcal{R}}\subset\mathbb{D}(A_{\textup{reg}})
\]
for the algebra of differential operators
\[
D=\sum_{m_1,\ldots,m_\rr}c_{m_1,\ldots,m_\rr}\partial_{x_1}^{m_1}\cdots\partial_{x_\rr}^{m_\rr}\in
\mathbb{D}(A_{\textup{reg}})
\]
with coefficients $c_{m_1,\ldots,m_\rr}\in\mathcal{R}$, where $\{x_1,\ldots,x_\rr\}$ is an orthonormal basis of $\mathfrak{h}_0$ with respect to $(\cdot,\cdot)$.
The algebra isomorphism \eqref{rhoA} now extends to a complex linear isomorphism
\[
\widetilde{r}_\ast^A: \mathcal{R}\otimes U(\mathfrak{h})\overset{\sim}{\longrightarrow}\mathbb{D}_{\mathcal{R}}, \quad f\otimes h\mapsto fr_\ast^A(h)
\]
for $f\in\mathcal{R}$ and $h\in U(\mathfrak{h})$. Finally, $\mathbb{D}_{\mathcal{R}}\otimes U(\mathfrak{k})^{\otimes 2}$ will denote the algebra of differential operators 
$D=\sum_{m_1,\ldots,m_\rr}c_{m_1,\ldots,m_\rr}\partial_{x_1}^{m_1}\cdots
\partial_{x_\rr}^{m_\rr}$ on $A_{\textup{reg}}$ with coefficients $c_{m_1,\ldots,m_\rr}$ in $\mathcal{R}\otimes U(\mathfrak{k})^{\otimes 2}$.
It acts naturally on $C^\infty(A_{\textup{reg}}; V_\sigma)$.

By the proof
of \cite[Thm. 3.1]{CM} we have for $f\in C^{\infty}_\sigma(G)$, 
$h\in U(\mathfrak{h})$ and $x,y\in U(\mathfrak{k})$,
\[
\bigl(r_\ast\bigl(\textup{Ad}_{a^{-1}}(x)hy\bigr)f\bigr)(a)=
\sigma(x\otimes S(y))\bigl(r_\ast(h)f\bigr)(a)\qquad \forall\, a\in A_{\textup{reg}},
\]
with $S$ the antipode of $U(\mathfrak{k})$, defined as the anti-algebra homomorphism of $U(\mathfrak{k})$ such that $S(x)=-x$ for all $x\in \mathfrak{k}$. Combined with
Theorem \ref{infKAK} this leads to the following result. 
\begin{thm}\label{thmRAD}
With the above conventions, define the linear map 
\[
\widehat{\Pi}: U(\mathfrak{g})\rightarrow \mathbb{D}_{\mathcal{R}}
\otimes U(\mathfrak{k})^{\otimes 2}\]
by $\widehat{\Pi}:=(\widetilde{r}^A_\ast\otimes\textup{id}_{U(\mathfrak{k})}\otimes S)
\Pi$,
and set
\[
\widehat{\Pi}^\sigma:=(\textup{id}_{\mathbb{D}_{\mathcal{R}}}\otimes\sigma)\widehat{\Pi}:
U(\mathfrak{g})\rightarrow \mathbb{D}_{\mathcal{R}}\otimes\textup{End}(V_\sigma).
\] 
\begin{enumerate}
\item[{\bf a.}] For $z\in U(\mathfrak{g})$,
\[
\bigl(r_\ast(z)f\bigr)\vert_{A_{\textup{reg}}}=\widehat{\Pi}^\sigma(z)\bigl(
f\vert_{A_{\textup{reg}}}\bigr)\qquad
\forall\, f\in C^{\infty}_\sigma(G).
\]
\item[{\bf b.}] The restrictions of $\widehat{\Pi}$ and $\widehat{\Pi}^\sigma$ to $Z(\mathfrak{g})$ are algebra homomorphisms. 
\end{enumerate}
\end{thm}
\begin{proof}
{\bf a.} This is a well-known result of Harish-Chandra, see, e.g., \cite[Thm. 3.1]{CM}.

{\bf b.} It is well-known that the differential operators $\widehat{\Pi}(z)$ ($z\in Z(\mathfrak{g})$) pairwise commute when acting on $C^\infty(A_{\textup{reg}};V_\sigma^M)$, see \cite[Thm. 3.3]{CM}.
The theory of formal spherical functions which we develop in Section \ref{RTHC}, implies that
they also commute as $U(\mathfrak{k})^{\otimes 2}$-valued differential operators. The key point is that all formal spherical functions are formal power series eigenfunctions of $\widehat{\Pi}(z)$ 
($z\in Z(\mathfrak{g})$) by Theorem \ref{mainTHMF}\,{\bf a}, which forces the differential operators $\widehat{\Pi}(z)$ ($z\in Z(\mathfrak{g})$) to commute as $U(\mathfrak{k})^{\otimes 2}$-valued differential operators by the results in Section \ref{S66} for the special case $N=0$.
\end{proof}
\begin{rema}\label{Minvariance}
By \cite[Prop. 2.5]{CM} we have $\widehat{\Pi}(z)\in\mathbb{D}_{\mathcal{R}}\otimes
U(\mathfrak{k}\oplus\mathfrak{k})^M$ for $z\in U(\mathfrak{g})^M$, where 
$U(\mathfrak{k}\oplus\mathfrak{k})^M$ is the space of $M$-invariance in 
$U(\mathfrak{k}\oplus\mathfrak{k})\simeq U(\mathfrak{k})^{\otimes 2}$ with respect to the diagonal adjoint action of $M$ on $U(\mathfrak{k}\oplus\mathfrak{k})$. In particular, 
$\widehat{\Pi}^\sigma(z)\in\mathbb{D}_{\mathcal{R}}\otimes\textup{End}_M(V_\sigma)$ for
$z\in U(\mathfrak{g})^M$, with $M$ acting diagonally on $V_\sigma$.
\end{rema}

\subsection{The radial component of the Casimir element}\label{S34}
In this subsection we recall the computation of the radial component of the Casimir element.
As before, let $e_\alpha\in\mathfrak{g}_{0,\alpha}$ ($\alpha\in R$) such that
$[e_\alpha,e_{-\alpha}]=t_{\alpha}$ and $\theta_0(e_\alpha)=-e_{-\alpha}$ ($\alpha\in R$), and $\{x_1,\ldots,x_\rr\}$ an orthonormal basis of $\mathfrak{h}_0$ with respect to $(\cdot,\cdot)$.
The Casimir element $\Omega\in Z(\mathfrak{g})$ is given by
\begin{equation}\label{Omega}
\begin{split}
\Omega&=\sum_{j=1}^\rr x_j^2+\sum_{\alpha\in R}e_\alpha e_{-\alpha}\\
&=\sum_{j=1}^\rr x_j^2+2t_\rho+2\sum_{\alpha\in R^+}e_{-\alpha}e_{\alpha}.
\end{split}
\end{equation}
By  \eqref{elementaryrel}, the second line of \eqref{Omega}, and by
\[
[y_\alpha,\textup{Ad}_{a^{-1}}y_\alpha]=(a^{-\alpha}-a^\alpha)t_\alpha\qquad 
\forall\, a\in A, 
\]
we obtain the following Cartan factorisation of $\Omega$, 
\begin{equation}\label{KAKomega}
\Omega=\sum_{j=1}^\rr x_j^2+
\frac{1}{2}\sum_{\alpha\in R}\left(\frac{a^\alpha+a^{-\alpha}}{a^\alpha-a^{-\alpha}}\right)t_\alpha
+\sum_{\alpha\in R}\left(
\frac{\textup{Ad}_{a^{-1}}(y_\alpha^2)-(a^{\alpha}+a^{-\alpha})\textup{Ad}_{a^{-1}}(y_\alpha)y_\alpha
+y_\alpha^2}{(a^{\alpha}-a^{-\alpha})^2}\right)
\end{equation}
for arbitrary $a\in A_{\textup{reg}}$. It follows that
\begin{equation}\label{piomega}
\begin{split}
\Pi(\Omega)&= \sum_{j=1}^\rr 1\otimes x_j^2\otimes 1\otimes 1
+\frac{1}{2}\sum_{\alpha\in R}\left(\frac{\xi_\alpha+\xi_{-\alpha}}{\xi_\alpha-\xi_{-\alpha}}\right)\otimes t_\alpha
\otimes 1\otimes 1\\
&+\sum_{\alpha\in R}\left\{\frac{1}{(\xi_{\alpha}-\xi_{-\alpha})^2}\otimes 1
\otimes (y_\alpha^2\otimes 1+1\otimes y_\alpha^2)-
\frac{(\xi_{\alpha}+\xi_{-\alpha})}{(\xi_{\alpha}-\xi_{-\alpha})^2}\otimes 1\otimes y_\alpha\otimes y_\alpha\right\}.
\end{split}
\end{equation}
This gives the following result, cf., e.g., \cite[Prop. 9.1.2.11]{W}.
\begin{cor}\label{corR1}
The differential operator $\widehat{\Pi}(\Omega)\in\mathbb{D}_{\mathcal{R}}\otimes U(\mathfrak{k})^{\otimes 2}$ is given by
\begin{equation}\label{pihat}
\widehat{\Pi}(\Omega)=\Delta+\frac{1}{2}\sum_{\alpha\in R}\left(\frac{\xi_\alpha+\xi_{-\alpha}}
{\xi_\alpha-\xi_{-\alpha}}\right)\partial_{t_\alpha}
+\sum_{\alpha\in R}\frac{1}{(\xi_{\alpha}-\xi_{-\alpha})^2}
\prod_{\epsilon\in\{\pm 1\}}(y_\alpha\otimes 1+\xi_{\epsilon\alpha}(1\otimes
y_\alpha))
\end{equation}
with $\Delta:=\sum_{j=1}^\rr\partial_{x_j}^2$ the Laplace-Beltrami operator on $A$.
\end{cor}

\begin{rema}
Note that the infinitesimal Cartan factorisations \eqref{KAKomega} of $\Omega$ are parametrised by elements $a\in A_{\textup{reg}}$. In the context of boundary Knizhnik-Zamolodchikov equations (see Section \ref{SectionbKZB}) these will provide the dynamical parameters.
\end{rema}

There are various ways to factorise $\Omega$, of which \eqref{Omega} and the infinitesimal 
Cartan de\-com\-po\-sition \eqref{KAKomega} are two natural ones. Another factorisation is 
\begin{equation}\label{Aomega}
\Omega=\sum_{j=1}^\rr x_j^2+\frac{1}{2}\sum_{\alpha\in R}\left(\frac{1+a^{-2\alpha}}{1-a^{-2\alpha}}\right)t_\alpha+2\sum_{\alpha\in R}\frac{e_{-\alpha}e_\alpha}{1-a^{-2\alpha}}
\end{equation}
for $a\in A_{\textup{reg}}$, which is a dynamical version of \eqref{Omega}. This formula can be easily proved by moving in \eqref{Aomega} positive root vectors $e_\alpha$ ($\alpha\in R^+$) to the left and using $[e_\alpha,e_{-\alpha}]=t_\alpha$, which causes
the "dynamical" dependence to drop out and reduces \eqref{Aomega} to the second formula of
\eqref{Omega}. The decomposition \eqref{Aomega} is the natural factorisation of $\Omega$ in the context of Etingof's and Schiffmann's \cite{ES}
generalised weighted trace functions and associated asymptotic KZB equations, see \cite{St2}.
\subsection{$\chi$-invariant vectors}\label{chisection}

Let $V$ be a $\mathfrak{g}_0$-module and fix $\chi\in\textup{ch}(\mathfrak{k}_0)$. We say that a vector $v\in V$ is $\chi$-invariant if $xv=\chi(x)v$ for all $x\in\mathfrak{k}_0$. We write $V^\chi$ for the subspace of $\chi$-invariant vectors in $V$,
\[
V^\chi=\{v\in V \,\, | \,\, e_\alpha v-e_{-\alpha}v=\chi(y_\alpha)v\quad \forall\, \alpha\in R^+\}.
\]
In case of the trivial one-dimensional representation $\chi_0\equiv 0$, 
we write $V^{\chi_0}=V^{\mathfrak{k}_0}$, which is the space of $\mathfrak{k}_0$-fixed vectors in $V$.
{}From the computation of the radial component of the Casimir $\Omega$ in the previous subsection, we obtain the following corollary.
\begin{cor}\label{corO}
Let $V$ be a $\mathfrak{g}_0$-module such that $\Omega|_V=c\,\textup{id}_V$ for some 
$c\in\mathbb{R}$. Fix $\chi\in\textup{ch}(\mathfrak{k}_0)$ and $v\in V^\chi$.
Then 
\[
\Bigl(\sum_{j=1}^\rr x_j^2+\frac{1}{2}\sum_{\alpha\in R}\Bigl(\frac{1+a^{-2\alpha}}{1-a^{-2\alpha}}\Bigr)t_\alpha
+\sum_{\alpha\in R}\frac{1}{(a^{-\alpha}-a^\alpha)^2}
\prod_{\epsilon\in\{\pm 1\}}(\textup{Ad}_{a^{-1}}(y_\alpha)-a^{-\epsilon\alpha}\chi(y_\alpha))
\Bigr)v=cv
\]
for all $a\in A_{\textup{reg}}$.
\end{cor}
If $V$ is $\mathfrak{h}_0$-diagonalisable then Corollary \ref{corO} reduces to explicit recursion relations for the weight components of $v\in V^\chi$. 

\begin{rema}
In the setup of the corollary, a vector $u\in V$ is a Whittaker vector of weight $a\in A_{\textup{reg}}$ if 
$e_\alpha u=a^\alpha u$ for all $\alpha\in R^+$. Recursion relations for the weight components of Whittaker vectors are used in \cite[\S 3.2]{DKT} to derive a path model for Whittaker vectors (\cite[Thm. 3.7]{DKT}), as well as for the associated Whittaker functions (\cite[Thm. 3.9]{DKT}). It would be interesting to see what this approach entails for $\sigma$-spherical functions with $\sigma$ a one-dimensional representation of $K\times K$, when the role of the Whittaker vectors is taken over by $\chi$-invariant vectors. 
\end{rema}

\subsection{Quantum $\sigma$-spin hyperbolic Calogero-Moser systems}\label{vvCM}

We gauge the commuting differential operators $\widehat{\Pi}(z)$ ($z\in Z(\mathfrak{g})$) to
give them the interpretation as quantum Hamiltonians for spin generalisations (in the physical sense) of the quantum hyperbolic Calogero-Moser system. This extends results from \cite{Ga,OP,HOI} and \cite[Part I, Chpt. 5]{HS}, which deal with the "spinless" cases.

Write
\[
A_+:=\{a\in A\,\,\, | \,\,\, a^\alpha>1\quad \forall\,\alpha\in R^+\}
\]
for the positive chamber of $A_{\textup{reg}}$.
Note that $\mathcal{R}$ is contained in the ring $C^\omega(A_{+})$
of analytic functions on $A_+$.

Let $\delta$ be the analytic function on $A_+$ given by
\begin{equation}\label{deltagauge}
\delta(a):=a^{\rho}\prod_{\alpha\in R^+}(1-a^{-2\alpha})^{\frac{1}{2}}.
\end{equation}
Conjugation by $\delta$ defines an
outer automorphism of $\mathbb{D}_{\mathcal{R}}$. For $z\in U(\mathfrak{g})$ we denote by
\begin{equation}\label{Hz}
H_z:=\delta\circ\widehat{\Pi}(z)\circ\delta^{-1}\in\mathbb{D}_{\mathcal{R}}\otimes U(\mathfrak{k})^{\otimes 2}
\end{equation}
the corresponding gauged differential operator. We furthermore write
\begin{equation}\label{fatH}
\mathbf{H}:=-\frac{1}{2}\delta\circ\bigl(\widehat{\Pi}(\Omega)+\|\rho\|^2\bigr)\circ\delta^{-1}.
\end{equation}

 \begin{prop}\label{qH}
 The assignment $z\mapsto H_z$ defines an algebra map
 $Z(\mathfrak{g})\rightarrow \mathbb{D}_{\mathcal{R}}\otimes U(\mathfrak{k}\oplus\mathfrak{k})^{M}$.  Furthermore, 
\begin{equation}\label{qHam}
\mathbf{H}=-\frac{1}{2}\Delta-\frac{1}{2}\sum_{\alpha\in R}\frac{1}{(\xi_\alpha-\xi_{-\alpha})^2}
\Bigl(\frac{\|\alpha\|^2}{2}+\prod_{\epsilon\in\{\pm 1\}}(y_\alpha\otimes 1+
\xi_{\epsilon\alpha}(1\otimes y_\alpha))\Bigr).
\end{equation}
\end{prop}
\begin{proof}
The first statement is immediate from Theorem \ref{thmRAD}.
The proof of \eqref{qHam} follows from the well known fact that
\[
\delta\circ\Bigl(\Delta+\frac{1}{2}\sum_{\alpha\in R}
\Bigl(\frac{\xi_\alpha+\xi_{-\alpha}}{\xi_\alpha-\xi_{-\alpha}}\Bigr)\partial_{t_\alpha}\Bigr)\circ\delta^{-1}=
\Delta-\|\rho\|^2+\sum_{\alpha\in R}\frac{\|\alpha\|^2}{2}\frac{1}{(\xi_\alpha-\xi_{-\alpha})^2},
\]
see, e.g.,  the proof of \cite[Part I, Thm. 2.1.1]{HS}.
\end{proof}
For $\sigma: U(\mathfrak{k})^{\otimes 2}\rightarrow\textup{End}(V_\sigma)$ a finite dimensional representation we write
\[
H_z^\sigma:=\bigl(\textup{id}_{\mathbb{D}_{\mathcal{R}}}\otimes\sigma\bigr)H_z
\in\mathbb{D}_{\mathcal{R}}\otimes \textup{End}(V_\sigma)\qquad (z\in U(\mathfrak{g})).
\]
Then $H_z^\sigma$ ($z\in Z(\mathfrak{g})$) are commuting $\textup{End}(V_\sigma)$-valued
differential operators on $A$ which, by Proposition \ref{qH}, serve as quantum Hamiltonians for the $\sigma$-spin generalisation of the quantum hyperbolic Calogero-Moser system with Schr{\"o}dinger operator 
\[
\mathbf{H}^\sigma:=\bigl(\textup{id}_{\mathbb{D}_{\mathcal{R}}}\otimes\sigma\bigr)(\mathbf{H}).
\]
We now list a couple of interesting special cases of the quantum $\sigma$-spin hyperbolic Calogero-Moser systems.\\
 
 \noindent
{\bf The spinless case:}  Take $\chi^\ell,\chi^r\in\textup{ch}(\mathfrak{k}_0)$. Their extension to
 complex linear algebra morphisms $U(\mathfrak{k})\rightarrow\mathbb{C}$ are again denoted by $\chi^\ell$ and 
 $\chi^r$. 
  Define $\chi^{\ell,r}_\alpha\in C^\omega(A)$ ($\alpha\in R$) by
 \begin{equation}\label{chilr}
 \chi^{\ell,r}_\alpha(a):=\chi^\ell(y_\alpha)+a^\alpha\chi^r(y_\alpha),\qquad a\in A_+.
 \end{equation}
Note that $\chi_{-\alpha}^{\ell,r}(a)=-(\chi^\ell(y_\alpha)+a^{-\alpha}\chi^r(y_\alpha))$ for $\alpha\in
 R$. The Schr{\"o}dinger operator $\mathbf{H}^{\chi^\ell\otimes\chi^r}$ then becomes
 \[
 \mathbf{H}^{\chi^\ell\otimes \chi^r}=
 -\frac{1}{2}\Delta-\frac{1}{2}\sum_{\alpha\in R}\frac{1}{(\xi_\alpha-\xi_{-\alpha})^2}
 \Bigl(\frac{\|\alpha\|^2}{2}-\chi_\alpha^{\ell,r}\chi_{-\alpha}^{\ell,r}\Bigr).
 \] 
 The special case
  \[
 \mathbf{H}^{\chi_0\otimes\chi_0}=-\frac{1}{2}\Delta-\frac{1}{2}\sum_{\alpha\in R}\frac{\|\alpha\|^2}{2}
 \frac{1}{(\xi_\alpha-\xi_{-\alpha})^2}
 \]
 with $\chi_0\in\textup{ch}(\mathfrak{k}_0)$ the trivial representation
 is the quantum Hamiltonian of the quantum hyperbolic Calogero-Moser system associated to 
 the Riemannian symmetric space $G/K$. If $\mathfrak{g}$ is simple and of type $C_\rr$ ($\rr\geq 1$) then  $\chi^\ell=c_\ell\chi_{\mathfrak{sp}}$ and
 $\chi^r=c_r\chi_{\mathfrak{sp}}$ for some $c_\ell, c_r\in\mathbb{C}$, 
  see Lemma \ref{SSonedimlem}. Using the explicit description of $\chi_{\mathfrak{sp}}$ from Lemma
 \ref{SSonedimlem}, we then obtain
 \begin{equation*}
 \mathbf{H}^{\chi^\ell\otimes\chi^r}=
-\frac{1}{2}\Delta-\frac{1}{2}\sum_{\alpha\in R_s^+}\frac{\|\alpha\|^2}{(\xi_\alpha-\xi_{-\alpha})^2}\\
+\frac{1}{4}\sum_{\beta\in R_\ell^+}\frac{\frac{1}{2}\|\beta\|^2+(c_\ell-c_r)^2}{(\xi_{\beta/2}+
\xi_{-\beta/2})^2}-\frac{1}{4}\sum_{\beta\in R_\ell^+}\frac{\frac{1}{2}\|\beta\|^2+(c_\ell+c_r)^2}
{(\xi_{\beta/2}-\xi_{-\beta/2})^2},
\end{equation*}
hence we recover a two-parameter subfamily of the $\textup{BC}_\rr$ quantum hyperbolic
 Calogero-Moser system. This extends 
\cite[Part I, Thm. 5.1.7]{HS}, which deals with the special case that $\chi^\ell=-\chi^r$ with 
$\chi^\ell\in\textup{ch}(\mathfrak{k}_0)$ integrating to a multiplicative character of $K$.\\
 
\noindent
{\bf The one-sided spin case:} Let $\chi\in\textup{ch}(\mathfrak{k}_0)$ and $\sigma_\ell: U(\mathfrak{k})\rightarrow
\textup{End}(V_\ell)$ a finite dimensional representation. Then 
\[
\mathbf{H}^{\sigma_\ell\otimes\chi}=-\frac{1}{2}\Delta-\frac{1}{2}\sum_{\alpha\in R}\frac{1}{(\xi_\alpha-\xi_{-\alpha})^2}
\Bigl(\frac{\|\alpha\|^2}{2}+\prod_{\epsilon\in\{\pm 1\}}(\sigma_\ell(y_\alpha)+
\xi_{\epsilon\alpha}\chi(y_\alpha))\Bigr).
\]
In the special case that $\chi=\chi_0\in\textup{ch}(\mathfrak{k}_0)$ is the trivial representation the Schr{\"o}dinger operator reduces to
\[
\mathbf{H}^{\sigma_\ell\otimes\chi_0}=-\frac{1}{2}\Delta-\frac{1}{2}\sum_{\alpha\in R}\frac{1}{(\xi_\alpha-\xi_{-\alpha})^2}
\Bigl(\frac{\|\alpha\|^2}{2}+\sigma_\ell(y_\alpha^2)\Bigr).
\]
Finally, if $\mathfrak{g}$ is simple and of type $C_\rr$ ($\rr\geq 1$) and $\chi=c\chi_{\mathfrak{sp}}$ with
$c\in\mathbb{C}$, then
\begin{equation*}
\begin{split}
\mathbf{H}^{\sigma_\ell\otimes\chi}=-\frac{1}{2}\Delta&-\frac{1}{2}\sum_{\alpha\in R^+_s}
\frac{\|\alpha\|^2+2\sigma_\ell(y_\alpha^2)}{(\xi_\alpha-\xi_{-\alpha})^2}\\
&+\frac{1}{4}\sum_{\beta\in R_\ell^+}\frac{\frac{1}{2}\|\beta\|^2+(\sigma_\ell(y_\beta)-c)^2}
{(\xi_{\beta/2}+\xi_{-\beta/2})^2}-\frac{1}{4}
\sum_{\beta\in R_\ell^+}\frac{\frac{1}{2}\|\beta\|^2+(\sigma_\ell(y_\beta)+c)^2}{(\xi_{\beta/2}-\xi_{-\beta/2})^2}.
\end{split}
\end{equation*}

\begin{rema}
Feh{\'e}r and Pusztai \cite{FP1,FP2} obtained the classical analog of the one-sided quantum spin  
Calogero-Moser system
by Hamiltonian reduction. This is extended to double-sided spin 
Calogero-Moser systems in \cite{Re}.
\end{rema}

\noindent
{\bf The matrix case:} The following special case is relevant for the theory of matrix-valued
spherical functions \cite{GPT}, \cite[\S 7]{HvP}. 
Let $\tau: \mathfrak{k}\rightarrow \mathfrak{gl}(V_\tau)$ be a finite dimensional representation.
Consider $\textup{End}(V_\tau)$ as left $U(\mathfrak{k})^{\otimes 2}$-module by
\begin{equation}\label{Endaction}
\sigma_\tau(x\otimes y)T:=\tau(x)T\tau(S(y))
\end{equation}
for $x,y\in U(\mathfrak{k})$ and $T\in\textup{End}(V_\tau)$. Note that $\textup{End}(V_\tau)\simeq
V_\tau\otimes V_\tau^\ast$ as $U(\mathfrak{k})^{\otimes 2}$-modules. 
The associated Schr{\"o}dinger operator $\mathbf{H}^{\sigma_\tau}$ acts on 
$T\in C^\infty(A_{\textup{reg}};\textup{End}(V_\tau))$ 
by 
\begin{equation*}
\begin{split}
\bigl(\mathbf{H}^{\sigma_\tau}T\bigr)(a)=&-\frac{1}{2}(\Delta T)(a)\\
&-\frac{1}{2}\sum_{\alpha\in R}\frac{\frac{\|\alpha\|^2}{2}T(a)+\tau(y_\alpha^2)T(a)
-(a^\alpha+a^{-\alpha})\tau(y_\alpha)T(a)\tau(y_\alpha)+T(a)\tau(y_\alpha^2)}
{(a^\alpha-a^{-\alpha})^2}
\end{split}
\end{equation*}
for $a\in A_{\textup{reg}}$.

\subsection{$\sigma$-Harish-Chandra series}\label{S5}

In this subsection we recall the construction of the Harish-Chandra series following \cite[Chpt. 9]{W}. They were defined by Harish-Chandra to analyse the asymptotic behaviour of matrix coefficients of admissible $G$-representations and of the associated spherical functions
(see, e.g., \cite{BS, CM, HS} and references therein).

Consider the ring $\mathbb{C}[[\xi_{-\alpha_1},\ldots,\xi_{-\alpha_\rr}]]$ of formal power series at infinity in $A_+$. We express elements $f\in\mathbb{C}[[\xi_{-\alpha_1},\ldots,\xi_{-\alpha_\rr}]]$ 
as $f=\sum_{\gamma\in Q_-}c_\gamma\xi_\gamma$
 with $c_\gamma\in\mathbb{C}$ and
 \[
 Q_-:=\bigoplus_{j=1}^\rr\mathbb{Z}_{\leq 0}\,\alpha_j
 \subseteq Q:=\mathbb{Z}R.
 \]
We consider $\mathcal{R}$ as subring of $\mathbb{C}[[\xi_{-\alpha_1},\ldots,\xi_{-\alpha_\rr}]]$
using power series expansion at infinity in $A_+$ (e.g.,
 $(1-\xi_{-2\alpha})^{-1}=\sum_{m=0}^{\infty}\xi_{-2m\alpha}$ for $\alpha\in R^+$). Similarly,
 we view $\xi_{-\rho}\delta$ as element in $\mathbb{C}[[\xi_{-\alpha_1},\ldots,\xi_{-\alpha_\rr}]]$ through its power series expansion at infinity, where $\delta$ is given by \eqref{deltagauge}.

For $B$ a complex associative algebra we write
 $B[[\xi_{-\alpha_1},\ldots,\xi_{-\alpha_\rr}]]\xi_\lambda$ for the
 $\mathbb{C}[[\xi_{-\alpha_1},\ldots,\xi_{-\alpha_\rr}]]$-module of
 formal series
 $g=\sum_{\gamma\in Q_-}d_\gamma\xi_{\lambda+\gamma}$
 with coefficients $d_\gamma\in B$. If $B=U(\mathfrak{k})^{\otimes 2}$ or $B=\textup{End}(V_\sigma)$ for some $\mathfrak{k}\oplus\mathfrak{k}$-module $V_\sigma$ then
$B[[\xi_{-\alpha_1},\ldots,\xi_{-\alpha_\rr}]]\xi_\lambda$ becomes a $\mathbb{D}_{\mathcal{R}}\otimes U(\mathfrak{k})^{\otimes 2}$-module.
  
Set
\begin{equation}\label{HCgeneric}
\mathfrak{h}_{\textup{HC}}^*:=\{\lambda\in\mathfrak{h}^* \,\, | \,\, 
(2(\lambda+\rho)+\gamma,\gamma)\not=0\quad \forall\, \gamma\in Q_-\setminus\{0\}\}.
\end{equation}
The Harish-Chandra series associated to the triple $(\mathfrak{g}_0,\mathfrak{h}_0,\theta_0)$
is the following formal $U(\mathfrak{k})^{\otimes 2}$-valued eigenfunction
of the 
$U(\mathfrak{k})^{\otimes 2}$-valued differential operator $\widehat{\Pi}(\Omega)$.

\begin{prop}\label{cordefHC}
Let $\lambda\in
\mathfrak{h}_{\textup{HC}}^*$.
There exists a unique $U(\mathfrak{k})^{\otimes 2}$-valued formal series 
\begin{equation}\label{Psie}
\Phi_\lambda:=\sum_{\gamma\in Q_-}\Gamma_\gamma(\lambda)\xi_{\lambda+\gamma}\in
U(\mathfrak{k})^{\otimes 2}[[\xi_{-\alpha_1},\ldots,\xi_{-\alpha_\rr}]]\xi_\lambda
\end{equation}
with coefficients $\Gamma_\gamma(\lambda)\in U(\mathfrak{k})^{\otimes 2}$
and $\Gamma_0(\lambda)=1$, satisfying 
\begin{equation}\label{evPsie}
\widehat{\Pi}(\Omega)\Phi_\lambda=(\lambda,\lambda+2\rho)\Phi_\lambda.
\end{equation}
\end{prop}

In fact, if $\lambda\in\mathfrak{h}_{\textup{HC}}^*$ then the eigenvalue equation \eqref{evPsie}
for a formal series of the form \eqref{Psie}
gives recursion relations for its coefficients $\Gamma_\lambda(\gamma)$
($\gamma\in Q_-$) which, together with the condition $\Gamma_0(\lambda)=1$, determine the coefficients $\Gamma_\gamma(\lambda)$ uniquely. We call the $\Gamma_\gamma(\lambda)\in U(\mathfrak{k})^{\otimes 2}$ ($\gamma\in Q_-$) the Harish-Chandra coefficients.

Let $\mathfrak{n}_+$ be the complexified Lie algebra of $N_+$. The sum $U(\mathfrak{h})+\theta(\mathfrak{n}_+)U(\mathfrak{g})$ in $U(\mathfrak{g})$ is an internal direct sum containing $Z(\mathfrak{g})$.
Denote by $\textup{pr}: Z(\mathfrak{g})\rightarrow U(\mathfrak{\mathfrak{h}})$ the restriction to
$Z(\mathfrak{g})$ of
the projection $U(\mathfrak{h})\oplus\theta(\mathfrak{n}_+)U(\mathfrak{g})\rightarrow U(\mathfrak{h})$ on the first direct summand.
Then $\textup{pr}$ is
an algebra homomorphism 
(see, e.g., \cite[\S 1]{CM}). 
The central character at $\lambda\in\mathfrak{h}^*$ is the algebra homomorphism 
\[
\zeta_\lambda: Z(\mathfrak{g})\rightarrow\mathbb{C},\qquad
z\mapsto \lambda(\textup{pr}(z))
\]
with
$\lambda(\textup{pr}(z))$ the evaluation of $\textup{pr}(z)\in U(\mathfrak{h})\simeq S(\mathfrak{h})$ at $\lambda$. By the second expression of the Casimir element $\Omega$ in \eqref{Omega} we have $\zeta_\lambda(\Omega)=(\lambda,\lambda+2\rho)$. Furthermore, by \cite[Prop. 7.4.7]{Di}, $\zeta_{\lambda-\rho}=\zeta_{\mu-\rho}$ for $\lambda,\mu\in\mathfrak{h}^*$ if and only if $\lambda\in W\mu$. 

\begin{prop}\label{Casimirs}
Let $\lambda\in\mathfrak{h}_{\textup{HC}}^*$. Then
\[
\widehat{\Pi}(z)\Phi_\lambda=\zeta_\lambda(z)\Phi_\lambda\qquad \forall\, z\in Z(\mathfrak{g})
\]
in $U(\mathfrak{k})^{\otimes 2}[[\xi_{-\alpha_1},\ldots,\xi_{-\alpha_\rr}]]\xi_\lambda$.
\end{prop}
\begin{proof}
Write $x^\eta:=x_1^{\eta_1}\cdots x_\rr^{\eta_\rr}\in S(\mathfrak{h})$
and $\partial^\eta:=\partial_{x_1}^{\eta_1}\cdots\partial_{x_\rr}^{\eta_\rr}\in\mathbb{D}_{\mathcal{R}}$ for
$\eta\in\mathbb{Z}_{\geq 0}^\rr$.
The leading symbol of 
$D=\sum_{\eta\in\mathbb{Z}_{\geq 0}^\rr}\bigl(\sum_{\gamma\in Q_-}
c_{\eta,\gamma}\xi_\gamma\bigr)\partial^\eta\in\mathbb{D}_{\mathcal{R}}\otimes U(\mathfrak{k})^{\otimes 2}$ 
is defined to be
\[
\mathfrak{s}_\infty(D):=\sum_{\eta\in\mathbb{Z}_{\geq 0}^\rr}c_{\eta,0}x^\eta\in S(\mathfrak{h})
\otimes U(\mathfrak{k})^{\otimes 2}.
\]
Fix $z\in Z(\mathfrak{g})$. Let 
$z_\lambda^\infty\in U(\mathfrak{k})^{\otimes 2}$ be the evaluation of the leading symbol
$\mathfrak{s}_\infty(\widehat{\Pi}(z))$ at $\lambda$. Note that the $\xi_\lambda$-component of the formal power series $\widehat{\Pi}(z)\Phi_\lambda$ 
is $z_\lambda^{\infty}$. Furthermore, $\widehat{\Pi}(z)\Phi_\lambda\in
U(\mathfrak{k})^{\otimes 2}[[\xi_{-\alpha_1},\ldots,\xi_{-\alpha_\rr}]]\xi_\lambda$ is an
eigenfunction of $\widehat{\Pi}(\Omega)$ with eigenvalue $(\lambda,\lambda+2\rho)$ by
Theorem \ref{thmRAD} {\bf b}.

For any $y\in U(\mathfrak{k})^{\otimes 2}$, 
the formal power series 
\[
\Phi_\lambda y:=\sum_{\gamma\in Q_-}(\Gamma_\gamma(\lambda)y)\xi_{\lambda+\gamma}
\]
is the unique eigenfunction of $\widehat{\Pi}(\Omega)$ of the form
$\sum_{\gamma\in Q_-}\widetilde{\Gamma}_\gamma(\lambda)\xi_{\lambda+\gamma}$ ($\widetilde{\Gamma}_\gamma(\lambda)\in U(\mathfrak{k})^{\otimes 2}$)
with eigenvalue $(\lambda,\lambda+2\rho)$ and 
leading coefficient $\widetilde{\Gamma}_0(\lambda)$ equal to $y$ (cf. Proposition \ref{cordefHC}).
It thus follows that
\[
\widehat{\Pi}(z)\Phi_\lambda=\Phi_\lambda z_\lambda^\infty.
\]
By \cite[Prop. 2.6(ii)]{CM} we have 
\[
\mathfrak{s}_\infty\bigl(\widehat{\Pi}(z)\bigr)=\mathfrak{s}_\infty\bigl(\widehat{\Pi}(\textup{pr}(z))\bigr),
\]
hence $z_\lambda^\infty=\lambda\bigl(\textup{pr}(z)\bigr)1_{U(\mathfrak{k})^{\otimes 2}}=
\zeta_\lambda(z)1_{U(\mathfrak{k})^{\otimes 2}}$. This concludes the proof of the proposition.
\end{proof}
\begin{rema}\label{Minvariance2}
By Remark \ref{Minvariance} and by an argument similar to the proof of Proposition \ref{Casimirs}, it follows that $\Gamma_\gamma(\lambda)\in U(\mathfrak{k}\oplus\mathfrak{k})^M$ for $\lambda\in\mathfrak{h}_{\textup{HC}}^*$ and $\gamma\in Q_-$.
\end{rema}

Fix a finite dimensional representation $\sigma: U(\mathfrak{k})^{\otimes 2}\rightarrow\textup{End}(V_\sigma)$.
For $\lambda\in\mathfrak{h}_{\textup{HC}}^*$ set
\begin{equation}\label{Ftau}
\Phi_\lambda^\sigma:=\sum_{\gamma\in Q_-}\sigma(\Gamma_\gamma(\lambda))\xi_{\lambda+\gamma}\in \textup{End}(V_\sigma)[[\xi_{-\alpha_1},\ldots,\xi_{-\alpha_\rr}]]\xi_\lambda.
\end{equation}
We call $\Phi_\lambda^\sigma$ the $\sigma$-Harish-Chandra series,
and 
\[
\Gamma_\gamma^\sigma(\lambda):=\sigma(\Gamma_\gamma(\lambda))\in
\textup{End}(V_\sigma),\qquad \gamma\in Q_-
\] 
the associated Harish-Chandra coefficients. 
\begin{rema}
Suppose that $V_\sigma$ integrates to a $K\times K$-representation.
Let $\textup{End}_M(V_\sigma)$ be the space of $M$-intertwiners $V_\sigma\rightarrow V_\sigma$ with respect to the diagonal action of $M$ on $V_\sigma$. Then $\Gamma_\gamma^\sigma(\lambda)\in\textup{End}_M(V_\sigma)$ for all $\gamma\in Q_-$ by Remark \ref{Minvariance2}.
\end{rema}
Note that $\Phi_\lambda^\sigma$ is the unique formal power series $\sum_{\gamma\in Q_-}\Gamma_\gamma^\sigma(\lambda)\xi_{\lambda+\gamma}$ with $\Gamma_\gamma^\sigma(\lambda)\in\textup{End}(V_\sigma)$ and $\Gamma_\gamma^\sigma(\lambda)=\textup{id}_{V_\sigma}$ 
satisfying $\widehat{\Pi}(\Omega)\Phi_\lambda^\sigma=(\lambda,\lambda+2\rho)\Phi_\lambda^\sigma$.
The $\sigma$-Harish-Chandra series in addition satisfies the eigenvalue equations $\widehat{\Pi}(z)\Phi_\lambda^\sigma=\zeta_\lambda(z)\Phi_\lambda^\sigma$ for all 
$z\in Z(\mathfrak{g})$. 

Endow $\textup{End}(V_\sigma)$ with the norm topology. The recursion relations arising from the 
eigenvalue equation $\widehat{\Pi}(\Omega)\Phi_\lambda^\sigma=\zeta_\lambda(\Omega)\Phi_\lambda^\sigma$
imply growth estimates for the Harish-Chandra coefficients $\Gamma_\gamma^\sigma(\lambda)$. It leads to the following result
(cf. \cite{W} and references therein).
\begin{prop}\label{holomorphic}
Let $\lambda\in\mathfrak{h}_{\textup{HC}}^*$. Then
\[
\Phi_\lambda^\sigma(a):=\sum_{\gamma\in Q_-}\Gamma_\gamma^\sigma(\lambda)a^{\lambda+\gamma},
\qquad a\in A_+
\]
defines an $\textup{End}(V_\sigma)$-valued analytic function on $A_+$. 
\end{prop}
\begin{rema}\label{remaHCspherplus} 
Set $G_{\textup{reg}}:=KA_+K\subset G$, 
which is an open dense subset of $G$. For $\lambda\in\mathfrak{h}_{\textup{HC}}^*$ and $v\in V_\sigma^M$ the function
\[
H_{\lambda}^{v}(k_1ak_2^{-1}):=\sigma(k_1,k_2)\Phi_\lambda^\sigma(a)v\qquad
(a\in A_+,\,\, k_1,k_2\in K)
\]
is a well defined smooth $V_\sigma$-valued function on $G_{\textup{reg}}$ satisfying 
\[
H_\lambda^v(k_1gk_2^{-1})=\sigma(k_1,k_2)H_\lambda^v(g)\qquad \forall\, g\in G_{\textup{reg}},\,\, \forall\, k_1,k_2\in K.
\]
 It in general does not extend to a $\sigma$-spherical function on $G$.
  \end{rema}

The Harish-Chandra series immediately provide "asymptotically free" common eigenfunctions for the quantum Hamiltonians $H_z^\sigma$ ($z\in Z(\mathfrak{g})$) of the quantum $\sigma$-spin hyperbolic Calogero-Moser system.

\begin{thm}\label{normalizedHCseriesthm}
Fix $\lambda\in\mathfrak{h}_{\textup{HC}}^*+\rho$. The $\textup{End}(V_\sigma)$-valued analytic
function 
\begin{equation}\label{normalizedHCseries}
\mathbf{\Phi}_\lambda^\sigma(a):=\delta(a)\Phi_{\lambda-\rho}^\sigma(a),\qquad a\in A_+
\end{equation}
has a series expansion of the form
\[
\mathbf{\Phi}_\lambda^\sigma=\sum_{\gamma\in Q_-}
\mathbf{\Gamma}_\gamma^\sigma(\lambda)\xi_{\lambda+\gamma}
\in\textup{End}(V_\sigma)[[\xi_{-\alpha_1},\ldots,\xi_{-\alpha_\rr}]]\xi_{\lambda}
\]
with $\mathbf{\Gamma}_\gamma^\sigma(\lambda)\in\textup{End}(V_\sigma)$ and $\mathbf{\Gamma}_0^\sigma(\lambda)=\textup{id}_{V_\sigma}$. It satisfies the Schr{\"o}dinger equation
\[
\mathbf{H}^\sigma\bigl(\mathbf{\Phi}_\lambda^\sigma\bigr)=-\frac{(\lambda,\lambda)}{2}
\mathbf{\Phi}_\lambda^\sigma
\]
as well as the eigenvalue equations
\[
H_z^\sigma\mathbf{\Phi}_\lambda^\sigma=\zeta_{\lambda-\rho}(z)\mathbf{\Phi}_\lambda^\sigma\qquad \forall\, z\in
Z(\mathfrak{g})
\]
as $\textup{End}(V_\sigma)$-analytic functions on $A_+$.
\end{thm}
\begin{proof}
This is an immediate consequence of Proposition \ref{Casimirs} and the 
definitions of the differential 
operators $\mathbf{H}^\sigma$ and $H_z^\sigma$ ($z\in Z(\mathfrak{g})$).
\end{proof}

\section{Principal series representations}\label{SectionRepTh}

We keep the conventions of the previous section. In particular, $(\mathfrak{g}_0,\mathfrak{h}_0,\theta_0)$
is a triple with $\mathfrak{g}_0$ a split real semisimple Lie algebra, $\mathfrak{h}_0$ a split
Cartan subalgebra and $\theta_0$ a Cartan involution such that $\theta_0|_{\mathfrak{h}_0}=
-\textup{id}_{\mathfrak{h}_0}$, and $(G,K)$ is the associated non-compact split symmetric pair.
We fix throughout this section two finite dimensional $K$-representations $\sigma_\ell: K\rightarrow\textup{GL}(V_\ell)$ and $\sigma_r: K\rightarrow\textup{GL}(V_r)$. We write $(\cdot,\cdot)_{V_\ell}$ and $(\cdot,\cdot)_{V_r}$ for scalar products on $V_\ell$ and $V_r$ turning $\sigma_\ell$ and $\sigma_r$ into unitary representations of $K$. 
We view $\textup{Hom}(V_r,V_\ell)$ as finite dimensional $K\times K$-representation with representation map
$\sigma: K\times K\rightarrow\textup{GL}(\textup{Hom}(V_r,V_\ell))$ given by
\begin{equation}\label{tauaction}
\sigma(k_\ell,k_r)T:=\sigma_\ell(k_\ell)T\sigma_r(k_r^{-1})
\end{equation}
for $k_\ell,k_r\in K$ and $T\in \textup{Hom}(V_r,V_\ell)$. It is isomorphic to the tensor product representation $V_\ell\otimes V_r^*$. 
For details on the first two subsections, see \cite[Chpt. 8]{Kn0}.

\subsection{Admissible representations and associated spherical functions}\label{ssss1}

Let $K^\wedge$ be the equivalence classes of the irreducible unitary representations of $K$. Recall that a representation $\pi: G\rightarrow \textup{GL}(\mathcal{H})$ of $G$ on a Hilbert space $\mathcal{H}$ is called {\it admissible} if the restriction $\pi|_K$ of $\pi$ to $K$ is unitary and if the $\tau$-isotypical component $\mathcal{H}(\tau)$ of $\pi|_K$ is finite dimensional for all $\tau\in K^\wedge$. 

Let $\pi: G\rightarrow\textup{GL}(\mathcal{H})$ be an admissible representation. Recall that a vector $v\in\mathcal{H}$ is called {\it smooth} if $g\mapsto \pi(g)v$ defines a smooth map $G\rightarrow\mathcal{H}$. The subspace $\mathcal{H}^\infty\subseteq\mathcal{H}$ of smooth vectors is $G$-stable and dense.
Differentiating the $G$-action on $\mathcal{H}^\infty$ turns $\mathcal{H}^\infty$
into a left $U(\mathfrak{g})$-module. We write $x\mapsto x_{\mathcal{H}^{\infty}}$ for the corresponding action of $x\in U(\mathfrak{g})$. 

The algebraic direct sum
\[
\mathcal{H}^{K-\textup{fin}}:=\bigoplus_{\tau\in K^\wedge}\mathcal{H}(\tau)\subseteq\mathcal{H}
\]
is the dense subspace of $K$-finite vectors in $\mathcal{H}$. It is contained 
in $\mathcal{H}^\infty$ since $\pi$ is admissible, and it inherits a $(\mathfrak{g},K)$-module structure from $\mathcal{H}^\infty$. The $(\mathfrak{g},K)$-module
$\mathcal{H}^{K-\textup{fin}}$ is called the Harish-Chandra module of $\mathcal{H}$.

For $\phi_\ell\in\textup{Hom}_K(\mathcal{H},V_\ell)$ and $\phi_r\in\textup{Hom}_K(V_r,\mathcal{H})$
we now obtain $\sigma$-spherical functions
\[
f^{\phi_\ell,\phi_r}_{\mathcal{H}}\in C^\infty_\sigma(G)
\]
by
\[
f^{\phi_\ell,\phi_r}_{\mathcal{H}}(g):=\phi_\ell\circ\pi(g)\circ\phi_r,\qquad g\in G.
\]
The $\sigma$-spherical functions $f^{\phi_\ell,\phi_r}_{\mathcal{H}}$ are actually $\textup{Hom}(V_r,V_\ell)$-valued real analytic functions on $G$, see, e.g., \cite[Thm. 8.7]{Kn0}. Furthermore,
$f_{\mathcal{H}}^{\phi_\ell,\phi_r}|_{A}$ takes values in $\textup{Hom}_M(V_r,V_\ell)$.

Since $V_\ell$ and $V_r$ are finite dimensional, we have canonical isomorphisms
\begin{equation}\label{algvsanal}
\begin{split}
\textup{Hom}_K(\mathcal{H},V_\ell)&\simeq\textup{Hom}_{\mathfrak{k}}(\mathcal{H}^\infty,V_\ell)
\simeq \textup{Hom}_{\mathfrak{k}}(\mathcal{H}^{K-\textup{fin}},V_\ell),\\
\textup{Hom}_K(V_r,\mathcal{H})&\simeq\textup{Hom}_{\mathfrak{k}}(V_r,\mathcal{H}^\infty)\simeq \textup{Hom}_{\mathfrak{k}}(V_r,\mathcal{H}^{K-\textup{fin}}).
\end{split}
\end{equation}

The $\sigma$-spherical function $f^{\phi_\ell,\phi_r}_{\mathcal{H}}$ can be expressed in terms of matrix coefficients of $\pi$ as follows.
Let $\{v_i\}_i$ and $\{w_j\}_j$ be linear bases of $V_\ell$ and $V_r$, respectively. Expand
$\phi_\ell\in\textup{Hom}_K(\mathcal{H},V_\ell)$ and $\phi_r\in\textup{Hom}_K(V_r,\mathcal{H})$ as
\[
\phi_\ell=\sum_i\langle \cdot,f_i\rangle_{\mathcal{H}}v_i,\qquad
\phi_r=\sum_j(\cdot,w_j)_{V_r}h_j
\]
with $f_i,h_j\in\mathcal{H}^{K-\textup{fin}}$, where $\langle\cdot,\cdot\rangle_{\mathcal{H}}$ is the scalar product of $\mathcal{H}$. The fact that $\phi_\ell$ and $\phi_r$ are $K$-intertwiners implies that $\sum_if_i\otimes v_i$ and
$\sum_jw_j\otimes h_j$ are $K$-fixed in $\mathcal{H}\otimes V_\ell$ and $V_r\otimes \mathcal{H}$,
respectively. The $\sigma$-spherical function $f^{\phi_\ell,\phi_r}_{\mathcal{H}}\in C^\infty_\sigma(G)$ is then given by
\begin{equation}\label{sphermatrix}
f^{\phi_\ell,\phi_r}_{\mathcal{H}}(g)=\sum_{i,j}\langle\pi(g)h_j,f_i\rangle_{\mathcal{H}}(\cdot,w_j)_{V_r}v_i.
\end{equation}
Clearly, for an admissible representation $(\pi,\mathcal{H})$, the subspace of $\sigma$-spherical functions spanned by $f^{\phi_\ell,\phi_r}_{\mathcal{H}}$ ($\phi_\ell\in\textup{Hom}_K(\mathcal{H}, V_\ell)$, $\phi_r\in\textup{Hom}_K(V_r,\mathcal{H})$),
is finite dimensional.

\subsection{Principal series representations and $K$-intertwiners}\label{subsectionPSR}\label{ssss2}
Recall that $M:=Z_K(\mathfrak{h}_0)\subseteq K$ is a finite group, since $\mathfrak{g}_0$ is split.
Furthermore, if $G$ has a complexification then $M$ is abelian (see \cite[Thm. 7.53]{Kn}).
We fix a finite dimensional irreducible representation $\xi: M\rightarrow \textup{GL}(L_\xi)$. Write $\langle\cdot,\cdot\rangle_{\xi}$ for the scalar product on $L_\xi$ turning it into a unitary representation.
Fix a linear functional $\lambda\in\mathfrak{h}^*$ and extend it to a representation 
$\eta_{\lambda}^{(\xi)}: P\rightarrow\textup{GL}(L_\xi)$ of the minimal parabolic subgroup $P=MAN_+$ of $G$ by
\[
\eta_{\lambda}^{(\xi)}(man):=a^{\lambda}\xi(m)\qquad (m\in M, a\in A, n\in N_+).
\] 
Consider the pre-Hilbert space $U^{(\xi)}_{\lambda}$ consisting of continuous, compactly supported 
functions $f: G\rightarrow L_\xi$ satisfying 
\[
f(gp)=\eta_{\lambda+\rho}^{(\xi)}(p^{-1})f(g)\qquad (g\in G, p\in P)
\]
with scalar product
\[
\langle f_1,f_2\rangle^{(\xi)}_{\lambda}:=\int_K\langle f_1(x), f_2(x)\rangle_\xi dx\qquad
(f_1,f_2\in U^{(\xi)}_{\lambda}).
\]
Consider the action of $G$ on $U^{(\xi)}_{\lambda}$ by $(\pi^{(\xi)}_\lambda(g)f)(g^\prime):=f(g^{-1}g^\prime)$ for $g,g^\prime\in G$ and 
$f\in U^{(\xi)}_{\lambda}$. Its extension to an admissible representation $\pi^{(\xi)}_{\lambda}: G\rightarrow
\textup{GL}(\mathcal{H}^{(\xi)}_{\lambda})$, with $\mathcal{H}^{(\xi)}_{\lambda}$ the Hilbert space completion of $U^{(\xi)}_{\lambda}$, is called the {\it principal series representation} of $G$. 
The representation $\pi_\lambda^{(\xi)}$ is
unitary if $\eta_\lambda^{(\xi)}$ is unitary, i.e., if $\lambda(\mathfrak{h}_0)\subset i\mathbb{R}$.

Analogously, let $\eta_\lambda: AN_+\rightarrow\mathbb{C}^*$
be the one-dimensional representation defined by $\eta_\lambda(an):=a^{\lambda}$ for $a\in A$ and $n\in N_+$, and consider the pre-Hilbert space $U_{\lambda}$ consisting of continuous, compactly supported 
functions $f: G\rightarrow \mathbb{C}$ satisfying 
\[
f(gb)=\eta_{\lambda+\rho}(b^{-1})f(g)\qquad (g\in G, b\in AN_+)
\]
with scalar product
\[
\langle f_1,f_2\rangle_{\lambda}:=\int_Kf_1(x)\overline{f_2(x)}\, dx\qquad
(f_1,f_2\in U_{\lambda}).
\]
Turning $U_\lambda$ into a $G$-representation by $(\pi_{\lambda}(g)f)(g^\prime):=f(g^{-1}g^\prime)$ for $g,g^\prime\in G$ and completing, gives an admissible representation $\pi_{\lambda}: G\rightarrow
\textup{GL}(\mathcal{H}_{\lambda})$.
Note that $\pi_\lambda|_K: K\rightarrow\textup{GL}(\mathcal{H}_\lambda)$ is isomorphic to the left regular representation
of $K$ on $L^2(K)$. In particular, $\textup{dim}(\mathcal{H}_\lambda(\tau))=\textup{deg}(\tau)^2$
for all $\tau\in K^\wedge$, where $\textup{deg}(\tau)$ is the degree of $\tau$. 
Furthermore, $\mathcal{H}_\lambda\simeq\bigoplus_{\xi\in M^\wedge}\bigl(\mathcal{H}_\lambda^{(\xi)}\bigr)^{\oplus\textup{deg}(\xi)}$. 

Define for
$\phi_\ell\in\textup{Hom}_K(\mathcal{H}_\lambda,V_\ell)$ the adjoint map $\phi^*_\ell: V_\ell\rightarrow
\mathcal{H}_\lambda$ by
\[
(\phi_\ell(f),v)_{V_\ell}=\langle f,\phi_\ell^*(v)\rangle_\lambda\qquad \forall\, f\in\mathcal{H}_\lambda,\,\, \forall
v\in V_\ell.
\]
Since $\mathcal{H}_\lambda$ is unitary as $K$-representation for all $\lambda\in\mathfrak{h}^*$,
the assignment $\phi_\ell\mapsto \phi_\ell^*$ defines a conjugate linear isomorphism from
$\textup{Hom}_K(\mathcal{H}_\lambda,V_\ell)$ onto $\textup{Hom}_K(V_\ell,\mathcal{H}_\lambda)$.
 
The $\sigma$-spherical functions $f_{\mathcal{H}_\lambda}^{\phi_\ell,\phi_r}$ obtained from the $G$-representation
$\mathcal{H}_\lambda$
using the $K$-inter\-twi\-ners $\phi_\ell\in\textup{Hom}_K(\mathcal{H}_\lambda,V_{\ell})$ and
$\phi_r\in\textup{Hom}_K(V_{r},\mathcal{H}_\lambda)$ now admit the following explicit description
in terms of the Eisenstein integral.
\begin{prop}\label{relEisPrin}
Fix $\lambda\in\mathfrak{h}^*$. 
\begin{enumerate}
\item[{\bf a.}] The map
$\j_{\lambda,V_r}: \textup{Hom}_K(V_r,\mathcal{H}_\lambda)\rightarrow V_r^*
$, 
\[\j_{\lambda,V_r}(\phi_r)(v):=\phi_r(v)(1)\qquad (v\in V_r),
\]
is a linear isomorphism.
\item[{\bf b.}] For
$\phi_\ell\in\textup{Hom}_K(\mathcal{H}_\lambda,V_\ell)$  
let $\iota_{\lambda,V_\ell}(\phi_\ell)\in V_\ell$ be the unique vector such that
\[
\bigl(v,\iota_{\lambda,V_\ell}(\phi_\ell)\bigr)_{V_\ell}=\phi_\ell^*(v)(1)\qquad \forall\, v\in V_\ell.
\]
The resulting map 
$\iota_{\lambda,V_\ell}: \textup{Hom}_K(\mathcal{H}_\lambda,V_\ell)\rightarrow
V_\ell$ is a linear isomorphism.
\item[{\bf c.}] The assignment $\phi_\ell\otimes\phi_r\mapsto 
T^{\phi_\ell,\phi_r}_\lambda:=
\iota_{\lambda,V_\ell}(\phi_\ell)\otimes\j_{\lambda,V_r}(\phi_r)$ defines
a linear isomorphism
\begin{equation*}
\textup{Hom}_K(\mathcal{H}_\lambda,V_{\ell})\otimes\textup{Hom}_K(V_{r},\mathcal{H}_\lambda)\overset{\sim}{\longrightarrow} V_{\ell}\otimes V_{r}^*\simeq\textup{Hom}(V_{r},V_{\ell}).
\end{equation*}
Furthermore,
\begin{equation}\label{relEis}
f_{\mathcal{H}_\lambda}^{\phi_\ell,\phi_r}(g)=E_\lambda^\sigma(g)T^{\phi_\ell,\phi_r}_\lambda\qquad (g\in G)
\end{equation}
for $\phi_\ell\in\textup{Hom}_K(\mathcal{H}_\lambda,V_{\ell})$ and
$\phi_r\in\textup{Hom}_K(V_{r},\mathcal{H}_\lambda)$, with $E_\lambda^\sigma(g)$
the Eisenstein integral \eqref{Eisenstein}.
\end{enumerate}
\end{prop}
\begin{proof}
Our choice of parametrisation of the $\sigma$-spherical functions associated to $\pi_\lambda$, which
deviates from the standard choice (see, e.g., \cite[\S 8.2]{Kn0}), plays an important in Section
\ref{SectionbKZB} when discussing the applications to asymptotic boundary KZB equations. We provide here a proof directly in terms of our present conventions.

We will assume without loss of generality that $\sigma_\ell, \sigma_r\in K^\wedge$.\\
{\bf a.} Since $\textup{dim}(\textup{Hom}_K(V_r,\mathcal{H}_\lambda))=\textup{deg}(\sigma_r)$
it suffices to show that $\j_{\lambda,V_r}$ is injective. 
Fix an orthonormal basis $\{v_i\}_i$ of $V_{r}$.
Let $\phi_r\in\textup{Hom}_K(V_r,\mathcal{H}_\lambda)$ and consider its expansion $\phi_r=
\sum_j(\cdot,v_j)_{V_r} h_j$
with $h_j\in\mathcal{H}_\lambda^{K-\textup{fin}}$. Then
\begin{equation}\label{jALT}
\j_{\lambda,V_r}(\phi_r)=\sum_jh_j(1)(\cdot,v_j)_{V_r}.
\end{equation}
Furthermore, for each index $j$ we have
\begin{equation}\label{transfoK}
h_j(x)=\sum_ih_i(1)(v_j,\sigma_r(x)v_i)_{V_r}\qquad \forall\, x\in K
\end{equation}
since $\phi_r$ is a $K$-intertwiner. 

Suppose now that $\j_{\lambda,V_r}(\phi_r)=0$. Then $h_j(1)=0$ for all $j$ by \eqref{jALT}. By
\eqref{transfoK} we conclude that $h_j(kan)=a^{-\lambda-\rho}h_j(k)=0$ for $k\in K$, $a\in A$ and $n\in N_+$, so $\phi_r=0$.\\
{\bf b.} This immediately follows from part {\bf a} and the fact that 
\[
(v,\iota_{\lambda,V_\ell}(\phi_\ell))_{V_\ell}=\j_{\lambda,V_\ell}(\phi^*_\ell)(v)
\]
for $v\in V_\ell$ and $\phi_\ell\in\textup{Hom}_K(\mathcal{H}_\lambda,V_\ell)$.\\
{\bf c.} The first statement immediately follows from {\bf a} and {\bf b}.
Let $\{v_i\}_i$ be an orthonormal basis of $V_{\ell}$ and $\{w_j\}_j$
an orthonormal basis of $V_{r}$.
For $\phi_\ell=\sum_i\langle \cdot,f_i\rangle_\lambda v_i\in\textup{Hom}_K(\mathcal{H}_\lambda,V_{\ell})$ and $\phi_r=\sum_j(\cdot,w_j)_{V_r}h_j\in \textup{Hom}_K(V_{r},\mathcal{H}_\lambda)$ with  $f_i,h_j\in\mathcal{H}_\lambda^{K-\textup{fin}}$ a direct computation gives
\[
f_{\mathcal{H}_\lambda}^{\phi_\ell,\phi_r}(g)=E_\lambda^\sigma(g)\widetilde{T}^{\phi_\ell,\phi_r}_\lambda
\]
with $\widetilde{T}^{\phi_\ell,\phi_r}_\lambda\in\textup{Hom}(V_r,V_\ell)$ given by
\[
\widetilde{T}_\lambda^{\phi_\ell,\phi_r}(w)=\sum_{i,j}h_j(1)\overline{f_i(1)}(w,w_j)_{V_r}v_i\qquad (w\in V_{r}).
\]
By \eqref{jALT} this can be rewritten as
\[
\widetilde{T}_\lambda^{\phi_\ell,\phi_r}(w)=\j_{\lambda,V_r}(\phi_r)(w)\sum_i\overline{f_i(1)}v_i
\qquad w\in V_r,
\]
hence it suffices to show that
\begin{equation}\label{iotaALT}
\iota_{\lambda,V_\ell}(\phi_\ell)=\sum_i\overline{f_i(1)}v_i.
\end{equation}
Define $\widetilde{\chi}_{\lambda,V_\ell}\in \mathcal{H}_\lambda^{K-\textup{fin}}$ by 
\[
\widetilde{\chi}_{\lambda,V_\ell}(kan):=\textup{deg}(\sigma_\ell)a^{-\lambda-\rho}\chi_{V_\ell}(k)\qquad
(k\in K, a\in A, n\in N_+),
\]
with $\chi_{V_\ell}$ the character of $V_\ell$. Fix $v\in V_{\ell}$. Since
$\phi^*_\ell(v)\in\mathcal{H}_\lambda(\sigma_\ell)$, its restriction $\phi_\ell^*(v)|_K$ to $K$ lies in the $\sigma_\ell$-isotypical component of $L^2(K)$ with respect to the left-regular $K$-action. By the Schur orthogonality relations we then have
\[
\phi^*_\ell(v)(1)=\textup{deg}(\sigma_\ell)\int_Kdx\,\phi^*_\ell(v)(x)\overline{\chi_{V_\ell}(x)}=
\langle\phi^*_\ell(v),\widetilde{\chi}_{\lambda,V_\ell}\rangle_\lambda=(v,\phi_\ell(\widetilde{\chi}_{\lambda,V_\ell}))_{V_\ell}.
\]
This show that
\[
\iota_{\lambda,V_\ell}(\phi_\ell)=
\phi_\ell(\widetilde{\chi}_{\lambda,V_\ell}).
\]
Now substitute $\phi_\ell=\sum_i\langle \cdot,f_i\rangle_\lambda v_i$ and use that 
$f_i\in\mathcal{H}_\lambda(\sigma^*_\ell)$ with $\sigma^*_\ell$ the irreducible 
$K$-re\-pre\-sen\-ta\-tion dual to $\sigma_\ell$, we get \eqref{iotaALT} by another application of the Schur's 
orthogonality relations,
\begin{equation*}
\iota_{\mu,V_\ell}(\phi_\ell)=\phi_\ell(\widetilde{\chi}_{\lambda,V_\ell})
=\sum_i\textup{deg}(\sigma_\ell)\Bigl(\int_K dx\, \chi_{V_\ell}(x)
\overline{f_i(x)}\Bigr) v_i=
\sum_i\overline{f_i(1)}v_i.
\end{equation*}
\end{proof}
\begin{rema}\label{relEisPrinRem}
{\bf a.} For $a\in A$ and $m\in M$ one has
\[
\sigma(m,m)E_\lambda^\sigma(a)=E_\lambda^\sigma(a).
\]
In particular, $E_\lambda^\sigma(a)$ maps $\textup{Hom}(V_r,V_\ell)$ into $\textup{Hom}_M(V_r,V_\ell)$.\\
{\bf b.} 
For $\xi\in M^\wedge$ and intertwiners $\phi_\ell\in\textup{Hom}_K(\mathcal{H}^{(\xi)}_{\lambda},V_{\ell})$ and $\phi_r\in\textup{Hom}_K(V_{r},\mathcal{H}^{(\xi)}_{\lambda})$, write
\[
\phi_\ell=\sum_i\langle \cdot,f_i\rangle^{(\xi)}_\lambda v_i,\qquad
\phi_r=\sum_j(\cdot,w_j)_{V_r}h_j
\]
with $f_i,h_j\in \mathcal{H}^{(\xi),K-\textup{fin}}_{\lambda}$. Then 
\begin{equation}\label{relEisxi}
f^{\phi_\ell,\phi_r}_{\mathcal{H}_{\lambda}^{(\xi)}}(g)=E_\lambda^\sigma(g)T^{(\xi),\phi_\ell,\phi_r}_{\lambda}\qquad (g\in G)
\end{equation}
with $T^{(\xi),\phi_\ell,\phi_r}_{\lambda}\in\textup{Hom}_M(V_{r},V_{\ell})$ the $M$-intertwiner 
\[
T^{(\xi),\phi_\ell,\phi_r}_{\lambda}(w):=\sum_{i,j}\langle h_j(1),f_i(1)\rangle_{\lambda}^{(\xi)}(w,w_j)_{V_r}v_i
\qquad (w\in V_{r}).
\]
\end{rema}

\subsection{Algebraic principal series representations}\label{ssss3}
We first introduce some general facts and notations regarding $\mathfrak{g}$-modules, following \cite{Di}.

Let $V$ be a $\mathfrak{g}$-module with representation map $\tau: \mathfrak{g}\rightarrow
\mathfrak{gl}(V)$. The representation map of $V$, viewed as $U(\mathfrak{g})$-module, will also be denoted by $\tau$. The dual of $V$ is defined by
\[
(\tau^*(x)f)(v)=-f(\tau(X)v),\qquad x\in \mathfrak{g},\,\, f\in V^*,\,\, v\in V.
\]

Fix a Lie subalgebra $\mathfrak{l}\subseteq\mathfrak{g}$ which is reductive in $\mathfrak{g}$. This means that $\mathfrak{g}$, viewed as $\mathfrak{l}$-module via the adjoint action, is semisimple (in particular, $\mathfrak{l}$ is a reductive Lie algebra). In this paper $\mathfrak{l}$ will either be the fix-point Lie subalgebra $\mathfrak{k}$ of the Chevalley involution $\theta$, which is reductive in $\mathfrak{g}$ by \cite[Prop. 1.13.3]{Di}, or it is the Cartan subalgebra $\mathfrak{h}$. Let $\mathfrak{l}^\wedge$ be the isomorphism classes of the finite dimensional irreducible $\mathfrak{l}$-modules. For $\tau\in\mathfrak{l}^\wedge$ we write
$\textup{deg}(\tau)$ for the degree of $\tau$ and 
$V(\tau)$ for the $\tau$-isotypical component of $V$. The isotypical component $V(\tau)$ decomposes in a direct sum of copies of $\tau$. The number of copies, denoted by $\textup{mtp}(\tau,V)\in\mathbb{Z}_{\geq 0}\cup\{\infty\}$, is called the multiplicity of $\tau$ in $V$. Set 
\[
V^{\mathfrak{l}-\textup{HC}}:=\sum_{\tau\in\mathfrak{l}^\wedge}V(\tau),
\]
which is a $\mathfrak{g}$-submodule of $V$ by \cite[Prop. 1.7.9]{Di} (the sum is direct). 
A $\mathfrak{g}$-module $V$ is called a Harish-Chandra module with respect to $\mathfrak{l}$ if $V=V^{\mathfrak{l}-\textup{HC}}$. A Harish-Chandra module $V$ is called admissible if $\textup{mtp}(\tau,V)<\infty$ for all $\tau\in\mathfrak{l}^\wedge$. 

Note that $V^{\mathfrak{l}-\textup{HC}}$ is the sum of all the finite dimensional {\it semisimple} $\mathfrak{l}$-submodules of $V$. It is contained in the $\mathfrak{g}$-submodule
\[
V^{\mathfrak{l}-\textup{fin}}:=\{v\in V \,\,\, | \,\,\, \dim\bigl(U(\mathfrak{l})v\bigr)<\infty \, \}
\]
of $V$, which alternatively can be defined as the sum of all the finite dimensional $\mathfrak{l}$-submodules of $V$. Note that $V^{\mathfrak{l}-\textup{HC}}=V^{\mathfrak{l}-\textup{fin}}$
if and only if the center $Z(\mathfrak{l})$ of $\mathfrak{l}$ acts semisimply on $V^{\mathfrak{l}-\textup{fin}}$. This in particular holds true when $\mathfrak{l}$ is semisimple.

For $\mathfrak{l}$-modules $U,V$ with $U$ or $V$ finite dimensional we identify
\begin{equation}\label{identify}
U\otimes V^*\simeq\textup{Hom}(V,U)
\end{equation}
as vector spaces by $u\otimes f\mapsto f(\cdot)u$. With the 
$U(\mathfrak{l})\otimes U(\mathfrak{l}$)-module structure on $\textup{Hom}(V,U)$ defined by
\[
((x\otimes z)f)(v):=xf(S(z)v)
\]
for $x,z\in U(\mathfrak{l})$, $v\in V$ and $f\in\textup{Hom}(V,U)$,
it is an isomorphism of $U(\mathfrak{l})\otimes U(\mathfrak{l})$-modules.

Differentiating the multiplicative character $\eta_\lambda: AN_+\rightarrow\mathbb{C}^*$ of the previous subsection gives a one-dimensional $\mathfrak{b}$-module, whose representation map we also denote by $\eta_\lambda$. Then $\eta_\lambda: \mathfrak{b}\rightarrow\mathbb{C}$ is concretely given by
\[
\eta_\lambda(h+u):=\lambda(h),\qquad h\in\mathfrak{h},\,\, u\in\mathfrak{n}_+.
\]
We write $\mathbb{C}_\lambda$ for the associated one-dimensional $U(\mathfrak{b})$-module.

\begin{defi}
Let $\lambda\in\mathfrak{h}^*$. Write
\[
Y_\lambda:=\textup{Hom}_{U(\mathfrak{b})}\bigl(U(\mathfrak{g}),\mathbb{C}_{\lambda+\rho}\bigr)
\]
for the space of linear functionals $f: U(\mathfrak{g})\rightarrow \mathbb{C}$ satisfying
$f(xz)=\eta_{\lambda+\rho}(x)f(z)$ for $x\in U(\mathfrak{b})$ and $z\in U(\mathfrak{g})$. We view $Y_\lambda$ as $\mathfrak{g}$-module by 
\[
(yf)(z):=f(zy),\qquad y\in\mathfrak{g},\,\, z\in U(\mathfrak{g}).
\] 
\end{defi}

By \cite[Chpt. 9]{Di}, the Harish-Chandra module $Y_\lambda^{\mathfrak{k}-\textup{HC}}$ is admissible with $\textup{mtp}(\tau, Y_\lambda^{\mathfrak{k}-\textup{HC}})=\textup{deg}(\tau)$
for all $\tau\in\mathfrak{k}^\wedge$. Consider $K^\wedge$ as subset of $\mathfrak{k}^\wedge$. Note that the inclusion $K^\wedge\hookrightarrow\mathfrak{k}^\wedge$ is strict unless $K$ is simply connected and semisimple. 
\begin{prop}\label{relAnalAlg}
For $\lambda\in\mathfrak{h}^*$ we have an injective morphism of $\mathfrak{g}$-modules
\begin{equation}\label{injpsr}
\mathcal{H}_\lambda^{K-\textup{fin}}\hookrightarrow Y_\lambda^{\mathfrak{k}-\textup{HC}},
\qquad f\mapsto \widetilde{f}
\end{equation}
with
\[
\widetilde{f}(z):=\bigl(r_*(S(z))f\bigr)(1)
\]
for $f\in \mathcal{H}_\lambda^{K-\textup{fin}}$ and $z\in U(\mathfrak{g})$. For $\tau\in K^\wedge$ the embedding restricts
to an isomorphism
\begin{equation}\label{HY}
\mathcal{H}_\lambda(\tau)\overset{\sim}{\longrightarrow} Y_\lambda(\tau)
\end{equation}
of $\mathfrak{k}$-modules.
\end{prop}
\begin{proof}
Let $f\in\mathcal{H}_\lambda^{K-\textup{fin}}$. Then $f: G\rightarrow\mathbb{C}$ is analytic and satisfies
$r_*(S(x))f=\eta_{\lambda+\rho}(x)f$ for all $x\in U(\mathfrak{b})$. Hence \eqref{injpsr} is a well defined injective linear map. A direct computation shows that \eqref{injpsr} intertwines the $\mathfrak{g}_0$-action. This proves the first part of the proposition. 

For $\tau\in K^\wedge$ we have $\textup{dim}\bigl(\mathcal{H}_\lambda(\tau)\bigr)=
\textup{deg}(\tau)^2=\textup{dim}\bigl(Y_\lambda(\tau)\bigr)$, hence \eqref{HY} follows from the first part of the proposition.
\end{proof}
\begin{rema}
The embedding \eqref{injpsr} is an isomorphism if $K$ is simply connected and semisimple.
In general, the algebraic description of the $(\mathfrak{g}_0,K)$-modules $\mathcal{H}_\lambda^{K-\textup{fin}}$ and 
$\mathcal{H}^{(\xi),K-\textup{fin}}_\lambda$ within $Y_\lambda^{\mathfrak{k}-\textup{HC}}$ amounts to 
taking the direct sum of isotypical components $Y_\lambda(\tau)$ for $\tau$ running over suitable subsets of $K^\wedge$ (see \cite[\S 9.3]{Di}).
\end{rema}
Let $\phi_\ell\in\textup{Hom}_K(\mathcal{H}_\lambda,V_\ell)$, $\phi_r\in\textup{Hom}_K(V_r,\mathcal{H}_\lambda)$.
The associated $\sigma$-spherical function $f_{\mathcal{H}_\lambda}^{\phi_\ell,\phi_r}\in C_\sigma^\infty(G)$
is an {\it elementary $\sigma$-spherical function}, meaning that it a common eigenfunction of the biinvariant differential operators on $G$. Indeed, by Proposition \ref{relAnalAlg} it suffices to note that $Y_\lambda$ admits a central character. This follows from \cite[Thm. 9.3.3]{Di},
\begin{equation}\label{ccY}
zf=\zeta_{\lambda-\rho}(z)f\qquad (z\in Z(\mathfrak{g}),\,\, f\in Y_\lambda).
\end{equation}
This also follows from the observation that $Y_\lambda$ is isomorphic to $M_{-\lambda-\rho}^*$ (see Lemma \ref{relCC}) and the fact that 
$\zeta_{\mu-\rho}=\zeta_{w\mu-\rho}$ for $w\in W$.

As a consequence, the restriction $f_{\mathcal{H}_\lambda}^{\phi_\ell,\phi_r}|_{A_{\textup{reg}}}$ of $f_{\mathcal{H}_\lambda}^{\phi_\ell,\phi_r}\in
C_\sigma^\infty(G)$ to $A_{\textup{reg}}$
are common eigenfunctions of $\widehat{\Pi}^\sigma(z)$ ($z\in Z(\mathfrak{g})$),
\begin{equation}\label{EisensteinintDE}
\widehat{\Pi}^\sigma(z)\bigl(f_{\mathcal{H}_\lambda}^{\phi_\ell,\phi_r}|_{A_{\textup{reg}}}\bigr)=\zeta_{\lambda-\rho}(z)
f_{\mathcal{H}_\lambda}^{\phi_\ell,\phi_r}|_{A_{\textup{reg}}}\qquad \forall z\in Z(\mathfrak{g})
\end{equation}
(it is sometimes more natural to write the eigenvalue as $\zeta_{w_0(\lambda+\rho)}(z)$ with
$w_0\in W$ the longest Weyl group element).
By Proposition \ref{relEisPrin} it follows that the restriction
$E_\lambda^\sigma|_{A_{\textup{reg}}}$ of the Eisenstein integral to $A_{\textup{reg}}$ is an $\textup{End}(\textup{Hom}(V_r,V_\ell))$-valued smooth function on $A_{\textup{reg}}$ satisfying the differential equations
\begin{equation}\label{EisensteinDE}
\widehat{\Pi}^\sigma(z)\bigl(E_\lambda^\sigma|_{A_{\textup{reg}}}\bigr)=\zeta_{\lambda-\rho}(z)
E_\lambda^\sigma|_{A_{\textup{reg}}}\qquad \forall z\in Z(\mathfrak{g}).
\end{equation}

\begin{cor}
The normalised smooth $\textup{End}(\textup{Hom}(V_r,V_\ell))$-valued function on $A_+$ defined by
\begin{equation*}
\begin{split}
\mathbf{E}_\lambda^\sigma(a^\prime):=\delta(a^\prime)E_\lambda^\sigma(a^\prime)=
\delta(a^\prime)\int_Kdx\,\xi_{-\lambda-\rho}(a(a^\prime{}^{-1}x))\sigma(x,k(a^\prime{}^{-1}x))
\end{split}
\end{equation*}
for $a^\prime\in A_+$ 
is a common $\textup{End}(\textup{Hom}(V_r,V_\ell))$-valued eigenfunction for the quantum Hamiltonians of the quantum $\sigma$-spin hyperbolic Calogero-Moser system,
\[
\mathbf{H}^\sigma\bigl(\mathbf{E}_\lambda^{\sigma}\bigr)=-\frac{(\lambda,\lambda)}{2}
\mathbf{E}_\lambda^{\sigma}
\]
and 
\[
H_z^\sigma\bigl(\mathbf{E}_\lambda^{\sigma}\bigr)=\zeta_{\lambda-\rho}(z)
\mathbf{E}_\lambda^{\sigma}\qquad (z\in Z(\mathfrak{g})).
\]
\end{cor}
\begin{rema}
For sufficiently generic $\lambda\in\mathfrak{h}^*$, the $\sigma$-Harish-Chandra series $\Phi_{w\lambda-\rho}^\sigma$ \textup{(}$w\in W$\textup{)} exist and satisfy the same differential equations 
\eqref{EisensteinDE} on $A_+$
as $E_\lambda^\sigma|_{A_+}$. Harish-Chandra's \cite{HCe} 
proved for generic $\lambda\in\mathfrak{h}^*$,
\begin{equation}\label{cfunctionexpansion}
E_\lambda^\sigma(a)T=\sum_{w\in W}c^\sigma(w;\lambda)\Phi_{w\lambda-\rho}^\sigma(a)T
\end{equation}
for $a\in A_+$ and $T\in\textup{Hom}_M(V_r,V_\ell)$,
with leading coefficients $c^\sigma(w;\lambda)\in\textup{End}(\textup{Hom}_M(V_r,V_\ell))$ called $c$-functions (see \cite[Thm. 5]{HCe}). The $c$-function expansion \eqref{cfunctionexpansion}
plays an important role in the harmonic analysis on $G$. 
\end{rema}

For the left hand side of \eqref{cfunctionexpansion}, Remark \ref{relEisPrinRem}\,{\bf b}
provides a representation theoretic interpretation in terms of the principal series representation of $G$. In the next section we obtain a similar representation theoretic interpretation for the $\sigma$-Harish-Chandra series
$\Phi_{w\lambda-\rho}^\sigma$ in terms of Verma modules.
\section{Formal elementary $\sigma$-spherical functions}\label{RTHC}
We fix in this section two finite dimensional semisimple representations
$\sigma_\ell: \mathfrak{k}\rightarrow\mathfrak{gl}(V_\ell)$ and $\sigma_r: \mathfrak{k}\rightarrow
\mathfrak{gl}(V_r)$. We write $\sigma$ for the $\mathfrak{k}\oplus\mathfrak{k}$-representation map $\sigma_\ell\otimes\sigma_r^*$ of the resulting semisimple $\mathfrak{k}\oplus\mathfrak{k}$-module $V_\ell\otimes V_r^*$. 

\subsection{Verma modules}

In this subsection we relate the algebraic principal series representations to Verma modules. 
Let $V$ be a $\mathfrak{g}$-module $V$. Write
\[
V[\mu]:=\{v\in V \,\, | \,\, hv=\mu(h)v\quad \forall\, h\in\mathfrak{h}\}
\]
for the weight space of $V$ of weight $\mu\in\mathfrak{h}^*$.
Then
\[
\overline{V}:=\prod_{\mu\in\mathfrak{h}^*}V[\mu]
\]
inherits from $V$ the structure of a 
$\mathfrak{g}$-module as follows. Let $v=(v[\mu])_{\mu\in\mathfrak{h}^*}\in \overline{V}$ 
and $z_\alpha\in\mathfrak{g}_\alpha$ ($\alpha\in R\cup\{0\}$), where $\mathfrak{g}_0:=\mathfrak{h}$.
Then $z_\alpha v=((z_\alpha v)[\mu])_{\mu\in\mathfrak{h}^*}$
with 
\[
(z_\alpha v)[\mu]:=z_\alpha v[\mu-\alpha].
\]
Clearly $V\subseteq \overline{V}$ as $\mathfrak{g}$-submodule. Note that $\overline{V}^{\mathfrak{h}-\textup{HC}}=V$ for $\mathfrak{h}$-semisimple $\mathfrak{g}$-modules $V$. 

For $\mu\in\mathfrak{h}^*$ write 
\begin{equation}\label{lambdaiso}
\textup{proj}^\mu_V: \overline{V}\twoheadrightarrow V[\mu],
\qquad v\mapsto v[\mu]
\end{equation}
for the canonical projection, 
and $\textup{incl}^\mu_V: V[\mu]\hookrightarrow V$ for the inclusion map. We omit the sublabel 
$V$ from the notations $\textup{proj}^\mu_V$ and $\textup{incl}^\mu_V$ if the representation $V$ is clear from the context.

\begin{defi}
The Verma module $M_\lambda$ with highest weight $\lambda\in\mathfrak{h}^*$ is the induced $\mathfrak{g}$-module
\[
M_\lambda:=U(\mathfrak{g})\otimes_{U(\mathfrak{b})}\mathbb{C}_\lambda.
\]
\end{defi}
The Verma module $M_\lambda$ and its irreducible quotient $L_\lambda$ are highest weight modules of highest weight $\lambda$.
In particular, they are $\mathfrak{h}$-diagonalizable with finite dimensional weight spaces. The weight decompositions are $M_{\lambda}=\bigoplus_{\mu\leq\lambda}M_{\lambda}[\mu]$ and $L_{\lambda}=\bigoplus_{\mu\leq\lambda}L_{\lambda}[\mu]$ with $\leq$ the dominance order on $\mathfrak{h}^*$ with respect to $R^+$ and with one-dimensional highest weight spaces
$M_\lambda[\lambda]$ and $L_\lambda[\lambda]$. We fix once and for all a highest weight vector $0\not=m_\lambda\in M_\lambda[\lambda]$, and write $0\not=\ell_\lambda\in L_\lambda[\lambda]$ for its projection onto $L_\lambda$. Note that $M_\lambda$ and $L_\lambda$ admit the central character $\zeta_\lambda$. 

The set $\mathfrak{h}^*_{\textup{irr}}$ of highest weights $\lambda$ for which $M_\lambda$ is irreducible is given by
\[
\mathfrak{h}_{\textup{irr}}^*=\{\lambda\in\mathfrak{h}^* \,\,\, | \,\,\,
(\lambda+\rho,\alpha^\vee)\not\in\mathbb{Z}_{>0} \quad \forall\, \alpha\in R^+\},
\]
with $\alpha^\vee:=2\alpha/\|\alpha\|^2$ the co-root of $\alpha$.
Note that $\mathfrak{h}_{\textup{HC}}^*\subseteq \mathfrak{h}_{\textup{irr}}^*$.

For a $\mathfrak{g}$-module $V$ write ${}^\theta V$ for $V$ endowed with the $\theta$-twisted
$\mathfrak{g}$-module structure  
\[
x\ast v:=\theta(x)v,\qquad x\in \mathfrak{g},\,\, v\in V.
\]

\begin{lem}\label{relCC}
Let $\lambda\in\mathfrak{h}^*$. 
\begin{enumerate}
\item[{\bf a.}] We have 
\[
M_{\lambda}^*\overset{\sim}{\longrightarrow}Y_{-\lambda-\rho}
\]
as $\mathfrak{g}$-modules, with the isomorphism $f\mapsto\widehat{f}$ given by
$\widehat{f}(x):=f(S(x)m_{\lambda})$
for $f\in M_{\lambda}^*$ and $x\in U(\mathfrak{g})$. 
\item[{\bf b.}] If $\lambda\in\mathfrak{h}_{\textup{irr}}^*$, then
\[
 \overline{M}_\lambda\simeq {}^\theta Y_{-\lambda-\rho}
\]
as $\mathfrak{g}$-modules. In particular, $M_\lambda\simeq {}^\theta Y_{-\lambda-\rho}^{\mathfrak{h}-\textup{HC}}$ as $\mathfrak{g}$-modules.
\end{enumerate}
\end{lem}
\begin{proof}
{\bf a.} This is immediate (it is a special case of \cite[Prop. 5.5.4]{Di}).\\
{\bf b.}
The Shapovalov form is the nondegenerate symmetric bilinear form on $L_\lambda$
satisfying 
\[
B_\lambda(xu,v)=-B_\lambda(u,\theta(x)v)
\]
for $x\in\mathfrak{g}$ and
$u,v\in L_\lambda$ and normalised by $B_\lambda(\ell_\lambda,\ell_\lambda)=1$. It induces an isomorphism of $\mathfrak{g}$-modules
\begin{equation}\label{Shapovaloviso}
{}^\theta \overline{L}_\lambda\overset{\sim}{\longrightarrow}
L_\lambda^*
\end{equation}
mapping $(v[\mu])_{\mu\in\mathfrak{h}^*}\in \overline{L}_\lambda$ to the linear functional 
$u\mapsto\sum_{\mu\in\mathfrak{h}^*}B_\lambda(v[\mu],u)$ on $L_\lambda$.
If $\lambda\in\mathfrak{h}_{\textup{irr}}^*$ then $M_\lambda=L_\lambda$ and the result
follows part {\bf a} of the lemma.
\end{proof}
\begin{rema}\label{AA}
{\bf a.} The dual $M^\vee$ of a module $M$ in category $\mathcal{O}$ is defined by 
$M^\vee:={}^\theta M^{*,\mathfrak{h}-\textup{HC}}$. The final conclusion of
part {\bf b} of the lemma corresponds to the well known fact that $L_\lambda^\vee\simeq L_\lambda$.
\\
{\bf b.}
Combining Proposition \ref{relAnalAlg} and Lemma \ref{relCC}\,{\bf a}, we have $\lambda\in\mathfrak{h}^*$ an embedding of
$\mathfrak{g}$-modules 
\[
\mathcal{H}_{\lambda}^{K-\textup{fin}}\hookrightarrow M_{-\lambda-\rho}^*,\qquad
f\mapsto \breve{f},
\]
with $\breve{f}(xm_{-\lambda-\rho}):=(r_\ast(x)f)(1)$ for all $x\in U(\mathfrak{g})$. It restricts to 
an isomorphism $\mathcal{H}_\lambda(\tau)\overset{\sim}{\longrightarrow} M_{-\lambda-\rho}^*(\tau)$ for each $\tau\in K^\wedge$.
\end{rema}
\subsection{Spaces of $\mathfrak{k}$-intertwiners} 
In Subsection \ref{subsectionPSR} we have constructed linear isomorphisms $\iota_{\lambda,V_\ell}: \textup{Hom}_K(\mathcal{H}_\lambda,V_\ell) \overset{\sim}{\longrightarrow} V_\ell$ and $\j_{\lambda,V_r}:
\textup{Hom}_K(V_r,\mathcal{H}_\lambda)\overset{\sim}{\longrightarrow} V_r^*$,
with $\mathcal{H}_\lambda$ the principal series representation. For Verma modules we have the following analogous result. Write $m_\lambda^*$ for the linear functional on $\overline{M}_\lambda$
that vanishes on $\prod_{\mu<\lambda}M_\lambda[\mu]$ and maps $m_\lambda$ to $1$.

\begin{prop}\label{relEisPrinalg}
Fix $\lambda\in\mathfrak{h}^*$.
\begin{enumerate}
\item[{\bf a.}] The map
$\textup{ev}_{\lambda,V_\ell}: \textup{Hom}_{\mathfrak{k}}(M_\lambda,V_\ell)\rightarrow V_\ell$,
\[
\textup{ev}_{\lambda,V_\ell}(\phi_\ell):=
\phi_\ell(m_\lambda),
\]
is a linear isomorphism.
\item[{\bf b.}] 
The linear map
$\textup{hw}_{\lambda,V_r}: \textup{Hom}_{\mathfrak{k}}(V_r,\overline{M}_\lambda)\rightarrow V_r^*$, 
defined by 
\[\textup{hw}_{\lambda,V_r}(\phi_r)(v):=m_\lambda^*(\phi_r(v)) \qquad (v\in V_r),
\]
is a linear isomorphism when $\lambda\in\mathfrak{h}_{\textup{irr}}^*$.
\end{enumerate}
\end{prop}
\begin{proof}
We assume without loss of generality that $\sigma_\ell,\sigma_r\in\mathfrak{k}^\wedge$.\\
{\bf a.} By Lemma \ref{relCC}\,{\bf a} we have 
\[
\textup{Hom}_{\mathfrak{k}}(M_\lambda,V_\ell)\simeq
\textup{Hom}_{\mathfrak{k}}(V_{\ell}^*, Y_{-\lambda-\rho})
\]
as vector spaces. The latter space is of dimension $\textup{deg}(\sigma_\ell)$. Hence it suffices
to show that $\textup{ev}_{\lambda,V_\ell}$ is injective.
This follows from $M_\lambda=U(\mathfrak{k})m_\lambda$, which is an immediate consequence of the Iwasawa decomposition $\mathfrak{g}_0=\mathfrak{k}_0\oplus\mathfrak{h}_0\oplus \mathfrak{n}_{0,+}$ of $\mathfrak{g}_0$.\\
{\bf b.} Since $\lambda\in\mathfrak{h}_{\textup{irr}}^*$ we have
$\textup{Hom}_{\mathfrak{k}}(V_r,\overline{M}_\lambda)\simeq
\textup{Hom}_{\mathfrak{k}}(V_r,Y_{-\lambda-\rho})$ as vector spaces by
Lemma \ref{relCC}, and the latter space is of dimension $\textup{deg}(\sigma_r)$.
It thus suffices to show that $\textup{hw}_{\lambda,V_r}$ is injective. 
Let $\phi_r\in\textup{Hom}_{\mathfrak{k}}(V_r,\overline{M}_\lambda)$ be a nonzero intertwiner and consider the nonempty set 
\[
\mathcal{P}:=\{\mu\in\mathfrak{h}^* \,\, | \,\, \textup{proj}_{M_\lambda}^\mu(\phi_r(v))\not=0 \,\,\, \textup{ for some }\, v\in V_r \}.
\]
Take a maximal element $\nu\in\mathcal{P}$ with respect to the dominance order $\leq$ on $\mathfrak{h}^*$. 
Fix $v\in V_r$ with $\textup{proj}_{M_\lambda}^\nu(\phi_r(v))\not=0$. Suppose that $e_\alpha(\phi_r(v)[\nu])\not=0$ in $M_\lambda$ for some $\alpha\in R^+$. Then 
$\textup{proj}_{M_\lambda}^{\nu+\alpha}(\phi_r(y_\alpha v))\not=0$, but this contradicts the fact that $\nu+\alpha\not\in\mathcal{P}$.
It follows that $\textup{proj}^\nu_{M_\lambda}(\phi_r(v))$
is a highest weight vector in $M_\lambda$ of highest weight $\nu$. This forces $\nu=\lambda$ since $M_\lambda$ is irreducible, hence $\textup{hw}_{\lambda,V_r}(\phi_r)(v)\not=0$. It follows that $\textup{hw}_{\lambda,V_r}$ is injective, which completes the proof.
\end{proof}
\begin{defi}\label{intertwinerparametrization}
Let 
$\lambda\in\mathfrak{h}^*$. 
\begin{enumerate}
\item[{\bf a.}]
We call $\textup{ev}_{\lambda,V}(\phi_{\ell})$
the expectation value of the intertwiner $\phi_\ell\in\textup{Hom}_{\mathfrak{k}}(M_\lambda,V_\ell)$. We write
$\phi_{\ell,\lambda}^v\in\textup{Hom}_{\mathfrak{k}}(M_\lambda,V_\ell)$ for the $\mathfrak{k}$-intertwiner with expectation value $v\in V_\ell$. 
\item[{\bf b.}] We call
$\textup{hw}_{\lambda,V_r}(\phi_r)$ the highest weight component of the intertwiner $\phi_r\in\textup{Hom}_{\mathfrak{k}}(V_r,\overline{M}_\lambda)$. If $\lambda\in\mathfrak{h}_{\textup{irr}}^*$  then we write
$\phi_{r,\lambda}^f\in\textup{Hom}_{\mathfrak{k}}(V_r,\overline{M}_\lambda)$ for
the intertwiner with highest weight component $f\in V_r^*$.
\end{enumerate}
\end{defi}

The exact relation with the intertwiners from Proposition
\ref{relEisPrin} is as follows. Consider for $\sigma_\ell: K\rightarrow\textup{GL}(V_\ell)$ a finite dimensional $K$-representation the chain of linear isomorphisms
\[
\textup{Hom}_K(V_\ell^*,\mathcal{H}_{-\lambda-\rho})
\overset{\sim}{\longrightarrow}\textup{Hom}_{\mathfrak{k}}(V_\ell^*,M_\lambda^*)
\overset{\sim}{\longrightarrow}\textup{Hom}_{\mathfrak{k}}(M_\lambda,V_\ell)
\overset{\sim}{\longrightarrow} V_\ell.
\]
The first isomorphism is the pushforward of the map defined in Remark \ref{AA}, the second map is transposition and the third map is $\textup{ev}_{\lambda,V_\ell}$. Their composition
is the linear isomorphism $\j_{-\lambda-\rho,V_\ell^*}: \textup{Hom}_K(V_\ell^*,\mathcal{H}_{-\lambda-\rho})\overset{\sim}{\longrightarrow} V_\ell$ defined in Proposition \ref{relEisPrin}.
Similarly, for $\lambda\in\mathfrak{h}_{\textup{irr}}^*$ and $\sigma_r: K\rightarrow\textup{GL}(V_r)$ a finite dimensional $K$-representation we have the
chain of linear isomorphisms
\[
\textup{Hom}_K(V_r,\mathcal{H}_{-\lambda-\rho})
\overset{\sim}{\longrightarrow}\textup{Hom}_{\mathfrak{k}}(V_r,M_\lambda^*)
\overset{\sim}{\longrightarrow}\textup{Hom}_{\mathfrak{k}}(V_r,\overline{M}_\lambda)
\overset{\sim}{\longrightarrow} V_r^*.
\]
In this case the first isomorphism is the pushforward of the map defined in Remark \ref{AA},
the second isomorphism is the pushforward of the $\mathfrak{g}$-intertwiner 
$M_\lambda^*\overset{\sim}{\longrightarrow} {}^\theta\overline{M}_\lambda$ realized by the
Shapovalov form (see Lemma \ref{relCC}\,{\bf b} and its proof), 
and the third map is $\textup{hw}_{\lambda,V_r}$. Their composition
is the linear isomorphism $\j_{-\lambda-\rho,V_r}: \textup{Hom}_K(V_r,\mathcal{H}_{-\lambda-\rho})\overset{\sim}{\longrightarrow} V_r^*$ defined in Proposition \ref{relEisPrin}.

The following corollary is the analogue of Proposition \ref{relEisPrin}\,{\bf c} for Verma modules.
\begin{cor}\label{relEisVerma}
Let $V_\ell$ and $V_r$ be finite dimensional semisimple $\mathfrak{k}$-modules. The linear map
\begin{equation*}
\textup{Hom}_{\mathfrak{k}}(M_\lambda,V_{\ell})\otimes\textup{Hom}_{\mathfrak{k}}(V_{r},\overline{M}_\lambda)\rightarrow V_{\ell}\otimes V_{r}^*\simeq \textup{Hom}(V_{r},V_{\ell})
\end{equation*}
defined by $\phi_\ell\otimes\phi_r\mapsto 
S^{\phi_\ell,\phi_r}_\lambda:=
\textup{ev}_{\lambda,V_\ell}(\phi_\ell)\otimes\textup{hw}_{\lambda,V_r}(\phi_r)$, is
a linear isomorphism when $\lambda\in\mathfrak{h}_{\textup{irr}}^*$. 
\end{cor}
In Subsection \ref{HCrepSec} we will give a representation interpretation of the analytic $\textup{Hom}(V_r,V_\ell)$-valued function $a\mapsto \Phi_\lambda^\sigma(a)S_\lambda^{\phi_\ell,\phi_r}$
 for $\phi_\ell\in\textup{Hom}_{\mathfrak{k}}(M_\lambda,V_{\ell})$,  
 $\phi_r\in \textup{Hom}_{\mathfrak{k}}(V_{r},\overline{M}_\lambda)$ and $a\in A_+$,
 with $\sigma$ the representation map of the $\mathfrak{k}\oplus\mathfrak{k}$-module
 $V_\ell\otimes V_r^*\simeq \textup{Hom}(V_r,V_\ell)$.
\subsection{The construction of the formal elementary spherical functions}
We first introduce $(\mathfrak{h},\mathfrak{k})$-finite and $(\mathfrak{k},\mathfrak{h})$-finite
matrix coefficients of Verma modules. Let $\lambda,\mu\in\mathfrak{h}^*$ with $\mu\leq\lambda$. Recall the projection and inclusion maps $\textup{incl}_{M_\lambda}^\mu$
and $\textup{proj}_{M_\lambda}^\mu$. They are $\mathfrak{h}$-intertwiners
\[
\textup{incl}_{M_\lambda}^\mu\in\textup{Hom}_{\mathfrak{h}}(M_\lambda[\mu],M_\lambda),
\qquad \textup{proj}_{M_\lambda}^\mu\in\textup{Hom}_{\mathfrak{h}}(
\overline{M}_\lambda,M_\lambda[\mu]).
\]
Let $\lambda,\mu\in\mathfrak{h}^*$ with $\mu\leq\lambda$, and fix
$\mathfrak{k}$-intertwiners
$\phi_\ell\in\textup{Hom}_{\mathfrak{k}}(M_\lambda,V_\ell)$ and 
$\phi_r\in\textup{Hom}_{\mathfrak{k}}(V_r,\overline{M}_\lambda)$.
We write 
\[
\phi_\ell^\mu:=\phi_\ell\circ \textup{incl}_{M_\lambda}^\mu\in\textup{Hom}(M_\lambda[\mu],V_\ell),
\qquad
\phi^\mu_r:=\textup{proj}_{M_\lambda}^\mu\circ\phi_r\in\textup{Hom}(V_r,M_\lambda[\mu])
\]
for the weight-$\mu$ components of $\phi_\ell$ and $\phi_r$, respectively.
The map $\phi^\mu_\ell$ encodes $(\mathfrak{k},\mathfrak{h})$-finite matrix coefficients of $M_\lambda$ of type $(\sigma_\ell,\mu)$, and $\phi^\mu_r$ the
$(\mathfrak{h},\mathfrak{k})$-finite matrix coefficients of $M_\lambda$ of type 
$(\mu,\sigma_r)$.
Formal elementary spherical functions are now defined to be the generating series of the compositions 
$\phi_\ell^\mu\circ\phi_r^\mu$ of the weight compositions of the $\mathfrak{k}$-intertwiners $\phi_\ell$
and $\phi_r$:
\begin{defi}\label{Glambda}
Let $\lambda\in\mathfrak{h}^*$.
For $\phi_\ell\in\textup{Hom}_{\mathfrak{k}}(M_\lambda,V_\ell)$ and 
$\phi_r\in\textup{Hom}_{\mathfrak{k}}(V_r,\overline{M}_\lambda)$ let
\[
F_{M_\lambda}^{\phi_\ell,\phi_r}\in (V_\ell\otimes V_r^*)[[\xi_{-\alpha_1},\ldots,\xi_{-\alpha_\rr}]]\xi_\lambda
\]
be the formal $V_\ell\otimes\mathbf{U}\otimes V_r^*$-valued power series
\[
F_{M_\lambda}^{\phi_\ell,\phi_r}:=\sum_{\mu\leq\lambda}(\phi_\ell\circ\phi_r^\mu)\xi_\mu
=\sum_{\mu\leq\lambda}(\phi_\ell^\mu\circ\phi_r^\mu)\xi_\mu.
\]
We call $F_{M_\lambda}^{\phi_\ell,\phi_r}$ the formal elementary $\sigma$-spherical function 
associated to $M_\lambda$, $\phi_\ell$ and $\phi_r$.
\end{defi}
Note that under the natural identification $\textup{Hom}(V_r,V_\ell)\simeq V_\ell\otimes V_r^*$
we have
\[
\phi_\ell^\lambda\circ\phi_r^\lambda=\textup{ev}_{\lambda,V_\ell}(\phi_\ell)\otimes
\textup{hw}_{\lambda,V_r}(\phi_r)=S_{\lambda}^{\phi_\ell,\phi_r},
\]
hence $F_{M_\lambda}^{\phi_\ell,\phi_r}$ has leading coefficient $S_{\lambda}^{\phi_\ell,\phi_r}$.
By Corollary \ref{relEisVerma}, if $\lambda\in\mathfrak{h}_{\textup{irr}}^*$ and 
$v\in V_\ell$, $f\in V_r^*$, the formal elementary $\sigma$-spherical function $F_{M_\lambda}^{\phi_{\ell,\lambda}^v,\phi_{r,\lambda}^f}$ is the unique formal elementary $\sigma$-spherical function associated to $M_\lambda$
with leading coefficient $v\otimes f$. We will denote it by $F_{M_\lambda}^{v,f}$.

\subsection{Relation to $\sigma$-Harish-Chandra series}\label{HCrepSec}
Recall the Harish-Chandra coefficients
$\Gamma_\lambda^\sigma(\mu)\in V_\ell\otimes V_r^*$ in the power series expansion of the $\sigma$-Harish-Chandra series $\Phi_\lambda^\sigma$, see Proposition \ref{cordefHC}. 
We have the following main result of Section \ref{RTHC}.

\begin{thm}\label{mainTHMF}
Let $\lambda\in\mathfrak{h}^*$, $\phi_\ell\in\textup{Hom}_{\mathfrak{k}}(M_\lambda,V_\ell)$
and $\phi_r\in\textup{Hom}_{\mathfrak{k}}(V_r,\overline{M}_\lambda)$.
\begin{enumerate}
\item[{\bf a.}] For $z\in Z(\mathfrak{g})$,
\begin{equation}\label{todo2}
\widehat{\Pi}^\sigma(z)F_{M_\lambda}^{\phi_\ell,\phi_r}=\zeta_\lambda(z)F_{M_\lambda}^{\phi_\ell,\phi_r}
\end{equation}
as identity in $(V_\ell\otimes V_r^*)[[\xi_{-\alpha_1},\ldots,\xi_{-\alpha_\rr}]]\xi_\lambda$.
\item[{\bf b.}] For $\lambda\in\mathfrak{h}_{\textup{HC}}^*$
and $\mu\leq\lambda$ 
we have
\begin{equation}\label{mixedmatrixformula}
\phi_\ell\circ\phi_r^\mu=\phi^\mu_\ell\circ\phi_r^\mu=\Gamma_\mu^\sigma(\lambda)S_\lambda^{\phi_\ell,
\phi_r}.
\end{equation}
\item[{\bf c.}] For $\lambda\in\mathfrak{h}_{\textup{HC}}^*$, $v\in V_\ell$ and $f\in V_r^*$ we have
\begin{equation}\label{relationFspher}
F_{M_\lambda}^{v,f}=\Phi_\lambda^\sigma(\cdot)(v\otimes f).
\end{equation}
In particular, $F_{M_\lambda}^{v,f}$ is a $V_\ell\otimes V_r^*$-valued analytic function on $A_+$.
\end{enumerate}
\end{thm}
\begin{proof}
Since $\mathfrak{h}_{\textup{HC}}^*\subseteq\mathfrak{h}_{\textup{irr}}^*$, part {\bf b} and {\bf c} of the theorem directly follow from part {\bf a}, Proposition \ref{cordefHC}, Proposition \ref{holomorphic} and the fact that the leading coefficient of $F_{M_\lambda}^{v\otimes f}$ is $v\otimes f$. It thus suffices to prove \eqref{todo2}.

Consider the $Q$-grading $U(\mathfrak{g})=\bigoplus_{\gamma\in Q}U[\gamma]$  with 
$U[\gamma]\subset U(\mathfrak{g})$ the subspace consisting of elements 
$x\in U(\mathfrak{g})$ satisfying $\textup{Ad}_a(x)=a^\gamma x$ for all $a\in A$. 
Set
\[
\Lambda:=\{\mu\in\mathfrak{h}^* \,\,\, | \,\,\, \mu\leq\lambda \},
\qquad
\Lambda_m:=\{\mu\in\Lambda\,\,\, | \,\,\, (\lambda-\mu,\rho^\vee)\leq m\}
\]
for $m\in\mathbb{Z}_{\geq 0}$, where $\rho^\vee=\frac{1}{2}\sum_{\alpha\in R^+}\alpha^\vee$. Then $(\lambda-\mu,\rho^\vee)$ is the height of $\lambda-\mu\in\sum_{i=1}^\rr\mathbb{Z}_{\geq 0}\alpha_i$ with respect to the basis $\{\alpha_1,\ldots,\alpha_\rr\}$ of $R$. We will prove that
\begin{equation}\label{todoformal2}
\widehat{\Pi}^{\sigma}(x)F_{M_\lambda}^{\phi_\ell,\phi_r}=\sum_{\mu\in\Lambda}\phi_\ell(x\phi_r^\mu)\xi_\mu
\qquad (x\in U[0])
\end{equation}
in $\textup{Hom}(V_r,V_\ell)[[\xi_{-\alpha_1},\ldots,\xi_{-\alpha_\rr}]]\xi_\lambda$.
This implies \eqref{todo2}, since $Z(\mathfrak{g})\subseteq U[0]$ and $M_\lambda$ admits the central character $\zeta_\lambda$. 

Fix $x\in U[0]$ and write $\Pi(x)=\sum_{j\in J}f_j\otimes h_j\otimes y_j\otimes z_j$ with
$f_j\in\mathcal{R}$, $h_j\in U(\mathfrak{h})$ and $y_j, z_j\in U(\mathfrak{k})$.
By Theorem \ref{infKAK} we get the infinitesimal Cartan decomposition 
\begin{equation}\label{infKAKspecific}
x=\textup{Ad}_a(x)=\sum_{j\in J}f_j(a)y_jh_j\textup{Ad}_a(z_j)\qquad (a\in A_{+})
\end{equation}
of $x\in U[0]$. We will now substitute this decomposition in the truncated version 
$\sum_{\lambda\in\Lambda_m}\phi_\ell(x\phi_r^\mu)\xi_\mu$ of the right hand side of \eqref{todoformal2}.

Note that $\sum_{\lambda\in\Lambda_m}\phi_\ell(x\phi_r^\mu)\xi_\mu\in 
\textup{Hom}(V_r,V_\ell)[\xi_{-\alpha_1},\ldots,\xi_{-\alpha_\rr}]\xi_\lambda$ is a trigonometric quasi-po\-ly\-no\-mial, hence it can be
evaluated at $a\in A_{+}$. Substituting \eqref{infKAKspecific} and using that
$\phi_\ell$ is a $\mathfrak{k}$-intertwiner we obtain the formula
\begin{equation}\label{blob1}
\sum_{\lambda\in\Lambda_m}\phi_\ell(x\phi_r^\mu)a^\mu=
\sum_{j\in J}\sum_{\mu\in\Lambda_m}f_j(a)\sigma_\ell(y_j)
\phi_\ell(h_j\textup{Ad}_a(z_j)\phi_r^\mu)a^\mu
\end{equation}
in $\textup{Hom}(V_r,V_\ell)$. Now expand $z_j=\sum_{\gamma\in I_j}z_j[\gamma]$ along the $Q$-grading of $U(\mathfrak{g})$ with $z_j[\gamma]\in U[\gamma]$ (but no longer in $U(\mathfrak{k})$).
Here $I_j\subset Q$ denotes the finite set of weights for which $z_j[\gamma]\not=0$. Then \eqref{blob1} implies
\begin{equation*}
\sum_{\mu\in\Lambda_m}\phi_\ell(x\phi_r^\mu)a^\mu
=\sum_{j\in J}\sum_{\gamma\in I_j}\sum_{\mu\in\Lambda_m}(\mu+\gamma)(h_j)
\sigma_\ell(y_j)\phi_\ell(z_j[\gamma]\phi_r^\mu)f_j(a)a^{\mu+\gamma}
\end{equation*}
for all $a\in A_{+}$. Let $\eta\geq\lambda$ such that 
$\lambda+\gamma\leq \eta$ for all $\gamma\in I:=\cup_{j\in J}I_j$. Then we conclude that
\begin{equation}\label{blob3}
\sum_{\mu\in\Lambda_m}\phi_\ell(x\phi_r^\mu)\xi_\mu=
\sum_{j\in J}\sum_{\gamma\in I_j}\sum_{\mu\in\Lambda_m}(\mu+\gamma)(h_j)\sigma_\ell(y_j)
\phi_\ell(z_j[\gamma]\phi_r^\mu)f_j\xi_{\mu+\gamma}
\end{equation}
in $\textup{Hom}(V_r,V_\ell)[[\xi_{-\alpha_1},\ldots,\xi_{-\alpha_\rr}]]\xi_\eta$ (here 
the $f_j\in\mathcal{R}$ are represented by their convergent power series 
$f_j=\sum_{\beta\in Q_-}
c_{j,\beta}\xi_\beta$ on $A_+$ ($c_{j,\beta}\in\mathbb{C}$)). We now claim that \eqref{blob3} is valid with the truncated sum over $\Lambda_m$ replaced by the sum over $\Lambda$,
\begin{equation}\label{blob4}
\sum_{\mu\in\Lambda}\phi_\ell(x\phi_r^\mu)\xi_\mu=
\sum_{j\in J}\sum_{\gamma\in I_j}\sum_{\mu\in\Lambda}(\mu+\gamma)(h_j)\sigma_\ell(y_j)
\phi_\ell(z_j[\gamma]\phi_r^\mu)f_j\xi_{\mu+\gamma}
\end{equation}
in $\textup{Hom}(V_r,V_\ell)[[\xi_{-\alpha_1},\ldots,\xi_{-\alpha_\rr}]]\xi_\eta$.

Fix $\nu\in\mathfrak{h}^*$ with $\nu\leq \eta$. 
It suffices to show that the $\xi_\nu$-component of the left 
(resp. right) hand side of \eqref{blob4} is the same as the $\xi_\nu$-component
of the left (resp. right) hand side of \eqref{blob3} when $m\in\mathbb{Z}_{\geq 0}$ satisfies
$(\eta-\nu,\rho^\vee)\leq m$. 

Choose $m\in\mathbb{Z}_{\geq 0}$ with $(\eta-\nu,\rho^\vee)\leq m$. The $\xi_\nu$-component of the left hand side of \eqref{blob4}
is zero if $\nu\not\in\Lambda$ and $\phi_\ell(x\phi_r^\nu)$ otherwise.
Since $(\lambda-\nu,\rho^\vee)\leq (\eta-\nu,\rho^\vee)\leq m$, this coincides with the $\xi_\nu$-component of the left hand side of \eqref{blob3}. The $\xi_\nu$-component of the right hand side of \eqref{blob4} is
\begin{equation}\label{blob5}
\sum_{(j,\beta,\gamma,\mu)}(\mu+\gamma)(h_j)c_{j,\beta}\sigma_\ell(y_j)
\phi_\ell(z_j[\gamma]\phi_r^\mu)\in V_\ell\otimes V_r^*
\end{equation}
with the sum over the finite set of four tuples $(j,\beta,\gamma,\mu)\in J\times Q_-\times I\times 
\Lambda$ satisfying $\gamma\in I_j$ and $\mu+\gamma+\beta=\nu$. For such a four tuple
we have 
\[
(\lambda-\mu,\rho^\vee)=(\lambda+\gamma-\nu+\beta,\rho^\vee)\leq
(\eta-\nu,\rho^\vee)\leq m,
\]
from which it follows that \eqref{blob5} is also the $\xi_\nu$-component of the right hand side of \eqref{blob3}. This concludes the proof of \eqref{blob4}.

Since $\phi_r$ is a $\mathfrak{k}$-intertwiner,
we have for fixed $\nu\in\mathfrak{h}^*$,
\[
\sum_{\stackrel{(\mu,\gamma)\in \Lambda\times I_j:}{\mu+\gamma=\nu}}
z_j[\gamma]\phi_r^\mu=\phi_r^\nu \sigma_r(z_j)
\]
in $\textup{Hom}(V_r,M_\lambda[\nu])$ (in particular, it is zero when $\nu\not\in\Lambda$).
Hence \eqref{blob4} simplifies to
\begin{equation}\label{blob6}
\begin{split}
\sum_{\mu\in\Lambda}\phi_\ell(x\phi_r^\mu)\xi_\mu&=
\sum_{j\in J}\sum_{\nu\in\Lambda}\bigl(\sigma_\ell(y_j)\phi_\ell(\phi_r^\nu)\sigma_r^*(z_j)\bigr)f_j\nu(h_j)\xi_\nu\\
&=\widehat{\Pi}^\sigma(x)F_{M_\lambda}^{\phi_\ell,\phi_r}
\end{split}
\end{equation}
in $\textup{Hom}(V_r,V_\ell)[[\xi_{-\alpha_1},\ldots,\xi_{-\alpha_\rr}]]\xi_\lambda$, as desired.
\end{proof}
Recall the normalisation factor $\delta$, defined by \eqref{deltagauge}. We will also view $\delta$ as formal series
in $\mathbb{C}[[\xi_{-\alpha_1},\ldots,\xi_{-\alpha_\rr}]]\xi_\rho$ through its power series expansion at infinity within $A_+$. 
\begin{defi}
Let $\lambda\in\mathfrak{h}^*$ and
fix $\mathfrak{k}$-intertwiners $\phi_\ell\in\textup{Hom}_{\mathfrak{k}}(M_{\lambda-\rho},V_\ell)$ and 
$\phi_r\in\textup{Hom}_{\mathfrak{k}}(V_r,\overline{M}_{\lambda-\rho})$. We call
\[
\mathbf{F}_\lambda^{\phi_\ell,\phi_r}:=\delta F_{M_{\lambda-\rho}}^{\phi_\ell,\phi_r}\in
(V_\ell\otimes V_r^*)[[\xi_{-\alpha_1},\ldots,\xi_{-\alpha_\rr}]]\xi_{\lambda}
\]
the normalised formal elementary $\sigma$-spherical function of weight $\lambda$.
We furthermore write  for $\lambda\in\mathfrak{h}^*_{\textup{irr}}+\rho$ and
$v\in V_\ell$, $f\in V_r^*$,
\[
\mathbf{F}_\lambda^{v, f}:=
\delta F_{M_{\lambda-\rho}}^{v\otimes f},
 \]
which is the normalised formal elementary $\sigma$-spherical function of weight $\lambda$
with leading coefficient $v\otimes f$.
\end{defi}
By Theorem \ref{mainTHMF}\,{\bf c} we have 
for $\lambda\in\mathfrak{h}_{\textup{HC}}^*+\rho$, $v\in V_\ell$ and $f\in V_r^*$,
\[
\mathbf{F}_\lambda^{v,f}=\mathbf{\Phi}_\lambda^\sigma(\cdot)(v\otimes f).
\]
In particular, $\mathbf{F}_\lambda^{v,f}$ is an $V_\ell\otimes V_r^*$-valued analytic function on $A_+$ when
$\lambda\in\mathfrak{h}_{\textup{HC}}^*+\rho$.

Theorem \ref{normalizedHCseriesthm} now immediately gives the interpretation of
$\mathbf{F}_\lambda^{\phi_\ell,\phi_r}$ as formal eigenstates for the $\sigma$-spin
quantum hyperbolic Calogero-Moser system for all weights $\lambda\in\mathfrak{h}^*$.
\begin{thm}\label{thmnorm}
Let $\lambda\in\mathfrak{h}^*$, $\phi_\ell\in\textup{Hom}_{\mathfrak{k}}(M_\lambda,V_\ell)$
and $\phi_r\in\textup{Hom}_{\mathfrak{k}}(V_r,\overline{M}_\lambda)$. The 
normalised formal elementary 
$\sigma$-spherical function $\mathbf{F}_\lambda^{\phi_\ell,\phi_r}$ of weight
$\lambda$ satisfies the Schr{\"o}dinger equation
\[
\mathbf{H}^\sigma(\mathbf{F}_\lambda^{\phi_\ell,\phi_r})=-\frac{(\lambda,\lambda)}{2}
\mathbf{F}_\lambda^{\phi_\ell,\phi_r}
\]
as well as the eigenvalue equations 
\begin{equation}\label{todo2gauged}
H_z^\sigma(\mathbf{F}_\lambda^{\phi_\ell,\phi_r})=\zeta_{\lambda-\rho}(z)\mathbf{F}_\lambda^{\phi_\ell,\phi_r},
\qquad z\in Z(\mathfrak{g})
\end{equation}
in $(V_\ell\otimes V_r^*)[[\xi_{-\alpha_1},\ldots,\xi_{-\alpha_\rr}]]\xi_{\lambda}$.
\end{thm}
\begin{proof}
This follows from the differential equations \eqref{todo2} for the formal elementary
$\sigma$-spherical functions, and the results in Subsection \ref{vvCM}.
\end{proof}
\begin{rema}
{\bf a.} 
Let $\lambda\in\mathfrak{h}_{\textup{HC}}^*$ and fix $V_\ell,V_r$ finite dimensional $K$-representations. Let $\sigma$ be the resulting tensor product representation of $K\times K$ on
$V_\sigma:=V_\ell\otimes V_r^*\simeq \textup{Hom}(V_r,V_\ell)$. For 
$T\in\textup{Hom}_M(V_r,V_\ell)$ write $H_\lambda^T$ for the $V_\sigma$-valued
smooth
function on $G_{\textup{reg}}$ constructed from the Harish-Chandra series $\Phi_\lambda^\sigma$ by
\[
H_\lambda^T(k_1ak_2^{-1}):=\sigma(k_1,k_2)\Phi_\lambda^\sigma(a)T
\qquad (a\in A_+,\,\, k_1,k_2\in K),
\]
see Remark  \ref{remaHCspherplus}. Then \eqref{relationFspher} gives an interpretation of $H_\lambda^T$ as formal elementary $\sigma$-spherical function associated with $M_\lambda$. 
This should be compared with 
\eqref{relEisxi},
which gives an interpretation of the Eisenstein integral as spherical function associated to the principal series representation of $G$.\\
{\bf b.} In \cite[Thm. 4.4]{K} Kolb proved an affine rank one analogue of Theorem \ref{mainTHMF} for the pair $(\widehat{\mathfrak{sl}}_2,\widehat{\theta})$, where $\widehat{\theta}$ the Chevalley involution on the affine Lie algebra $\widehat{\mathfrak{sl}}_2$ associated to $\mathfrak{sl}_2$. 
The generalisation of Theorem \ref{mainTHMF} to arbitrary split affine symmetric pairs
will be discussed in a follow-up paper.
\end{rema}

\subsection{The rank one example}\label{rankoneSection}
In this subsection we consider $\mathfrak{g}_0=\mathfrak{sl}(2;\mathbb{R})$ 
with linear basis
\[
H=\left(\begin{matrix} 1 & 0\\ 0 & -1\end{matrix}\right), \quad
E=\left(\begin{matrix} 0 & 1\\ 0 & 0\end{matrix}\right),\quad
F=\left(\begin{matrix}  0 & 0\\ 1 & 0\end{matrix}\right),
\] 
and we take $\mathfrak{h}_0=\mathbb{R}H$ as split Cartan subalgebra.  Then $\theta_0(x)=-x^t$
with $x^t$ the transpose of $x\in\mathfrak{g}_0$. Note that $\frac{H}{2\sqrt{2}}\in\mathfrak{h}_0$ has norm one with respect to the Killing form. Let $\alpha$ be the unique positive root, satisfying $\alpha(H)=2$. Then $t_\alpha=\frac{H}{4}$, and we can take $e_\alpha=\frac{E}{2}$ and $e_{-\alpha}=\frac{F}{2}$. With this choice we have $\mathfrak{k}_0=\mathbb{R}y$ with $y:=y_\alpha=\frac{E}{2}-\frac{F}{2}$.

We identify $\mathfrak{h}^*\overset{\sim}{\longrightarrow} \mathbb{C}$ by the map $\lambda\mapsto \lambda(H)$. The positive root $\alpha\in\mathfrak{h}^*$ then corresponds to $2$. The bilinear form
on $\mathfrak{h}^*$ becomes $(\lambda,\mu)=\frac{1}{8}\lambda\mu$ for $\lambda,\mu\in\mathbb{C}$. Furthermore, 
$\mathfrak{h}^*_{\textup{irr}}=\mathfrak{h}^*_{\textup{HC}}$ becomes $\mathbb{C}\setminus
\mathbb{Z}_{\geq 0}$. We also identify $A\overset{\sim}{\longrightarrow} \mathbb{R}_{>0}$ by $\exp_A(sH)\mapsto e^s$  ($s\in\mathbb{R}$). With these identifications, the multiplicative character 
$\xi_\lambda$ on $A$ ($\lambda\in\mathfrak{h}^*$) becomes $\xi_\lambda(a)=a^{\lambda}$ for $a\in\mathbb{R}_{>0}$ and $\lambda\in\mathbb{C}$. 

For $\nu\in\mathbb{C}$ let $\chi_\nu\in\mathfrak{k}^\wedge$ be the one-dimensional representation mapping $y$ to $\nu$. We write $\mathbb{C}_\nu$ for $\mathbb{C}$ regarded as $\mathfrak{k}$-module by $\chi_\nu$. For $\lambda\in\mathbb{C}\setminus\mathbb{Z}_{\geq 0}$ and 
$\nu_\ell, \nu_r\in\mathbb{C}$, the scalar-valued Harish-Chandra series $\Phi_\lambda^{\chi_{\nu_\ell}\otimes \chi_{\nu_r}}$ is the unique analytic function on $A_+$ admitting a power series of the form
\begin{equation}\label{asymptotics}
\Phi_\lambda^{\chi_{\nu_\ell}\otimes\chi_{\nu_r}}(a)=a^\lambda \sum_{k\geq 0}
\Gamma_{-2k}^{\chi_{\nu_\ell}\otimes\chi_{\nu_r}}(\lambda)a^{-2k}\qquad
(a\in A_+)
\end{equation}
with $\Gamma_0^{\chi_{\nu_\ell}\otimes\chi_{\nu_r}}(\lambda)=1$ that satisfies the differential equation
\begin{equation}\label{GHE}
\widehat{\Pi}(\Omega)\Phi_\lambda^{\chi_{\nu_\ell}\otimes\chi_{\nu_r}}=\frac{\lambda(\lambda+2)}{8}\Phi_\lambda^{\chi_{\nu_\ell}\otimes\chi_{\nu_r}}
\end{equation}
(see Proposition \ref{cordefHC}). By Corollary \ref{corR1},
\begin{equation*}
\begin{split}
(\textup{id}_{\mathbb{D}_{\mathcal{R}}}&\otimes\chi_{\nu_\ell}\otimes\chi_{\nu_r})\widehat{\Pi}(\Omega)=
\frac{1}{8}\Bigl(a\frac{d}{da}\Bigr)^2+\frac{1}{4}\left(\frac{a^2+a^{-2}}{a^2-a^{-2}}\right)
a\frac{d}{da}+\frac{2(\nu_\ell+a^2\nu_r)(\nu_\ell+a^{-2}\nu_r)}{(a^2-a^{-2})^2}\\
&=\frac{1}{8}\left(\Bigl(a\frac{d}{da}\Bigr)^2+\left(
\frac{a+a^{-1}}{a-a^{-1}}+\frac{a-a^{-1}}{a+a^{-1}}\right)a\frac{d}{da}
+\frac{4(\nu_\ell+\nu_r)^2}{(a-a^{-1})^2}-\frac{4(\nu_\ell-\nu_r)^2}{(a+a^{-1})^2}\right).
\end{split}
\end{equation*}
The equation \eqref{GHE} is the second-order differential equation that is solved by the associated Jacobi function
(cf., e.g., \cite[\S 4.2]{Ko2}), and $\Phi_\lambda^{\chi_{\nu_\ell}\otimes\chi_{\nu_r}}$ 
is the corresponding asymptotically free solution. 

An explicit expression for
$\Phi_{\lambda}^{\chi_{\nu_\ell}\otimes\chi_{\nu_r}}$ can be now derived as follows.
Rewrite \eqref{GHE} as a second order differential 
equation for $h\Phi_\lambda^{\chi_{\nu_\ell}\otimes\chi_{\nu_r}}$ with
\[
h(a):=(a+a^{-1})^{i(\nu_r-\nu_\ell)}(a-a^{-1})^{-i(\nu_\ell+\nu_r)}.
\]
One then recognizes the resulting differential equation as 
the second-order differential equation \cite[(2.10)]{Ko2} satisfied by the Jacobi function (which is the Gauss' hypergeometric differential equation after an appropriate change of coordinates). Its solutions can be expressed in terms of the
Gauss' hypergeometric series 
\begin{equation}\label{Gauss}
{}_2F_1(a,b;c \, \vert\, s):=\sum_{k=0}^{\infty}\frac{(a)_k(b)_k}{(c)_kk!}s^k,
\end{equation}
where $(a)_k:=a(a+1)\cdots (a+k-1)$ is the Pochhammer symbol (the series \eqref{Gauss} converges for $|s|<1$). 
Then $\Phi_\lambda^{\chi_{\nu_\ell}\otimes\chi_{\nu_r}}$ corresponds to the solution \cite[(2.15)]{Ko2} of the second-order differential equation \cite[(2.10)]{Ko2}. Performing the straightforward  computations gives the following result.
\begin{prop}\label{GaussExpres}
For $\lambda\in\mathbb{C}\setminus\mathbb{Z}_{\geq 0}$ we have
\begin{equation*}
\Phi_\lambda^{\chi_{\nu_\ell}\otimes\chi_{\nu_r}}(a)=
(a+a^{-1})^\lambda\left(\frac{a-a^{-1}}{a+a^{-1}}\right)^{i(\nu_\ell+\nu_r)}
{}_2F_1\Bigl(-\frac{\lambda}{2}+i\nu_\ell, -\frac{\lambda}{2}+i\nu_r; -\lambda \,\, | \,\, 
\frac{4}{(a+a^{-1})^2}\Bigr)
\end{equation*}
for $a>1$.
\end{prop}
 If $\chi_{-\nu_\ell}, \chi_{-\nu_r}\in K^\wedge$ (i.e., if 
$\nu_\ell, \nu_r\in i\mathbb{Z}$), the restriction of the elementary spherical function $f_{\pi_\lambda}^{\chi_{-\nu_\ell},\chi_{-\nu_r}}$ to $A$ is an associated Jacobi function. It can also be expressed
in terms of a single ${}_2F_1$
(see \cite[\S 4.2]{Ko2} and references therein for details).

We compute now an alternative expression for $\Phi_\lambda^{\chi_{\nu_\ell}\otimes\chi_{\nu_r}}$
using its realisation as generating function for compositions of weight components of $\mathfrak{k}$-intertwiners (see Theorem \ref{mainTHMF}).

The Verma module $M_\lambda$ with highest weight $\lambda\in\mathbb{C}$ is explicitly realised as
\[
M_\lambda=\bigoplus_{k=0}^\infty\mathbb{C}u_k
\]
with $u_k:=\frac{1}{k!}F^km_\lambda$ and $\mathfrak{g}$-action $Hu_k=(\lambda-2k)u_k$, $Eu_k=(\lambda-k+1)u_{k-1}$ and $Fu_k=(k+1)u_{k+1}$, where we have set $u_{-1}:=0$. 
We will identify 
\[
\textup{Hom}_{\mathfrak{k}}(\mathbb{C}_\nu,\overline{M}_\lambda)\overset{\sim}{\longrightarrow}
\overline{M}_\lambda^{\chi_\nu},\qquad \psi\mapsto \psi(1),
\]
with $\overline{M}_\lambda^{\chi_\nu}$ the space of $\chi_\nu$-invariant vectors in $\overline{M}_\lambda$ (cf.\ Subsection \ref{chisection}). Note furthermore that
\[
\textup{Hom}_{\mathfrak{k}}(M_\lambda,\mathbb{C}_\nu)\simeq M_{\lambda}^{\ast, \chi_{-\nu}}.
\]
To apply Theorem \ref{mainTHMF} to $\Phi_\lambda^{\chi_{\nu_\ell}\otimes\chi_{\nu_r}}$, we thus need to describe the weight components of the nonzero vectors in the one-dimensional subspaces  $M_{\lambda}^{\ast, \chi_{-\nu_\ell}}$ and $\overline{M}_\lambda^{\chi_{-\nu_r}}$ for $\lambda\in\mathbb{C}\setminus\mathbb{Z}_{\geq 0}$ (here we use that
$\mathbb{C}_{\nu_r}\simeq\mathbb{C}_{-\nu_r}^*$ as $\mathfrak{k}$-modules). For this we need 
some facts about Meizner-Pollaczek polynomials, which we recall from
\cite[\S 9.7]{KS}.

Meixner-Pollaczek polynomials are orthogonal polynomials depending on two
parameters $(\lambda,\phi)$, of which we only need the special case $\phi=\frac{\pi}{2}$.
The monic Meixner-Pollaczek polynomials $\{p_k^{(\lambda)}(s) \,\, | \,\, k\in\mathbb{Z}_{\geq 0}\}$ with $\phi=\frac{\pi}{2}$ are given by
\[
p_k^{(\lambda)}(s)=(2\lambda)_k\Bigl(\frac{i}{2}\Bigr)^k{}_2F_1\bigl(-k,\lambda+is; 2\lambda \, \vert\,  2\bigr).
\]
They satisfy the three-term recursion relation
\begin{equation}\label{3termMP}
p_{k+1}^{(\lambda)}(s)-sp_k^{(\lambda)}(s)+\frac{k(k+2\lambda-1)}{4}p_{k-1}^{(\lambda)}(s)=0,
\end{equation}
where $p_{-1}^{(\lambda)}(s):=0$. 

The following result should be compared with \cite{BW, Ko}, where mixed matrix coefficients of discrete series representations of $\textup{SL}(2;\mathbb{R})$ with respect to hyperbolic and elliptic one-parameter subgroups of $\textup{SL}(2,\mathbb{R})$ are expressed in terms of Meixner-Pollaczek polynomials. 
\begin{lem}\label{r1}
Fix $\lambda\in\mathbb{C}\setminus\mathbb{Z}_{\geq 0}$ and $\nu\in\mathbb{C}$.
\begin{enumerate}
\item[{\bf a.}] We have
\[
\overline{M}_\lambda^{\chi_{-\nu}}=\mathbb{C}v_{\lambda; \nu}
\]
with the $(\lambda-2k)$-weight coefficient of $v_{\lambda;\nu}$ given by
\[
v_{\lambda;\nu}[\lambda-2k]=\frac{(-2)^kp_k^{(-\lambda/2)}(-\nu)}{(-\lambda)_k}u_k,
\qquad k\in\mathbb{Z}_{\geq 0}.
\]
\item[{\bf b.}] We have
\[
M_\lambda^{\ast,\chi_{-\nu}}=\mathbb{C}\psi_{\lambda;\nu}
\]
with $\psi_{\lambda;\nu}$ satisfying
\[
\psi_{\lambda;\nu}(u_k)=\frac{2^kp_k^{(-\lambda/2)}(-\nu)}{k!},\qquad k\in\mathbb{Z}_{\geq 0}.
\]
\end{enumerate}
\end{lem}
\begin{proof}
{\bf a.} The requirement $yv=-\nu v$ for an element $v\in \overline{M}_\lambda$ 
with weight components of the form
\[
v[\lambda-2k]=\frac{(-2)^kc_k}{(-\lambda)_k}u_k
\]
is equivalent to the condition that the coefficients $c_k\in\mathbb{C}$ ($k\geq 0$) satisfy the three-term recursion relation
\[
c_{k+1}+\nu c_k+\frac{k(k-\lambda-1)}{4}c_{k-1}=0,\qquad k\in\mathbb{Z}_{\geq 0},
\]
where $c_{-1}:=0$. By \eqref{3termMP}, the solution of this three-term recursion relation satisfying $c_0=1$ is given by  
$c_k=p_k^{(-\lambda/2)}(-\nu)$ ($k\in\mathbb{Z}_{\geq 0}$).\\
{\bf b.} The proof is similar to the proof of part {\bf a}.
\end{proof}

Let $\nu_\ell, \nu_r\in\mathbb{C}$ and $\lambda\in\mathbb{C}\setminus\mathbb{Z}_{\geq 0}$. We obtain from Lemma \ref{r1} and Theorem \ref{mainTHMF}
the following expression for the Harish-Chandra series 
$\Phi_\lambda^{\chi_{\nu_\ell}\otimes\chi_{\nu_r}}$.
\begin{cor}
Fix $\lambda\in\mathbb{C}\setminus\mathbb{Z}_{\geq 0}$
and $\nu_\ell, \nu_r\in\mathbb{C}$. We have for $a\in \mathbb{R}_{>1}$,
\[
\Phi_\lambda^{\chi_{\nu_\ell}\otimes\chi_{\nu_r}}(a)=a^\lambda\sum_{k=0}^{\infty}
\frac{4^{k}p_k^{(-\lambda/2)}(-\nu_\ell)p_k^{(-\lambda/2)}(-\nu_r)}{(-\lambda)_kk!}
(-a^{-2})^k.
\]
\end{cor}
\begin{proof}
Note that $v_{\lambda;\nu}[\lambda]=m_\lambda$
and $\textup{ev}_{\lambda,\mathbb{C}_{\chi_\nu}}(\psi_{\lambda;\nu})=\psi_{\lambda;\nu}(m_\lambda)=1$,
hence 
\[
S^{\psi_{\lambda;\nu_\ell},v_{\lambda; \nu_r}}_\lambda=1
\]
(we identify $\textup{Hom}(\mathbb{C}_{\nu_r},\mathbb{C}_{\nu_\ell})\overset{\sim}{\longrightarrow}
\mathbb{C}$ by $T\mapsto T(1)$). 
By Theorem \ref{mainTHMF}\,{\bf b} we then get
\[
\Phi_\lambda^{\chi_{\nu_\ell}\otimes\chi_{\nu_r}}(a)=\sum_{k=0}^{\infty}\psi_{\lambda;\nu_\ell}(v_{\lambda;\nu_r}[\lambda-2k])
a^{\lambda-2k}
\]
for $a\in \mathbb{R}_{>1}$. The result now follows from Lemma \ref{r1}.
\end{proof}
Combined with Proposition \ref{GaussExpres}, we reobtain the following special case of the
Poisson kernel identity \cite[\S 2.5.2 (12)]{Er} for Meixner-Pollaczek polynomials.
\begin{cor}
We have for $a\in\mathbb{R}_{>1}$,
\begin{equation*}
\begin{split}
a^\lambda\sum_{k=0}^{\infty}&
\frac{4^{k}p_k^{(-\lambda/2)}(-\nu_\ell)p_k^{(-\lambda/2)}(-\nu_r)}{(-\lambda)_kk!}
(-a^{-2})^k=\\
&=(a+a^{-1})^\lambda\left(\frac{a-a^{-1}}{a+a^{-1}}\right)^{i(\nu_\ell+\nu_r)}
{}_2F_1\Bigl(-\frac{\lambda}{2}+i\nu_\ell, -\frac{\lambda}{2}+i\nu_r; -\lambda \,\, | \,\, 
\frac{4}{(a+a^{-1})^2}\Bigr).
\end{split}
\end{equation*}
\end{cor}

\begin{rema}
For a different representation theoretic interpretation of the Poisson kernel identity for Meixner-Pollaczek polynomials, see \cite[Prop. 2.1]{KJ}.
\end{rema}

\section{$N$-point spherical functions}\label{SectionbKZB}

\subsection{Factorisations of the Casimir element}\label{S61}

Let $M,M^\prime,U$ be $\mathfrak{g}$-modules. We will call $\mathfrak{g}$-intertwiners 
$M\rightarrow M^\prime\otimes U$ vertex operators.
This terminology is stretching the standard 
representation theoretic notion of vertex operators as commonly used in the context of Wess-Zumino-Witten conformal field theory. In that case (see, e.g., \cite{JM}), it refers to intertwiners $M\rightarrow M^\prime\otimes U(z)$ for affine Lie algebra representations $M,M^\prime$ and $U(z)$ with $M$ and $M^\prime$ highest weight representations (playing the role of auxiliary spaces), and $U(z)$ an evaluation representation (playing the role of state space).  

The space $\textup{Hom}(M,M^\prime\otimes U)$ of all linear maps $M\rightarrow M^\prime\otimes U$
admits the following left and right $U(\mathfrak{g})^{\otimes 2}$-action,
\begin{equation*}
\begin{split}
\bigl((x\otimes y)\Psi\bigr)(m):=(x\otimes y)\Psi(m),\\
\bigl(\Psi\ast (x\otimes y)\bigr)(m):=(1\otimes S(x))\Psi(ym)
\end{split}
\end{equation*}
for $x,y\in U(\mathfrak{g})$, $\Psi\in\textup{Hom}(M,M^\prime\otimes U)$ and $m\in M$, with $S$ the antipode of $U(\mathfrak{g})$. Here we suppress the representation maps if no confusion can arise. We sometimes also write $x_M$ 
for the action of $x\in U(\mathfrak{g})$ on the $\mathfrak{g}$-module $M$.
\begin{defi}
We say that a triple $(\tau^\ell,\tau^r,d)$ with $\tau^\ell,\tau^r\in\mathfrak{g}\otimes\mathfrak{g}$
and $d\in U(\mathfrak{g})$ is a factorisation of the Casimir element $\Omega\in Z(\mathfrak{g})$
if for all $\mathfrak{g}$-modules $M,M^\prime,U$ and for all vertex operators
$\Psi\in\textup{Hom}_{\mathfrak{g}}\bigl(M,M^\prime\otimes U\bigr)$ we have
\begin{equation}\label{operatorKZBpre}
\frac{1}{2}\bigl((\Omega\otimes 1)\Psi-\Psi\Omega\bigr)=\tau^\ell\Psi-\Psi\ast \tau^r+(1\otimes d)\Psi
\end{equation}
in $\textup{Hom}(M,M^\prime\otimes U)$. 
\end{defi}
Suppose that $(\tau^\ell,\tau^r,d)$ is a factorisation of $\Omega$.
If $\Omega_{M}=\zeta_M\textup{id}$ and $\Omega_{M^\prime}=\zeta_{M^\prime}\textup{id}_{M^\prime}$ for some constants $\zeta_M,\zeta_{M^\prime}\in\mathbb{C}$ then we arrive at the {\it asymptotic operator Knizhnik-Zamolodchikov-Bernard (KZB) equation} for 
 vertex operators $\Psi\in\textup{Hom}_{\mathfrak{g}}(M,M^\prime\otimes U)$,
\begin{equation}\label{operatorKZB}
\frac{1}{2}(\zeta_{M^\prime}-\zeta_M)\Psi=\tau^\ell\Psi-\Psi\ast\tau^r+(1\otimes d)\Psi
\end{equation}
(compare with the operator KZ equation from \cite[Thm. 2.1]{FR}).

Consider now vertex operators $\Psi_i\in\textup{Hom}_{\mathfrak{g}}(M_i,M_{i-1}\otimes U_i)$
for $i=1,\ldots,N$. Denote $\mathbf{U}:=U_1\otimes U_2\otimes\cdots\otimes U_N$ and write
$\mathbf{\Psi}\in\textup{Hom}_{\mathfrak{g}}(M_N,M_0\otimes\mathbf{U})$ for the composition
\begin{equation}\label{multiplevertex}
\mathbf{\Psi}:=(\Psi_1\otimes\textup{id}_{U_2\otimes\cdots\otimes U_N})
\cdots (\Psi_{N-1}\otimes\textup{id}_{U_N})\Psi_N
\end{equation}
of the vertex operators $\Psi_i$.
Assume that $\Omega_{M_i}=\zeta_{M_i}\textup{id}_{M_i}$ for some constants $\zeta_{M_i}\in\mathbb{C}$
($0\leq i\leq N$). The asymptotic operator KZB equation \eqref{operatorKZB} now extends to the following system of equations for $\mathbf{\Psi}$.
\begin{cor}\label{bulkKZB}
Let $(\tau^\ell,\tau^r,d)$ be a factorisation of $\Omega$ with expansions $\tau^\ell=\sum_k \alpha_k^\ell\otimes \beta_k^\ell$
and $\tau^r=\sum_m\alpha_m^r\otimes\beta_m^r$ in $\mathfrak{g}\otimes\mathfrak{g}$.
Under the above assumptions and conventions we have
\begin{equation}\label{multioperatorKZB}
\begin{split}
\frac{1}{2}(\zeta_{M_{i-1}}-\zeta_{M_i})\mathbf{\Psi}&=\Bigl(\sum_{j=1}^{i-1}\tau^\ell_{U_jU_i}-\sum_{j=i+1}^N
\tau^r_{U_iU_j}+d_{U_i}\Bigr)\mathbf{\Psi}\\
&+\sum_{k}(\alpha_k^\ell)_{M_0}(\beta_k^\ell)_{U_i}\mathbf{\Psi}
+\sum_m(\alpha_m^r)_{U_i}\mathbf{\Psi}(\beta_m^r)_{M_N}
\end{split}
\end{equation}
for $i=1,\ldots,N$.
\end{cor}
\begin{proof}
Write $\Psi_{M_i}:=\Psi_i\otimes\textup{id}_{U_{i+1}\otimes\cdots\otimes U_N}$. Then
\begin{equation}\label{Fstep}
\frac{1}{2}(\zeta_{M_{i-1}}-\zeta_{M_i})\mathbf{\Psi}=\frac{1}{2}\Psi_{M_1}\cdots\Psi_{M_{i-1}}\bigl(\Omega_{M_{i-1}}\Psi_{M_i}-
\Psi_{M_i}\Omega_{M_i}\bigr)\Psi_{M_{i+1}}\cdots\Psi_{M_N}.
\end{equation}
Now  \eqref{operatorKZBpre} gives 
\[
\frac{1}{2}\bigl(\Omega_{M_{i-1}}\Psi_{M_i}-
\Psi_{M_i}\Omega_{M_i}\bigr)=\sum_k(\alpha_k^\ell)_{M_{i-1}}(\beta_k^\ell)_{U_i}\Psi_{M_i}+
\sum_m(\alpha_m^r)_{U_i}\Psi_{M_i}(\beta_m^r)_{M_i}+d_{U_i}\Psi_{M_i}.
\]
Substitute this equation in \eqref{Fstep} and 
push the action of $\alpha_k^\ell$ on $M_{i-1}$ (resp. the action of $\beta_m^r$ on $M_i$)
through the product $\Psi_{M_1}\cdots\Psi_{M_{i-1}}$
(resp. $\Psi_{M_{i+1}}\cdots\Psi_{M_N}$) of vertex operators
using the fact that
\[
(x\otimes 1)\Psi-\Psi x=-(1\otimes x)\Psi
\]
for $x\in\mathfrak{g}$ and $\Psi\in\textup{Hom}_{\mathfrak{g}}(M,M^\prime\otimes U)$.
This immediately results in \eqref{multioperatorKZB}.
\end{proof}

The asymptotic operator KZB equations \eqref{multioperatorKZB} for an appropriate factorisation of $\Omega$ give rise to boundary KZB type equations that are solved by asymptotical $N$-point correlation functions for boundary WZW conformal field theory on a cylinder. Here asymptotical means that the "positions" of the local observables in the correlation functions escape to infinity. We will define the asymptotical $N$-point correlation functions directly in Subsections \ref{S63} \& \ref{intertwinersection}, and call them (formal) $N$-point spherical functions. The discussion how they arise as limits of correlation functions is postponed to a followup paper. 

We give now first two families of examples of factorisations of $\Omega$.
The first family is related to the expression \eqref{Aomega} of $\Omega$. It leads to asymptotic KZB equations for generalised weighted trace functions (see \cite{ES}). As we shall see later, this family also gives rise to the asymptotic boundary KZB equations for the (formal) $N$-point spherical functions using a reflection argument. The second family is related to the Cartan decomposition \eqref{KAKomega} of $\Omega$, and leads directly to the asymptotic boundary KZB equations.
This second derivation of the asymptotic boundary KZB equations is expected to be crucial for the generalisation to quantum groups.

Felder's  \cite{F}, \cite[\S 2]{ES} trigonometric dynamical $r$-matrix $r\in\mathcal{R}\otimes \mathfrak{g}^{\otimes 2}$
is given by
\begin{equation}\label{Felderr}
r:=-\frac{1}{2}\sum_{j=1}^\rr x_j\otimes x_j-\sum_{\alpha\in R}\frac{e_{-\alpha}\otimes e_\alpha}{1-\xi_{-2\alpha}}.
\end{equation}
Set 
\[
r^{\theta_1}:=(\theta\otimes \textup{id}_{\mathfrak{g}})r=
\frac{1}{2}\sum_{j=1}^\rr x_j\otimes x_j+\sum_{\alpha\in R}\frac{e_\alpha\otimes e_\alpha}
{1-\xi_{-2\alpha}}.
\]
For $s=\sum_is_i\otimes t_i\in\mathfrak{g}\otimes\mathfrak{g}$ we write $s_{21}:=
\sum_it_i\otimes s_i$ and $s_{21}^{\theta_2}:=(1\otimes \theta)s_{21}$. Note that
$r^{\theta_1}$ is a symmetric tensor,
\[
r_{21}^{\theta_2}=r^{\theta_1}.
\]
We will write below $r_{21}^{\theta_2}$ for the occurrences of the $\theta$-twisted $r$-matrices in the asymptotic boundary KZB equations, since this is natural when viewing the asymptotic boundary KZB equations as formal limit of integrable boundary qKZB equations (this will be discussed in future work).

Define folded $r$-matrices by
\[
r^{\pm}:=\pm r+r_{21}^{\theta_2}.
\]
Note that $r^+\in\mathcal{R}\otimes\mathfrak{k}\otimes\mathfrak{g}$ and $r^-\in
\mathcal{R}\otimes\mathfrak{p}\otimes \mathfrak{g}$.
The folded $r$-matrices $r^{\pm}$ are explicitly given by
\begin{equation}\label{explicitsigmatauexpl}
\begin{split}
r^+&=\sum_{\alpha\in R}\frac{y_\alpha\otimes e_\alpha}{1-\xi_{-2\alpha}},\\
r^-&=\sum_{j=1}^\rr x_j\otimes x_j+\sum_{\alpha\in R}\frac{(e_\alpha+e_{-\alpha})\otimes e_\alpha}
{1-\xi_{-2\alpha}}.
\end{split}
\end{equation}
\begin{prop}\label{propfactorization}
Fix $a\in A_{\textup{reg}}$. The following triples $(\tau^\ell,\tau^r,d)$ give factorisations of the
Casimir $\Omega\in Z(\mathfrak{g})$.
\begin{enumerate}
\item[{\bf a.}] 
$\tau^\ell=\tau^r=r(a)$, $d=-\frac{1}{2}\sum_{\alpha\in R^+}\Bigl(\frac{1+a^{-2\alpha}}{1-a^{-2\alpha}}
\Bigr)t_\alpha$.
\item[{\bf b.}] $\tau^\ell=r^+(a)$, $\tau^r=-r^-(a)$, $d=b(a)$
with
\begin{equation}\label{kappa}
b:=\frac{1}{2}\sum_{j=1}^\rr x_j^2-\frac{1}{2}\sum_{\alpha\in R^+}
\Bigl(\frac{1+\xi_{-2\alpha}}{1-\xi_{-2\alpha}}\Bigr)t_\alpha+\sum_{\alpha\in R}
\frac{e_\alpha^2}{1-\xi_{-2\alpha}}\in\mathcal{R}\otimes U(\mathfrak{g}).
\end{equation}
\end{enumerate}
\end{prop}
\begin{proof}
The factorisations are obtained from the explicit expressions
 \eqref{Aomega} and 
\eqref{KAKomega} of $\Omega$ by moving Lie algebra elements in the resulting expression of $\frac{1}{2}\bigl((\Omega\otimes 1)\Psi-\Psi\Omega\bigr)$ through the vertex operator $\Psi$ 
following a particular (case-dependent) strategy. The elementary formulas we need are
\begin{equation}\label{firsteqn}
\begin{split}
(x\otimes 1)\Psi-\Psi x&=-(1\otimes x)\Psi,\\
(xy\otimes 1)\Psi-\Psi xy&=-(1\otimes x)\Psi y-(x\otimes y)\Psi
\end{split}
\end{equation}
for $x,y\in \mathfrak{g}$ and $\Psi\in\textup{Hom}_{\mathfrak{g}}(M,M^\prime\otimes U)$. Note that the second formula gives an expression of $(xy\otimes 1)\Psi-\Psi xy$
with $x$ no longer acting on $M$ and $y$ no longer acting on $M^\prime$. For case {\bf b} we also need formulas such that both $x$ and $y$ are not acting on $M^\prime$ (resp. on $M$),
\begin{equation}\label{secondeqn}
\begin{split}
(xy\otimes 1)\Psi-\Psi xy&=-(1\otimes x)\Psi y-(1\otimes y)\Psi x+(1\otimes yx)\Psi,\\
(xy\otimes 1)\Psi-\Psi xy&=-(x\otimes y)\Psi-(y\otimes x)\Psi-(1\otimes xy)\Psi
\end{split}
\end{equation}
for $x,y\in\mathfrak{g}$ and $\Psi\in\textup{Hom}_{\mathfrak{g}}(M,M^\prime\otimes U)$.
These equations are easily obtained by combining the two formulas of \eqref{firsteqn}.\\
{\bf a.} Substitute \eqref{Aomega} into $\frac{1}{2}((\Omega\otimes 1)\Psi-\Psi\Omega)$ and apply \eqref{firsteqn} to the terms. 
The resulting identity can be written as $\tau^\ell\Psi-\Psi\ast\tau^r+
(1\otimes d)\Psi$ with $(\tau^\ell,\tau^r,d)$ as stated.\\
{\bf b.} Use that
\[
\frac{1}{2}\bigl(
(\Omega\otimes 1)\Psi-\Psi\Omega\bigr)=
\frac{1}{2}\bigl((\textup{Ad}_a(\Omega)\otimes 1)\Psi-\Psi\textup{Ad}_a(\Omega)\bigr)
\]
and substitute \eqref{KAKomega} in the right hand side of this equation. For the quadratic terms $xy$ ($x,y\in\mathfrak{g}$) in the resulting formula we use the second formula of \eqref{firsteqn} when $x\in\mathfrak{k}$ and $y\in\textup{Ad}_a(\mathfrak{k})$, the first formula of \eqref{secondeqn} when both 
$x,y\in\textup{Ad}_a(\mathfrak{k})$ or both $x,y\in\mathfrak{h}$, and the second formula of \eqref{secondeqn} when both $x,y\in\mathfrak{k}$. It results in the formula \eqref{operatorKZBpre}
with $(\tau^\ell,\tau^r,d)$ given by
\begin{equation*}
\begin{split}
\tau^\ell&=\frac{1}{2}\sum_{\alpha\in R}y_\alpha\otimes\frac{((a^\alpha+a^{-\alpha})\textup{Ad}_a(y_\alpha)-2y_\alpha)}{(a^\alpha-a^{-\alpha})^2},\\
\tau^r&=-\sum_{j=1}^\rr x_j\otimes x_j+\frac{1}{2}\sum_{\alpha\in R}
\frac{(a^\alpha+a^{-\alpha})y_\alpha-2\textup{Ad}_a(y_\alpha)}{(a^\alpha-a^{-\alpha})^2}
\otimes\textup{Ad}_a(y_\alpha),\\
d&=\frac{1}{2}\sum_{j=1}^\rr x_j^2-\frac{1}{2}\sum_{\alpha\in R^+}\Bigl(\frac{1+a^{-2\alpha}}{1-a^{-2\alpha}}
\Bigr)t_\alpha+\frac{1}{2}\sum_{\alpha\in R}\frac{(\textup{Ad}_a(y_\alpha^2)-y_\alpha^2)}
{(a^\alpha-a^{-\alpha})^2}.
\end{split}
\end{equation*}
By the elementary identities 
\begin{equation*}
\begin{split}
2\textup{Ad}_a(y_\alpha)-(a^\alpha+a^{-\alpha})y_\alpha&=(a^\alpha-a^{-\alpha})(e_\alpha+e_{-\alpha}),\\
\textup{Ad}_a(y_\alpha^2)-y_\alpha^2&=(a^{2\alpha}-1)e_\alpha^2+(a^{-2\alpha}-1)e_{-\alpha}^2
\end{split}
\end{equation*}
the above expressions for $\tau^\ell, \tau^r$ and $d$ simplify to the expressions 
\eqref{explicitsigmatauexpl} and \eqref{kappa} for $r^+(a),-r^-(a)$ and $b(a)$.
\end{proof}
\begin{rema}\label{remarkfactorization}
The limit ${}^{\infty}r:=\lim_{a\rightarrow\infty}r(a)$ (meaning $a^{\alpha}\rightarrow\infty$ for all $\alpha\in R^+$) gives the classical $r$-matrix
\[
{}^{\infty}r=-\frac{1}{2}\sum_{j=1}^\rr x_j\otimes x_j-\sum_{\alpha\in R^+}e_{-\alpha}\otimes e_\alpha
\in\theta(\mathfrak{b})\otimes\mathfrak{b}.
\]
The corresponding limit for the folded $r$-matrices $r^{\pm}(a)$ gives
\[
{}^{\infty}r^{\pm}:=\pm {}^{\infty}r+{}^{\infty}r_{21}^{\theta_2}.
\]
As a consequence of Proposition \ref{propfactorization}
we then obtain the following two (nondynamical) factorisations of the Casimir $\Omega$,
\begin{enumerate}
\item[{\bf a.}] $(\sigma,\tau,d)=({}^{\infty}r,{}^{\infty}r,-2t_\rho)$.
\item[{\bf b.}] $(\sigma,\tau,d)=({}^{\infty}r^+,-{}^{\infty}r^-,{}^{\infty}b)$ with 
\begin{equation}
{}^{\infty}b:=\frac{1}{2}\sum_{j=1}^\rr x_j^2-t_\rho+\sum_{\alpha\in R^+}e_\alpha^2\in U(\mathfrak{g}).
\end{equation}
\end{enumerate}
\end{rema}

\subsection{Differential vertex operators}\label{S62}
In Subsection \ref{S63} we apply the results of the previous subsection to $M_i=\mathcal{H}_{\lambda_i}^\infty$ with $\lambda_i\in\mathfrak{h}^*$ ($i=0,\ldots,N$) and to finite dimensional $G$-representations $U_j$ ($j=1,\ldots,N$). Before doing so, we first describe the 
appropriate class of vertex operators in this context, which consists of $G$-equivariant differential operators.
We use the notion of vector-valued $G$-equivariant differential operators between spaces of global sections of complex vector bundles, see \cite[Chpt. II]{He} as well as \cite[\S 1]{KR}.

Identify the $G$-space $\mathcal{H}_{\lambda}^{\infty}$ with the space of global smooth sections of the complex line bundle $\mathcal{L}_\lambda:=(G\times\mathbb{C})/\sim_\lambda$ over $G/AN_+\simeq K$, with equivalence relation
$\sim_\lambda$ given by 
\begin{equation}\label{sim}
(gb,\eta_{\lambda+\rho}(b^{-1})c)\sim_\lambda (g,c),\qquad  (g\in G, b\in AN_+, c\in\mathbb{C}).
\end{equation}
For $\lambda,\mu\in\mathfrak{h}^*$ and $U$ a finite dimensional $G$-representation let $\mathbb{D}(\mathcal{H}_{\lambda}^\infty,\mathcal{H}_{\mu}^{\infty}\otimes U)$ be the space of differential 
$G$-intertwiners $\mathcal{H}_{\lambda}^{\infty}\rightarrow\mathcal{H}_{\mu}^{\infty}\otimes U$
(note that $\mathcal{H}_\mu^\infty\otimes U$ is the space of smooth section of the vector bundle
$(G\times U)/\sim_\lambda$, with $\sim_\lambda$ given by the same formula \eqref{sim} with $c\in U$). We call $D\in\mathbb{D}(\mathcal{H}_\lambda^\infty,\mathcal{H}_\mu^{\infty}\otimes U)$ a {\it differential vertex operator}. 

Let $\mathbb{D}_0^\prime(\mathcal{L}_\lambda)$ be the $\mathfrak{g}$-module consisting of distributions on $\mathcal{L}_\lambda$ supported at $1\in K$. Note that $\mathbb{D}_0^\prime(\mathcal{L}_\lambda)$ is contained in the continuous linear dual of $\mathcal{H}_\lambda^\infty$. A straightforward adjustment of the proof of \cite[Lem. 2.4]{CS} yields a linear isomorphism
\[
\mathbb{D}(\mathcal{H}_{\lambda}^\infty,\mathcal{H}_{\mu}^{\infty}\otimes U)\simeq
\textup{Hom}_{\mathfrak{g}}(\mathbb{D}_0^\prime(\mathcal{L}_\mu),\mathbb{D}_0^\prime(\mathcal{L}_\lambda)\otimes U)
\]
via dualisation. Furthermore,
\[
M_{-\lambda-\rho}\simeq \mathbb{D}_0^\prime(\mathcal{L}_\lambda)
\]
as $\mathfrak{g}$-modules by Schwartz' theorem, with the distribution $\omega$ associated to 
$xm_{-\lambda-\rho}$
($x\in U(\mathfrak{g})$) defined by
\[
\omega(\phi):=\bigl(r_\ast(x)\phi\bigr)(1)\qquad (\phi\in\mathcal{H}_\lambda^\infty)
\] 
(see again the proof of \cite[Lem. 2.4]{CS}). We thus reach the following conclusion.
\begin{prop}\label{isoprop}
For $\lambda,\nu\in\mathfrak{h}^*$ and $U$ a finite dimensional $G$-representation we have
\begin{equation}\label{isodiff}
\mathbb{D}(\mathcal{H}_{\lambda}^\infty,\mathcal{H}_{\mu}^{\infty}\otimes U)\simeq
\textup{Hom}_{\mathfrak{g}}(M_{-\mu-\rho}, M_{-\lambda-\rho}\otimes U).
\end{equation}
\end{prop}

\begin{rema}\label{explicitinverse}
The inverse of the isomorphism \eqref{isodiff} can be described explicitly as follows. 
Fix a vertex operator $\Psi\in\textup{Hom}_{\mathfrak{g}}(M_{-\mu-\rho}, 
M_{-\lambda-\rho}\otimes U)$. 
Let $\{u_i\}_i$ be a linear basis of $U$ and write $\{u_i^*\}_i$ for its dual basis.  
Let $Y_i\in U(\mathfrak{n}_-)$ be the unique elements such that
\[
\Psi(m_{-\mu-\rho})=\sum_iY_im_{-\lambda-\rho}\otimes u_i,
\]
where \textup{(}recall\textup{)} $\mathfrak{n}_-=\theta(\mathfrak{n}_+)$ is the nilpotent subalgebra of $\mathfrak{g}$ opposite to $\mathfrak{b}$.
Under the isomorphism \eqref{isodiff}, the intertwiner $\Psi$ is mapped to the differential vertex operator 
\[
D_\Psi=\sum_iL_i\vert_{\mathcal{H}_\lambda^\infty}\otimes u_i\in\mathbb{D}(\mathcal{H}_\lambda^\infty,\mathcal{H}_\mu^\infty\otimes U)
\]
with $L_i$ the differential operators
\[
L_i:=\sum_ju_i^*(\cdot\,u_j)r_\ast(Y_j)
\]
on $G$.
\end{rema}

For 
$\Psi_V\in\textup{Hom}_{\mathfrak{g}}(M_{-\mu-\rho}, M_{-\lambda-\rho}\otimes V)$
and $\Psi_U\in\textup{Hom}_{\mathfrak{g}}(M_{-\nu-\rho}, M_{-\mu-\rho}\otimes U)$
set 
\begin{equation}\label{compositionvertex}
\begin{split}
\Psi_{V,U}:=(\Psi_V\otimes\textup{id}_U)\Psi_U&\in\textup{Hom}_{\mathfrak{g}}(M_{-\nu-\rho}, M_{-\lambda-\rho}\otimes V\otimes U),\\
D_{U,V}:=(D_{\Psi_U}\otimes\textup{id}_V)D_{\Psi_V}&\in\mathbb{D}(\mathcal{H}_\lambda^\infty,\mathcal{H}_\nu^\infty
\otimes U\otimes V).
\end{split}
\end{equation}
These two composition rules are compatible with the isomorphism from Proposition \ref{isoprop}:
\begin{prop}
Let $P_{UV}: U\otimes V\rightarrow V\otimes U$ be the $G$-linear isomorphism flipping the two tensor components. Then 
\[
D_{\Psi_{V,U}}=(\textup{id}_{\mathcal{H}_\nu^\infty}\otimes P_{UV})D_{U,V}
\]
 in $\mathbb{D}(\mathcal{H}_\lambda^\infty, \mathcal{H}_\nu^\infty\otimes V\otimes U)$.
 \end{prop}
\begin{proof}
This follows by a straightforward but lengthy computation using Remark \ref{explicitinverse}.
 \end{proof}

Next we consider the parametrisation of the spaces of vertex operators.
Write $m_\mu^*\in M_\mu^*$ for the linear functional satisfying $m_\mu^*(m_\mu)=1$ and $m_\mu^*(v)=0$ for $v\in\bigoplus_{\nu<\mu}M_\mu[\nu]$. 
\begin{defi}
Let $U$ be a finite dimensional $\mathfrak{g}$-module,
$\lambda,\mu\in\mathfrak{h}^*$, and $\Psi\in\textup{Hom}_{\mathfrak{g}}(M_\lambda,M_\mu\otimes U)$. Then
\[
\langle\Psi\rangle:=(m_{\mu}^*\otimes\textup{id}_U)\Psi(m_\lambda)\in U[\lambda-\mu]
\]
is called the expectation value of the vertex operator $\Psi$.
\end{defi}
Set $N_-:=\Theta_0(N_+)$. The multiplication map 
\[
N_-\times A\times N_+\rightarrow G
\]
defines a diffeomorphism onto the open connected submanifold $N_-AN_+$ of $G$. Write
$\widetilde{\mathcal{H}}_\lambda^\infty$ for the space of smooth functions $f: N_-AN_+\rightarrow\mathbb{C}$ satisfying $f(n_-an_+)=a^{-\lambda-\rho}f(n_-)$ for
$a\in A$ and $n_{\pm}\in N_{\pm}$. 

In the following lemma we express the expectation value of $\Psi\in\textup{Hom}_{\mathfrak{g}}(M_{-\mu-\rho},M_{-\lambda-\rho}\otimes U)$
in terms of the differential vertex operator $D_\Psi\in\mathbb{D}(\mathcal{H}_\lambda^\infty,\mathcal{H}_\mu^\infty\otimes U)$, which we now view
as the $U$-valued differential operator $\widetilde{\mathcal{H}}_\lambda^\infty\rightarrow\widetilde{\mathcal{H}}_\mu^\infty\otimes U$ mapping $f\in\widetilde{\mathcal{H}}_\lambda^\infty$ to $\sum_iL_i(f)u_i$
(see Remark \ref{explicitinverse} for the notations).
\begin{lem}
For $\Psi\in\textup{Hom}_{\mathfrak{g}}(M_{-\mu-\rho}, M_{-\lambda-\rho}\otimes U)$ we
have
\[
\bigl(D_\Psi\mathbb{I}_\lambda\bigr)(1)=\langle \Psi\rangle
\]
in $U[\lambda-\mu]$,
where $\mathbb{I}_\lambda\in\widetilde{\mathcal{H}}_\lambda^\infty$ is defined by
\[
\mathbb{I}_\lambda(n_-an_+):=a^{-\lambda-\rho}\qquad (a\in A,\, n_{\pm}\in N_{\pm}).
\]
\end{lem}
\begin{proof}
Using the notations from Remark \ref{explicitinverse}, we have
\[
\langle\Psi\rangle=\sum_i\epsilon(Y_i)u_i
\]
with $\epsilon$ the counit of $U(\mathfrak{n}_-)$. On the other hand,
\[
\bigl(D_\Psi\mathbb{I}_\lambda\bigr)(1)=\sum_i(L_i\mathbb{I}_\lambda)(1)u_i=
\sum_i\bigl(r_*(Y_i)\mathbb{I}_\lambda\bigr)(1)u_i=\sum_i\epsilon(Y_i)u_i,
\]
hence the result.
\end{proof}
By \cite[Lem. 3.3]{E} we have the following result.
\begin{lem}
\label{evlem}
Let $U$ be a finite dimensional $\mathfrak{g}$-module, $\lambda\in\mathfrak{h}^*$
and $\mu\in\mathfrak{h}_{\textup{irr}}^*$.
The expectation value map $\langle\cdot\rangle$ defines a linear isomorphism
\[
\langle\cdot\rangle: \textup{Hom}_{\mathfrak{g}}(M_\lambda,M_{\mu}\otimes U)\overset{\sim}{\longrightarrow} 
U[\lambda-\mu].
\] 
\end{lem}
The weights of a finite dimensional $\mathfrak{g}$-module $U$ lie in the integral weight lattice
\[
P:=\{\mu\in\mathfrak{h}^* \,\, | \,\, (\mu,\alpha^\vee)\in\mathbb{Z} \quad \forall\, \alpha\in R\}.
\]
Hence for $\mu\in\mathfrak{h}_{\textup{irr}}^*$, the space $\textup{Hom}_{\mathfrak{g}}(M_\lambda, M_\mu\otimes U)$ of vertex operators is trivial unless $\lambda\in \mu+P$. 
At a later stage (see Section \ref{sectionBFO}), we want to restrict to highest weights $\lambda_0\in\mathfrak{h}_{\textup{irr}}^*$ such that for any vertex operator $\mathbf{\Psi}\in\textup{Hom}_{\mathfrak{g}}(M_{\lambda_N},M_{\lambda_0}\otimes\mathbf{U})$, given as a product of vertex operators $\Psi_i\in\textup{Hom}_{\mathfrak{g}}(M_{\lambda_i},
M_{\lambda_{i-1}}\otimes U_i)$ ($i=1,\ldots,N$), has the property that $\lambda_{i-1}\in\mathfrak{h}_{\textup{irr}}^*$ for $i=1,\ldots,N$ (i.e., all vertex operators are determined by their expectation values). In that case we will restrict to highest weights from the dense open subset
\[
\mathfrak{h}_{\textup{reg}}^*:=\{\nu\in\mathfrak{h}^*\,\, | \,\, (\nu,\alpha^\vee)\not\in\mathbb{Z}\,\, \quad
\forall\, \alpha\in R\}
\]
of $\mathfrak{h}$. The (differential) vertex operators are then denoted as follows.
\begin{defi}\label{notVO}
Let $\lambda\in\mathfrak{h}_{\textup{reg}}^*$.
\begin{enumerate}
\item[{\bf a.}] If $U$ is a finite dimensional $\mathfrak{g}$-module and $u\in U[\lambda-\mu]$ is
a weight vector of weight $\lambda-\mu$, then we write $\Psi_\lambda^u\in\textup{Hom}_{\mathfrak{g}}(M_\lambda,M_{\mu}\otimes U)$ for the unique vertex operator with
expectation value $\langle\Psi_\lambda^u\rangle=u$. 
\item[{\bf b.}] If $U$ is a finite dimensional $G$-representation and $u\in U[\lambda-\mu]$ is 
a weight vector of weight $\lambda-\mu$, then we write
$D_\lambda^u\in\mathbb{D}(\mathcal{H}_\lambda^\infty,\mathcal{H}_\mu^\infty\otimes U)$
for the unique differential vertex operator with $(D_\lambda^u\mathbb{I}_\lambda)(1)=u$.
\end{enumerate}
\end{defi}
The expectation value of products of vertex operators gives rise to the fusion operator. We recall its definition in Subsection \ref{sectionBFO}, where we also discuss boundary versions of fusion operators.

\subsection{$N$-point spherical functions and asymptotic boundary KZB equations}\label{S63}
Fix finite dimensional $G$-representations $U_1,\ldots,U_N$ with representation maps $\tau_{U_1},\ldots,\tau_{U_N}$,
and differential vertex operators $D_i\in\mathbb{D}(\mathcal{H}_{\lambda_i}^\infty,
\mathcal{H}_{\lambda_{i-1}}^\infty\otimes U_i)$ for $i=1,\ldots,N$. Write $\underline{\lambda}=(\lambda_0,\lambda_1,\ldots,\lambda_N)$ and
$\mathbf{U}:=U_1\otimes\cdots\otimes U_N$. Write $\mathbf{D}\in\mathbb{D}(\mathcal{H}_{\lambda_N}^\infty,\mathcal{H}_{\lambda_0}^{\infty}\otimes\mathbf{U})$ for the product of the $N$
differential vertex operators $D_i$ ($1\leq i\leq N$),
\[
\mathbf{D}=(D_1\otimes\textup{id}_{U_2\otimes\cdots\otimes U_N})\cdots
(D_{N-1}\otimes\textup{id}_{U_N})D_N,
\]
which we call a differential vertex operator of weight $\underline{\lambda}$.

Fix two finite dimensional $K$-representations $V_\ell$ and $V_r$, with representation maps
$\sigma_\ell$ and $\sigma_r$ respectively. Let $\sigma_\ell^{(N)}$ be the representation map of the
tensor product $K$-representation $V_\ell\otimes\mathbf{U}$. We consider
$(V_\ell\otimes\mathbf{U})\otimes V_r^*\simeq\textup{Hom}(V_r,V_\ell\otimes\mathbf{U})$ as $K\times K$-representation, with representation map $\sigma^{(N)}:=\sigma_\ell^{(N)}\otimes\sigma_r^*$.
Note that if $\phi_\ell\in\textup{Hom}_K(\mathcal{H}_{\lambda_0},V_\ell)$ then 
\[
\bigl(\phi_\ell\otimes\textup{id}_{\mathbf{U}}\bigr)\mathbf{D}\in
\textup{Hom}_K(\mathcal{H}_{\lambda_N},V_\ell\otimes\mathbf{U})
\]
by \eqref{algvsanal}.

\begin{defi}\label{Nsphdef}
Let $\phi_\ell\in\textup{Hom}_K(\mathcal{H}_{\lambda_0},V_\ell)$,
$\phi_r\in\textup{Hom}_K(V_r,\mathcal{H}_{\lambda_N})$ and $\mathbf{D}$ a differential vertex operator of weight $\underline{\lambda}$.  We call the
elementary $\sigma^{(N)}$-spherical function
\begin{equation}\label{Sphericalfunctionwithinsertions}
f_{\mathcal{H}_{\underline{\lambda}}}^{\phi_\ell,\mathbf{D},\phi_r}(g):=(\phi_\ell\otimes\textup{id}_{\mathbf{U}})\mathbf{D}
(\pi_{\lambda_N}(g)\phi_r)\qquad (g\in G)
\end{equation}
a $N$-point $\sigma^{(N)}$-spherical function associated with
the $(N+1)$-tuple of principal series representations
$\mathcal{H}_{\underline{\lambda}}:=(\mathcal{H}_{\lambda_0},\ldots,\mathcal{H}_{\lambda_N})$.
\end{defi}
To keep the notations manageable we write from now on the action of $U(\mathfrak{g})$ and $G$ on $\mathcal{H}_{\lambda}^{\infty}$ without specifying the representation map if no confusion can arise. For instance, for $x\in U(\mathfrak{g})$, $g\in G$ and $v\in\mathcal{H}_{\lambda_N}^{\infty}$ we write
$gxv\in\mathcal{H}_{\lambda_N}^{\infty}$ for the smooth vector $\pi_{\lambda_N}(g)((x)_{\mathcal{H}_{\lambda_N}^{\infty}}v)$, and the $N$-point spherical function will be written as
\[
f_{\mathcal{H}_{\underline{\lambda}}}^{\phi_\ell,\mathbf{D},\phi_r}(g):=(\phi_\ell\otimes\textup{id}_{\mathbf{U}})\mathbf{D}
(g\phi_r)\qquad (g\in G).
\]

\begin{rema}
In Subsection \ref{intertwinersection} we define {\it formal} $N$-point spherical functions, which are asymptotical $N$-point correlation functions for boundary Wess-Zumino-Witten conformal field theory on the cylinder when the positions escape to infinity. The $N$-point spherical functions in Definition \ref{Nsphdef} are their analogues in the context of principal series. 
\end{rema}

By Proposition \ref{relEisPrin}\,{\bf c} the $N$-point spherical function $f_{\mathcal{H}_{\underline{\lambda}}}^{\phi_\ell,\mathbf{D},\phi_r}$ admits the Eisenstein type integral representation
\begin{equation}\label{EisN}
\begin{split}
f_{\mathcal{H}_{\underline{\lambda}}}^{\phi_\ell,\mathbf{D},\phi_r}(g)=
\int_Kdx\, \xi_{-\lambda_N-\rho}(a(g^{-1}x))&\Bigl(\sigma_\ell(x)\otimes\tau_{U_1}(x)\otimes\cdots\\
&\quad\cdots\otimes
\tau_{U_N}(x)\otimes\sigma_r^*(k(g^{-1}x))\Bigr)
T_{\lambda_N}^{(\phi_\ell\otimes\textup{id}_{\mathbf{U}})\mathbf{D}, \phi_r}
\end{split}
\end{equation}
with the vector $T_{\lambda_N}^{(\phi_\ell\otimes\textup{id}_{\mathbf{U}})\mathbf{D}, \phi_r}\in V_\ell\otimes\mathbf{U}\otimes V_r^*$ given by
\[
T_{\lambda_N}^{(\phi_\ell\otimes\textup{id}_{\mathbf{U}})\mathbf{D}, \phi_r}=
\iota_{\lambda_N,V_\ell\otimes\mathbf{U}}((\phi_\ell\otimes\textup{id}_{\mathbf{U}})\mathbf{D})\otimes\j_{\lambda_N,V_r}(\phi_r).
\]
Theorem \ref{thmRAD}\,{\bf a} gives the family of differential equations
\begin{equation}\label{diffN}
\widehat{\Pi}^{\sigma^{(N)}}(z)\bigl(f_{\mathcal{H}_{\underline{\lambda}}}^{\phi_\ell,\mathbf{D},\phi_r}|_{A_{\textup{reg}}}\bigr)=
\zeta_{\lambda_N-\rho}(z)f_{\mathcal{H}_{\underline{\lambda}}}^{\phi_\ell,\mathbf{D},\phi_r}|_{A_{\textup{reg}}},
\qquad z\in Z(\mathfrak{g})
\end{equation}
for the restriction of $f_{\mathcal{H}_{\underline{\lambda}}}^{\phi_\ell,\mathbf{D},\phi_r}$ to $A_{\textup{reg}}$.
We will now show that 
$f_{\mathcal{H}_{\underline{\lambda}}}^{\phi_\ell,\mathbf{D},\phi_r}|_{A_{\textup{reg}}}$ satisfies $N$ additional first order asymptotic boundary KZB type differential equations.
Recall the factorisation $(r^+(a),-r^-(a),b(a))$ of $\Omega$ for $a\in A_{\textup{reg}}$, with $r^{\pm}$ the folded $r$-matrices \eqref{explicitsigmatauexpl} and $b$ given by \eqref{kappa}. 
\begin{prop}\label{bKZBresult}
The $N$-point $\sigma^{(N)}$-spherical function 
$f_{\mathcal{H}_{\underline{\lambda}}}^{\phi_\ell,\mathbf{D},\phi_r}$ satisfies
\begin{equation}\label{bKZBungauged}
\begin{split}
\Bigl(&\frac{(\lambda_{i-1},\lambda_{i-1})}{2}-\frac{(\lambda_i,\lambda_i)}{2}+
\sum_{j=1}^\rr(x_j)_{U_i}\partial_{x_j}\Bigr)f_{\mathcal{H}_{\underline{\lambda}}}^{\phi_\ell,\mathbf{D},\phi_r}|_{A_{\textup{reg}}}=\\
&\quad\,=\Bigl(
r_{V_\ell U_i}^++\sum_{j=1}^{i-1}r_{U_jU_i}^++b_{U_i}+\sum_{j=i+1}^Nr_{U_iU_j}^-
+\widetilde{r}^+_{U_iV_r^*}\Bigr)
f_{\mathcal{H}_{\underline{\lambda}}}^{\phi_\ell,\mathbf{D},\phi_r}|_{A_{\textup{reg}}}
\end{split}
\end{equation}
for $i=1,\ldots,N$,
with right boundary term
\begin{equation}\label{rtildeplus}
\widetilde{r}^+:=\sum_{\alpha\in R}
\frac{e_{\alpha}\otimes y_\alpha}{\xi_\alpha-\xi_{-\alpha}}\in\mathcal{R}\otimes\mathfrak{g}\otimes\mathfrak{k}
\end{equation}
satisfying $\widetilde{r}^+(a)=(\textup{Ad}_{a^{-1}}\otimes 1)r_{21}^+(a)$ for $a\in A_{\textup{reg}}$.
\end{prop}
We derive the asymptotic boundary KZB type equations \eqref{bKZBungauged} in two different ways. The first proof uses Proposition \ref{propfactorization}\,{\bf b} involving the folded versions of Felder's dynamical $r$-matrix, the second proof uses Proposition \ref{propfactorization}\,{\bf a} with a reflection argument. The second argument is of interest from the conformal field theoretic point of view,
and provides some extra insights in the term $b$ \eqref{kappa} appearing in the 
asymptotic boundary KZB equations.\\  
{\bf Proof 1} (using the factorisation of $\Omega$ in terms of folded $r$-matrices).\\
Let $a\in A_{\textup{reg}}$. By Corollary \ref{bulkKZB} applied to the factorisation $(r^+(a), -r^-(a), b(a))$ of $\Omega$,
we have
\begin{equation*}
\begin{split}
\frac{1}{2}(\zeta_{\lambda_{i-1}-\rho}(\Omega)-\zeta_{\lambda_i-\rho}(\Omega))f_{\mathcal{H}_{\underline{\lambda}}}^{\phi_\ell,\mathbf{D},\phi_r}(a)&=
\Bigl(\sum_{j=1}^{i-1}r_{U_jU_i}^+(a)+b_{U_i}(a)+\sum_{j=i+1}^Nr_{U_iU_j}^-(a)\Bigr)
f_{\mathcal{H}_{\underline{\lambda}}}^{\phi_\ell,\mathbf{D},\phi_r}(a)\\
&+\sum_k\bigl(\phi_\ell\otimes\textup{id}_{\mathbf{U}}\bigr)\bigl((\alpha_k^+)_{\mathcal{H}_{\lambda_0}^{\infty}}(\beta_k^+)_{U_i}
\mathbf{D}(a\phi_r)\bigr)\\
&-\sum_k\bigl(\phi_\ell\otimes\textup{id}_{\mathbf{U}}\bigr)\bigl((\alpha_k^-)_{U_i}\mathbf{D}
(\beta_k^-a\phi_r)\bigr)
\end{split}
\end{equation*}
where we have written $r^{\pm}(a)=\sum_k\alpha_k^{\pm}\otimes\beta_k^{\pm}$. Now using
$r^+(a)\in\mathfrak{k}\otimes\mathfrak{g}$ and 
\[
\bigl(1\otimes\textup{Ad}_{a^{-1}})r^-(a)=
\sum_{j=1}^\rr x_j\otimes x_j+\widetilde{r}^+(a)
\]
with $\widetilde{r}^+(a)\in\mathfrak{g}\otimes\mathfrak{k}$ given by \eqref{rtildeplus}, 
the asymptotic boundary KZB type equation
\eqref{bKZBungauged} follows from the fact that $\phi_\ell$ and $\phi_r$ are $K$-intertwiners
and $\zeta_{\lambda_{i-1}-\rho}(\Omega)-\zeta_{\lambda_i-\rho}(\Omega)=
(\lambda_{i-1},\lambda_{i-1})-(\lambda_i,\lambda_i)$.\\
{\bf Proof 2} (using a reflection argument).\\
Let $a\in A_{\textup{reg}}$. Recall that the unfolded factorisation of $\Omega$ is $(r(a), r(a), d(a))$ with
\[
d(a):=-\frac{1}{2}\sum_{\alpha\in R^+}\Bigl(\frac{1+a^{-2\alpha}}{1-a^{-2\alpha}}\Bigr)t_\alpha.
\]
Then it follows from a direct computation that
\begin{equation}\label{kappaalt}
b(a)=d(a)+m(r^{\theta_1}(a)).
\end{equation}
Furthermore,
\begin{equation}\label{twistedcomm}
-r^{\theta_1}(a)D-D\ast r^{\theta_1}(a)=(1\otimes m(r^{\theta_1}(a)))D
\end{equation}
for a differential vertex operator $D$.
This follows from the fact that $r^{\theta_1}(a)$ is a symmetric tensor in
$\mathfrak{g}\otimes\mathfrak{g}$ and
\[
-(x\otimes x)D-D\ast (x\otimes x)=(1\otimes x^2)D
\]
for $x\in\mathfrak{g}$. The proof of the asymptotic boundary KZB equation \eqref{bKZBungauged} using a reflection argument now proceeds as follows. Corollary \ref{bulkKZB} gives 
\begin{equation*}
\begin{split}
\Bigl(\frac{(\lambda_{i-1},\lambda_{i-1})}{2}-&\frac{(\lambda_i,\lambda_i)}{2}\Bigr)f_{\mathcal{H}_{\underline{\lambda}}}^{\phi_\ell,\mathbf{D},\phi_r}(a)=
\Bigl(\sum_{j=1}^{i-1}r_{U_jU_i}(a)-\sum_{j=i+1}^Nr_{U_iU_j}(a)+d_{U_i}(a)\Bigr)
f_{\mathcal{H}_{\underline{\lambda}}}^{\phi_\ell,\mathbf{D},\phi_r}(a)\\
&+\bigl(\phi_\ell\otimes\textup{id}_{\mathbf{U}}\bigr)\bigl(r_{\mathcal{H}_{\lambda_0}^{\infty}U_i}(a)\mathbf{D}(a\phi_r)\bigr)
+\sum_k\bigl(\phi_\ell\otimes\textup{id}_{\mathbf{U}}\bigr)((\alpha_k)_{U_i}
\mathbf{D}(\beta_ka\phi_r))
\end{split}
\end{equation*}
with $r(a)=\sum_k\alpha_k\otimes\beta_k$. We now apply the identity
\[
\bigl(\phi_\ell\otimes\textup{id}_{\mathbf{U}}\bigr)r_{\mathcal{H}_{\lambda_0}^{\infty}U_i}(a)=
r_{V_\ell U_i}^+(a)\bigl(\phi_\ell\otimes\textup{id}_{\mathbf{U}}\bigr)-
\bigl(\phi_\ell\otimes\textup{id}_{\mathbf{U}}\bigr)r^{\theta_1}_{\mathcal{H}_{\lambda_0}^{\infty}U_i}(a)
\]
in $\textup{Hom}(\mathcal{H}_{\lambda_0}^{\infty}\otimes\mathbf{U},V_\ell\otimes\mathbf{U})$ to the left boundary term and 
the identity 
\begin{equation*}
\begin{split}
\sum_k(\alpha_k)_{U_i}\mathbf{D}(\beta_ka\phi_r)&=-\sum_{j=1}^\rr (x_j)_{U_i}\mathbf{D}(ax_j\phi_r)\\
&+\sum_k(\theta(\alpha_k))_{U_i}\mathbf{D}(\beta_ka\phi_r)+
\widetilde{r}^+_{U_iV_r^*}(a)\mathbf{D}(a\phi_r)
\end{split}
\end{equation*}
in $\mathbf{U}\otimes V_r^*$ to the right boundary term. The latter equality follows
from the easily verified identities
\begin{equation*}
\begin{split}
-\bigl(\textup{id}\otimes \theta\textup{Ad}_{a^{-1}}\bigr)(r(a))&=
-\sum_{j=1}^\rr x_j\otimes x_j+\bigl(\textup{id}\otimes\textup{Ad}_{a^{-1}}\bigr)(r^{\theta_1}(a)),\\
-\bigl(\textup{Ad}_{a^{-1}}\otimes\textup{id}\bigr)(r_{21}^+(a))&=\bigl(\textup{id}\otimes\textup{Ad}_{a^{-1}}\bigr)(r(a))+\bigl(\textup{id}\otimes \theta\textup{Ad}_{a^{-1}}\bigr)(r(a)).
\end{split}
\end{equation*}
We thus arrive at the formula
\begin{equation}\label{tttt}
\begin{split}
\Bigl(&\frac{(\lambda_{i-1},\lambda_{i-1})}{2}-\frac{(\lambda_i,\lambda_i)}{2}+\sum_{j=1}^\rr(x_j)_{U_i}\partial_{x_j}\Bigr)f_{\mathcal{H}_{\underline{\lambda}}}^{\phi_\ell,\mathbf{D},\phi_r}(a)=\\
&=\Bigl(r_{V_\ell U_i}^+(a)+\sum_{j=1}^{i-1}r_{U_jU_i}(a)+d_{U_i}(a)-\sum_{j=i+1}^Nr_{U_iU_j}(a)
+\widetilde{r}_{U_iV_r^*}^+(a)\Bigr)f_{\mathcal{H}_{\underline{\lambda}}}^{\phi_\ell,\mathbf{D},\phi_r}(a)\\
&-\bigl(\phi_\ell\otimes\textup{id}_{\mathbf{U}}\bigr)\bigl(r_{\mathcal{H}_{\lambda_0}^{\infty}U_i}^{\theta_1}(a)\mathbf{D}(a\phi_r)\bigr)
+\sum_k\bigl(\phi_\ell\otimes\textup{id}_{\mathbf{U}}\bigr)((\theta(\alpha_k))_{U_i}
\mathbf{D}(\beta_ka\phi_r)).
\end{split}
\end{equation}
Now pushing
$r^{\theta_1}_{\mathcal{H}_{\lambda_0}^{\infty}U_i}(a)$ through the differential vertex operators
$D_j$ ($1\leq j<i$) and pushing the action of $\beta_k$ through $D_j$ ($i<j\leq N$) using \eqref{firsteqn} and using the fact that $r^{\theta_1}(a)$ is a symmetric tensor in $\mathfrak{g}\otimes\mathfrak{g}$, the last line becomes
\begin{equation*}
\begin{split}
-\bigl(\phi_\ell\otimes\textup{id}_{\mathbf{U}}\bigr)&\bigl(r_{\mathcal{H}_{\lambda_0}^{\infty}U_i}^{\theta_1}(a)\mathbf{D}(a\phi_r)\bigr)
+\sum_k\bigl(\phi_\ell\otimes\textup{id}_{\mathbf{U}}\bigr)((\theta(\alpha_k))_{U_i}
\mathbf{D}(\beta_ka\phi_r))\\
&=\Bigl(\sum_{j=1}^{i-1}r^{\theta_1}_{U_jU_i}(a)+\sum_{j=i+1}^Nr^{\theta_1}_{U_iU_j}(a)
\Bigr)f_{\mathcal{H}_{\underline{\lambda}}}^{\phi_\ell,\mathbf{D},\phi_r}(a)\\
&+\bigl(\phi_\ell\otimes\textup{id}_{\mathbf{U}}\bigr)
(D_{\mathcal{H}_{\lambda_1}^{\infty}}\cdots D_{\mathcal{H}_{\lambda_{i-1}}^{\infty}}\bigl(\widetilde{D}_i\otimes\textup{id}_{U_{i+1}\otimes\cdots\otimes U_N}\bigr)
D_{\mathcal{H}_{\lambda_{i+1}}^{\infty}}\cdots D_{\mathcal{H}_{\lambda_N}^{\infty}}(a\phi_r))
\end{split}
\end{equation*}
with $D_{\mathcal{H}_{\lambda_j}^{\infty}}:=D_j\otimes\textup{id}_{U_{j+1}\otimes\cdots\otimes U_N}$
and
\[
\widetilde{D}_i:=-r^{\theta_1}(a)D_i-D_i\ast r^{\theta_1}(a).
\]
Applying now \eqref{twistedcomm} we arrive at
\begin{equation*}
\begin{split}
-\bigl(\phi_\ell\otimes\textup{id}_{\mathbf{U}}\bigr)&\bigl(r_{\mathcal{H}_{\lambda_0}^{\infty}U_i}^{\theta_1}(a)\mathbf{D}(a\phi_r)\bigr)
+\sum_k\bigl(\phi_\ell\otimes\textup{id}_{\mathbf{U}}\bigr)((\theta(\alpha_k))_{U_i}
\mathbf{D}(\beta_ka\phi_r))\\
&=\Bigl(\sum_{j=1}^{i-1}r^{\theta_1}_{U_jU_i}(a)+\bigl(m(r^{\theta_1}(a))\bigr)_{U_i}+
\sum_{j=i+1}^Nr^{\theta_1}_{U_iU_j}(a)
\Bigr)f_{\mathcal{H}_{\underline{\lambda}}}^{\phi_\ell,\mathbf{D},\phi_r}(a).
\end{split}
\end{equation*}
Combined with \eqref{tttt} and \eqref{kappaalt}, we obtain \eqref{bKZBungauged}.
\qed\\
\vspace{.2cm}\\

Write $\kappa^{\textup{core}}\in\mathcal{R}\otimes U(\mathfrak{g})$ for the element
\begin{equation}\label{kcore}
\kappa^{\textup{core}}:=\frac{1}{2}\sum_{j=1}^\rr x_j^2+\sum_{\alpha\in R}\frac{e_\alpha^2}{1-\xi_{-2\alpha}}
\end{equation}
and define $\kappa\in\mathcal{R}\otimes 
U(\mathfrak{k})\otimes U(\mathfrak{g})\otimes U(\mathfrak{k})$ by
\begin{equation}\label{kdef}
\kappa:=\sum_{\alpha\in R}\frac{y_\alpha\otimes e_\alpha\otimes 1}{1-\xi_{-2\alpha}}
+1\otimes \kappa^{\textup{core}}\otimes 1+\sum_{\alpha\in R}\frac{1\otimes e_\alpha\otimes y_\alpha}{\xi_\alpha-\xi_{-\alpha}}.
\end{equation}
Furthermore, write
\begin{equation}\label{E}
E:=\sum_{j=1}^\rr\partial_{x_j}\otimes x_j\in\mathbb{D}_{\mathcal{R}}\otimes U(\mathfrak{g}).
\end{equation}
The asymptotic boundary KZB operators are now defined as follows.

\begin{defi}
The first-order differential operators 
\begin{equation}\label{bKZBoper}
\mathcal{D}_i:=E_i
-\sum_{j=1}^{i-1}r_{ji}^+-\kappa_i-\sum_{j=i+1}^Nr_{ij}^-
\end{equation}
in
$\mathbb{D}_{\mathcal{R}}\otimes U(\mathfrak{k})\otimes U(\mathfrak{g})^{\otimes N}
\otimes U(\mathfrak{k})$ \textup{(}$i\in\{1,\ldots,N\}$\textup{)}
are called the asymptotic boundary KZB operators.
Here the subindices indicate in which tensor factor of $U(\mathfrak{g})^{\otimes N}$ the
$U(\mathfrak{g})$-components of $E$, $\kappa$ and $r^{\pm}$ are placed.
\end{defi}
\begin{rema}
Note that $\kappa^{\textup{core}}$ is the part of $\kappa$ that survives when the $U(\mathfrak{k})$-components act according to the trivial representation of $\mathfrak{k}$. Note furthermore that
\[
\mathcal{D}_i\in\mathbb{D}_{\mathcal{R}}\otimes U(\mathfrak{k})\otimes
\bigl(U(\mathfrak{k})^{\otimes (i-1)}\otimes U(\mathfrak{g})^{\otimes
(N-i+1)}\bigr)\otimes U(\mathfrak{k}).
\]
\end{rema}

Consider the family $H_z^{(N)}\in\mathbb{D}_{\mathcal{R}}\otimes 
U(\mathfrak{k})^{\otimes (N+2)}$ ($z\in Z(\mathfrak{g})$) of commuting differential operator
\[
H_z^{(N)}:=\bigl(\Delta^{(N)}\otimes \textup{id}_{U(\mathfrak{k})}\bigr)H_z
\]
with $\Delta^{(N)}: U(\mathfrak{k})\rightarrow U(\mathfrak{k})^{\otimes (N+1)}$ the $N$th
iterate comultiplication map of $U(\mathfrak{k})$ and $H_z$ given by \eqref{Hz}.
Then $\mathbf{H}^{(N)}:=-\frac{1}{2}(H_\Omega^{(N)}+\|\rho\|^2)$ is the quantum double spin Calogero-Moser Hamiltonian
\begin{equation*}
\begin{split}
\mathbf{H}^{(N)}=&
-\frac{1}{2}\Delta+V^{(N)},\\
V^{(N)}:=&-\frac{1}{2}\sum_{\alpha\in R}\frac{1}{(\xi_\alpha-\xi_{-\alpha})^2}
\Bigl(\frac{\|\alpha\|^2}{2}+\prod_{\epsilon\in\{\pm 1\}}
(\Delta^{(N)}(y_\alpha)\otimes 1+\xi_{\epsilon\alpha}(1^{\otimes (N+1)}\otimes y_\alpha))
\Bigr)
\end{split}
\end{equation*}
by Proposition \ref{qH}.

\begin{thm}\label{mainthmbKZBEisenstein}
Let $\lambda\in\mathfrak{h}^*$, $\phi_\ell\in\textup{Hom}_K(\mathcal{H}_{\lambda_0},V_\ell)$,
$\phi_r\in\textup{Hom}_K(V_r,\mathcal{H}_{\lambda_N})$ and $\mathbf{D}$ a differential
vertex operator of weight $\underline{\lambda}=(\lambda_0,\ldots,\lambda_N)$.
Consider the smooth $V_\ell\otimes\mathbf{U}\otimes V_r^*$-valued function 
\[
\mathbf{f}_{\underline{\lambda}}^{\phi_\ell,\mathbf{D},\phi_r}:=
\delta f_{\mathcal{H}_{\underline{\lambda}}}^{\phi_\ell,\mathbf{D},\phi_r}|_{A_+}
\]
on $A_+$, 
called the normalised $N$-point $\sigma^{(N)}$-spherical function of weight $\underline{\lambda}$, which admits the explicit integral representation
\begin{equation*}
\begin{split}
\mathbf{f}_{\underline{\lambda}}^{\phi_\ell,\mathbf{D},\phi_r}(a^\prime)
=\delta(a^\prime)
\int_Kdx\, \xi_{-\lambda_N-\rho}(a(a^\prime{}^{-1}x))&\bigl(\sigma_\ell(x)\otimes\tau_{U_1}(x)\otimes\cdots\\
&\quad\cdots\otimes
\tau_{U_N}(x)\otimes\sigma_r^*(k(a^\prime{}^{-1}x))\bigr)
T_{\lambda_N}^{(\phi_\ell\otimes\textup{id}_{\mathbf{U}})\mathbf{D}, \phi_r}.
\end{split}
\end{equation*}
It 
 satisfies the systems of differential equations
\begin{equation}\label{totalDEx}
\begin{split}
\mathcal{D}_i\bigl(\mathbf{f}_{\underline{\lambda}}^{\phi_\ell,\mathbf{D},\phi_r}\bigr)&=
\Bigl(\frac{(\lambda_i,\lambda_i)}{2}-
\frac{(\lambda_{i-1},\lambda_{i-1})}{2}\Bigr)\mathbf{f}_{\underline{\lambda}}^{\phi_\ell,\mathbf{D},\phi_r}\qquad
 (i=1,\ldots,N),\\
 \mathbf{H}^{(N)}\bigl(\mathbf{f}_{\underline{\lambda}}^{\phi_\ell,\mathbf{D},\phi_r}\bigr)&=
 -\frac{(\lambda_N,\lambda_N)}{2}\mathbf{f}_{\underline{\lambda}}^{\phi_\ell,\mathbf{D},\phi_r}
 \end{split}
 \end{equation}
 on $A_+$. Furthermore, $H_z^{(N)}\bigl(\mathbf{f}_{\underline{\lambda}}^{\phi_\ell,\mathbf{D},\phi_r}\bigr)=
\zeta_{\lambda_N-\rho}(z)\mathbf{f}_{\underline{\lambda}}^{\phi_\ell,\mathbf{D},\phi_r}$ on $A_+$
for $z\in Z(\mathfrak{g})$.
\end{thm}
\begin{proof}
The integral representation follows from \eqref{EisN}.
The second line of \eqref{totalDEx} follows from \eqref{diffN}.
By Proposition \ref{bKZBresult} we have
\begin{equation}\label{todo100x}
\widetilde{\mathcal{D}}_if_{\underline{\lambda}}^{\phi_\ell,\mathbf{D},\phi_r}=\Bigl(\frac{(\lambda_i,\lambda_i)}{2}-\frac{(\lambda_{i-1},\lambda_{i-1})}{2}\Bigr)f_{\underline{\lambda}}^{\phi_\ell,\mathbf{D},\phi_r}\qquad
 (i=1,\ldots,N)
\end{equation}
with $\widetilde{\mathcal{D}}_i=E_i-\sum_{j=1}^{i-1}r_{ji}^+-\widetilde{\kappa}_i-\sum_{j=i+1}^Nr_{ij}^-$
and 
\[
\widetilde{\kappa}:=r^+\otimes 1+1\otimes b\otimes 1+1\otimes \widetilde{r}^+=
\kappa-\frac{1}{2}\sum_{\alpha\in R^+}\Bigl(\frac{1+\xi_{-2\alpha}}{1-\xi_{-2\alpha}}
\Bigr)t_\alpha.
\]
To prove the first set of equations of \eqref{totalDEx} it thus suffices to show that
\begin{equation}\label{tododelta}
\delta E(\delta^{-1})=-\frac{1}{2}\sum_{\alpha\in R^+}\Bigl(\frac{1+\xi_{-2\alpha}}{1-\xi_{-2\alpha}}
\Bigr)t_\alpha
\end{equation}
in $\mathcal{R}\otimes\mathfrak{h}$. This
follows from the following computation,
\begin{equation*}
\begin{split}
\delta E(\delta^{-1})&=-\sum_{j=1}^\rr\left(\rho(x_j)x_j-\sum_{\alpha\in R^+}\frac{\xi_{-2\alpha}\alpha(x_j)x_j}{1-\xi_{-2\alpha}}\right)\\
&=-t_{\rho}-\sum_{\alpha\in R^+}\frac{\xi_{-2\alpha}t_\alpha}{1-\xi_{-2\alpha}}=
-\frac{1}{2}\sum_{\alpha\in R^+}\Bigl(\frac{1+\xi_{-2\alpha}}{1-\xi_{-2\alpha}}
\Bigr)t_\alpha.
\end{split}
\end{equation*}
\end{proof}

\subsection{Formal $N$-point spherical functions}\label{intertwinersection}

In this subsection we introduce the analogue of $N$-point spherical functions in the context of Verma modules. They give rise to asymptotically free solutions of the asymptotic boundary KZB operators. 

Fix finite dimensional $\mathfrak{g}$-representations $\tau_i: \mathfrak{g}\rightarrow
\mathfrak{gl}(U_i)$ ($1\leq i\leq N$). Let $\underline{\lambda}=(\lambda_0,\ldots,\lambda_N)$ with $\lambda_i\in\mathfrak{h}^*$ 
and choose vertex operators $\Psi_i\in\textup{Hom}_{\mathfrak{g}}(M_{\lambda_i},M_{\lambda_{i-1}}\otimes U_i)$ for $i=1,\ldots,N$. Set
\begin{equation}\label{Psi}
\mathbf{\Psi}:=(\Psi_1\otimes\textup{id}_{U_2\otimes\cdots\otimes U_N})
\cdots (\Psi_{N-1}\otimes\textup{id}_{U_N})\Psi_N,
\end{equation}
which is a $\mathfrak{g}$-intertwiner $M_{\lambda_N}\rightarrow M_{\lambda_0}\otimes\mathbf{U}$.
Let  $(\sigma_\ell,V_\ell), (\sigma_r,V_r)$ be two finite dimensional semisimple $\mathfrak{k}$-modules. 
Write $\sigma_\ell^{(N)}$ for the representation map of the finite dimensional $\mathfrak{k}$-module 
$V_\ell\otimes\mathbf{U}$, where $\mathfrak{k}$ acts diagonally on $V_\ell\otimes\mathbf{U}$. Note that $\sigma_\ell^{(N)}$ is semisimple since $\mathfrak{k}$ is reductive in $\mathfrak{g}$ (see \cite[\S 1.7]{Di}).
Denote by $\sigma^{(N)}=\sigma_\ell^{(N)}\otimes\sigma_r^*$ the representation map of the associated finite dimensional semisimple
$\mathfrak{k}\oplus\mathfrak{k}$-module $(V_\ell\otimes\mathbf{U})\otimes V_r^*$. 

Note that
$(\phi_\ell\otimes\textup{id}_{\mathbf{U}})\mathbf{\Psi}\in\textup{Hom}_{\mathfrak{k}}(M_{\lambda_N},V_\ell\otimes\mathbf{U})$ for 
$\phi_\ell\in\textup{Hom}_{\mathfrak{k}}(M_{\lambda_0},V_\ell)$
\begin{defi}\label{defNpoint}
Let $\phi_\ell\in\textup{Hom}_{\mathfrak{k}}(M_{\lambda_0},V_\ell)$, $\phi_r\in\textup{Hom}_{\mathfrak{k}}(V_r,\overline{M}_{\lambda_N})$ and let $\mathbf{\Psi}$ be a vertex operator of the form \eqref{Psi}. The formal elementary $\sigma^{(N)}$-spherical function 
\begin{equation}\label{relCFnormal}
\begin{split}
F_{M_{\underline{\lambda}}}^{\phi_\ell,\mathbf{\Psi},\phi_r}:=&
F_{M_{\lambda_N}}^{(\phi_\ell\otimes\textup{id}_{\mathbf{U}})\mathbf{\Psi},\phi_r}\\
=&\sum_{\mu\leq\lambda_N}((\phi_\ell\otimes\textup{id}_{\mathbf{U}})\mathbf{\Psi}\phi_r^\mu)\xi_\mu
\in (V_\ell\otimes\mathbf{U}\otimes V_r^*)[[\xi_{-\alpha_1},
\ldots,\xi_{-\alpha_\rr}]]\xi_{\lambda_N}
\end{split}
\end{equation}
is called a formal $N$-point 
$\sigma^{(N)}$-spherical function associated with the $(N+1)$-tuple of Verma modules
$(M_{\lambda_0},\ldots,M_{\lambda_N})$.
 \end{defi} 

By Theorem \ref{mainTHMF}, the formal $N$-point $\sigma^{(N)}$-spherical function
$F_{M_{\underline{\lambda}}}^{\phi_\ell,\mathbf{\Psi},\phi_r}$
is analytic on $A_+$ for $\lambda_N\in\mathfrak{h}_{\textup{HC}}^*$.

Recall the normalisation factor $\delta$ defined by \eqref{deltagauge} (which we will view as formal power series in 
$\mathbb{C}[[\xi_{-\alpha_1},\ldots,\xi_{-\alpha_\rr}]]\xi_\rho$).
\begin{defi}\label{shiftdef}
Let $\Psi_i\in\textup{Hom}_{\mathfrak{g}}(M_{\lambda_i-\rho},M_{\lambda_{i-1}-\rho}\otimes U_i)$
($1\leq i\leq N$) and write 
\[
\mathbf{\Psi}\in\textup{Hom}_{\mathfrak{g}}(M_{\lambda_N-\rho},
M_{\lambda_0-\rho}\otimes\mathbf{U})
\] 
for the resulting vertex operator \eqref{Psi}. We call
\[
\mathbf{F}_{\underline{\lambda}}^{\phi_\ell,\mathbf{\Psi},\phi_r}:=\delta F_{M_{\underline{\lambda}-\rho}}^{\phi_\ell,\mathbf{\Psi},\phi_r}
\in (V_\ell\otimes\mathbf{U}\otimes V_r^*)[[\xi_{-\alpha_1},\ldots,\xi_{-\alpha_\rr}]]\xi_{\lambda_N}
\]
a normalised $N$-point $\sigma^{(N)}$-spherical function of weight
$\underline{\lambda}-\rho:=(\lambda_0-\rho,\lambda_1-\rho,\ldots,\lambda_N-\rho)$.
\end{defi}
For weight $\underline{\lambda}$ with 
$\lambda_N\in\mathfrak{h}_{\textup{HC}}^*+\rho$ the normalised formal $N$-point $\sigma^{(N)}$-spherical function
$\mathbf{F}_{\underline{\lambda}}^{\phi_\ell,\mathbf{\Psi},\phi_r}$ is an $V_\ell\otimes\mathbf{U}\otimes V_r^*$-valued analytic function on $A_+$.
In terms of the normalised formal elementary $\sigma^{(N)}$-spherical functions, we have
\begin{equation}\label{relCF}
\mathbf{F}^{\phi_\ell,\mathbf{\Psi},\phi_r}_{\underline{\lambda}}=\mathbf{F}_{\lambda_N}^{(\phi_\ell\otimes\textup{id}_{\mathbf{U}})\mathbf{\Psi},\phi_r},
\end{equation}
and hence 
\begin{equation}\label{CMeigenfunctions}
\begin{split}
\mathbf{H}^{(N)}\bigl(\mathbf{F}_{\underline{\lambda}}^{\phi_\ell,\mathbf{\Psi},\phi_r}\bigr)&=
-\frac{(\lambda_N,\lambda_N)}{2}\,\mathbf{F}_{\underline{\lambda}}^{\phi_\ell,\mathbf{\Psi},\phi_r},\\
H_z^{(N)}\bigl(\mathbf{F}_{\underline{\lambda}}^{\phi_\ell,\mathbf{\Psi},\phi_r}\bigr)&=
\zeta_{\lambda_N-\rho}(z)\mathbf{F}_{\underline{\lambda}}^{\phi_\ell,\mathbf{\Psi},\phi_r},\qquad
z\in Z(\mathfrak{g})
\end{split}
\end{equation}
by Theorem \ref{thmnorm}.
We now show by
a suitable adjustment of the algebraic arguments from Subsection \ref{S61}
that the normalised formal $N$-point $\sigma^{(N)}$-spherical functions are eigenfunctions
of the asymptotic boundary KZB operators.

\begin{thm}\label{mainthmbKZB}
Let $\phi_\ell\in\textup{Hom}_{\mathfrak{k}}(M_{\lambda_0},V_\ell)$, $\phi_r\in\textup{Hom}_{\mathfrak{k}}(V_r,\overline{M}_{\lambda_N})$ and let $\mathbf{\Psi}$ be a product of $N$ vertex operators as given in Definition \ref{shiftdef}. The 
normalised formal $N$-point $\sigma^{(N)}$-spherical function $\mathbf{F}_{\underline{\lambda}}^{\phi_\ell,\mathbf{\Psi},\phi_r}$ satisfies the system of differential equations
\begin{equation}\label{totalDE}
\mathcal{D}_i\bigl(\mathbf{F}_{\underline{\lambda}}^{\phi_\ell,\mathbf{\Psi},\phi_r}\bigr)=\Bigl(\frac{(\lambda_i,\lambda_i)}{2}-\frac{(\lambda_{i-1},\lambda_{i-1})}{2}\Bigr)
\mathbf{F}_{\underline{\lambda}}^{\phi_\ell,\mathbf{\Psi},\phi_r}\qquad
 (i=1,\ldots,N)
\end{equation}
in $(V_\ell\otimes\mathbf{U}\otimes V_r^*)[[\xi_{-\alpha_1},
\ldots,\xi_{-\alpha_\rr}]]\xi_{\lambda_N}$. For $\lambda_N\in\mathfrak{h}_{\textup{HC}}^*+\rho$ the differential
equations \eqref{totalDE} are valid as analytic $V_\ell\otimes\mathbf{U}\otimes V_r^*$-valued
analytic functions on $A_+$.
\end{thm}
\begin{proof}
As in the proof of Theorem \ref{mainthmbKZBEisenstein}, the differential equations \eqref{totalDE} are equivalent to
\begin{equation}\label{todo100}
\widetilde{\mathcal{D}}_i\bigl(F_{M_{\underline{\lambda}}}^{\phi_\ell,\mathbf{\Psi},\phi_r}\bigr)=\frac{1}{2}(\zeta_{\lambda_i}(\Omega)-
\zeta_{\lambda_{i-1}}(\Omega))F_{M_{\underline{\lambda}}}^{\phi_\ell,\mathbf{\Psi},\phi_r}\qquad
 (i=1,\ldots,N)
\end{equation}
with $\widetilde{\mathcal{D}}_i=E_i-\sum_{j=1}^{i-1}r_{ji}^+-\widetilde{\kappa}_i-\sum_{j=i+1}^Nr_{ij}^-$
and 
\[
\widetilde{\kappa}:=\kappa-\frac{1}{2}\sum_{\alpha\in R^+}\Bigl(\frac{1+\xi_{-2\alpha}}{1-\xi_{-2\alpha}}
\Bigr)t_\alpha=
r^+\otimes 1+1\otimes b\otimes 1+1\otimes \widetilde{r}^+
\]
(here $b$ is given by \eqref{kappa}).

Write $\Lambda=\{\mu\in\mathfrak{h}^*\,\, | \,\, \mu\leq\lambda_N\}$ and 
$\Lambda_m:=\{\mu\in\Lambda \,\, | \,\, (\lambda_N-\mu,\rho^\vee)\leq m\}$ ($m\in\mathbb{Z}_{\geq 0}$).
Consider the $V_\ell\otimes\mathbf{U}\otimes V_r^*$-valued quasi-polynomial
\[
F_{M_{\underline{\lambda}},m}^{\phi_\ell,\mathbf{\Psi},\phi_r}:=\sum_{\mu\in\Lambda_m}
((\phi_\ell\otimes\textup{id}_{\mathbf{U}})\mathbf{\Psi}\phi_r^\mu)\xi_\mu
\]
for $m\in\mathbb{Z}_{\geq 0}$. Fix $a\in A_+$. Then we have
\begin{equation*}
\begin{split}
(\zeta_{\lambda_{i-1}}(\Omega)-&\zeta_{\lambda_i}(\Omega))F_{M_{\underline{\lambda}},m}^{\phi_\ell,\mathbf{\Psi},
\phi_r}(a)=\\
&=\sum_{\mu\in\Lambda_m}
(\phi_\ell\otimes\textup{id}_{\mathbf{U}})(\Psi_{M_{\lambda_1}}\cdots
\Psi_{M_{\lambda_{i-1}}}\widetilde{\Psi}_i
\Psi_{M_{\lambda_{i+1}}}\cdots
\Psi_{M_{\lambda_N}}\phi_r^\mu)a^\mu
\end{split}
\end{equation*}
with $\Psi_{M_{\lambda_i}}:=\Psi_i\otimes\textup{id}_{U_{i+1}\otimes\cdots\otimes U_N}$ and
$\widetilde{\Psi}_i:=\Omega_{M_{\lambda_{i-1}}}\Psi_{M_{\lambda_i}}-
\Psi_{M_{\lambda_i}}\Omega_{M_{\lambda_i}}$. By the asymptotic operator KZB equation \eqref{multioperatorKZB} applied to the factorisation $(r^+(a), -r^-(a), b(a))$ of $\Omega$
(see Proposition \ref{propfactorization}\,{\bf b}), we get
\begin{equation*}
\begin{split}
\Bigl(\Bigl(\frac{\zeta_{\lambda_{i-1}}(\Omega)}{2}-&\frac{\zeta_{\lambda_i}(\Omega)}{2}+E_{U_i}
-\sum_{j=1}^{i-1}r_{U_jU_i}^+-r_{V_\ell U_i}^+-b_{U_i}-\sum_{j=i+1}^Nr_{U_iU_j}^-\Bigr)
F_{M_{\underline{\lambda}},m}^{\phi_\ell,\mathbf{\Psi},\phi_r}\Bigr)(a)=\\
&\qquad\quad=-\sum_{\mu\in\Lambda_m}\sum_{\alpha\in R}\frac{(e_\alpha+e_{-\alpha})_{U_i}}{(1-a^{-2\alpha})}
(\phi_\ell\otimes\textup{id}_{\mathbf{U}})(\mathbf{\Psi}(e_\alpha)_{M_{\lambda_N}}\phi_r^\mu)a^\mu.
\end{split}
\end{equation*}
This being valid for all $a\in A_+$, hence we get
\begin{equation}\label{trunkm}
\begin{split}
\Bigl(\frac{1}{2}(\zeta_{\lambda_{i-1}}(\Omega)-&\zeta_{\lambda_i}(\Omega))+E_{U_i}
-\sum_{j=1}^{i-1}r_{U_jU_i}^+-r_{V_\ell U_i}^+-b_{U_i}-\sum_{j=i+1}^Nr_{U_iU_j}^-\Bigr)
F_{M_{\underline{\lambda}},m}^{\phi_\ell,\mathbf{\Psi},\phi_r}=\\
&=-\sum_{\mu\in\Lambda_m}\sum_{\alpha\in R}(e_\alpha+e_{-\alpha})_{U_i}
(\phi_\ell\otimes\textup{id}_{\mathbf{U}})(\mathbf{\Psi}(e_\alpha)_{M_{\lambda_N}}\phi_r^\mu)\frac{\xi_\mu}{(1-\xi_{-2\alpha})}
\end{split}
\end{equation}
viewed as identity in
$(V_\ell\otimes\mathbf{U}\otimes V_r^*)[[\xi_{-\alpha_1},\ldots,\xi_{-\alpha_\rr}]]\xi_{\lambda_N}$ (so
$(1-\xi_{-2\alpha})^{-1}=\sum_{k=0}^{\infty}\xi_{-\alpha}^{2k}$ if $\alpha\in R^+$ and
$(1-\xi_{-2\alpha})^{-1}=-\sum_{k=1}^{\infty}\xi_\alpha^{2k}$ if $\alpha\in R^-$, and analogous expansions for the coefficients of $r^{\pm}$ and $\widetilde{\kappa}$).
We claim that the identity \eqref{trunkm} is also valid when the summation over $\Lambda_m$ is replaced by summation over $\Lambda$:
\begin{equation}\label{notrunk}
\begin{split}
\Bigl(\frac{1}{2}(\zeta_{\lambda_{i-1}}(\Omega)-&\zeta_{\lambda_i}(\Omega))+E_{U_i}
-\sum_{j=1}^{i-1}r_{U_jU_i}^+-r_{V_\ell U_i}^+-b_{U_i}-\sum_{j=i+1}^Nr_{U_iU_j}^-\bigr)
F_{M_{\underline{\lambda}}}^{\phi_\ell,\mathbf{\Psi},\phi_r}=\\
&=-\sum_{\mu\in\Lambda}\sum_{\alpha\in R}(e_\alpha+e_{-\alpha})_{U_i}
(\phi_\ell\otimes\textup{id}_{\mathbf{U}})(\mathbf{\Psi}(e_\alpha)_{M_{\lambda_N}}\phi_r^\mu)
\frac{\xi_\mu}{(1-\xi_{-2\alpha})}
\end{split}
\end{equation}
in $(V_\ell\otimes\mathbf{U}\otimes V_r^*)[[\xi_{-\alpha_1},\ldots,\xi_{-\alpha_\rr}]]\xi_{\lambda_N}$.
Take $\eta\in\Lambda$ and let $m\in\mathbb{N}$ such that $(\lambda_N-\eta,\rho^\vee)\leq m$.
Then the $\xi_\eta$-coefficient of the left hand side of \eqref{notrunk} is the same as the $\xi_\eta$-coefficient of the left hand side of \eqref{trunkm} since the coefficients of $r^{\pm}$ and $\widetilde{\kappa}$ are in
$\mathbb{C}[[\xi_{-\alpha_1},\ldots,\xi_{-\alpha_\rr}]]$. Exactly the same argument applies to the $\xi_\mu$-coefficients of the right hand sides of \eqref{notrunk} and \eqref{trunkm}, from which the claim follows. Now rewrite the right hand side of \eqref{notrunk} as
\begin{equation*}
\begin{split}
-\sum_{\mu\in\Lambda}\sum_{\alpha\in R}&(e_\alpha+e_{-\alpha})_{U_i}
(\phi_\ell\otimes\textup{id}_{\mathbf{U}})(\mathbf{\Psi}(e_\alpha)_{M_{\lambda_N}}\phi_r^\mu)
\frac{\xi_\mu}{(1-\xi_{-2\alpha})}=\\
=&-\sum_{\alpha\in R}\sum_{\nu\in\Lambda}(e_\alpha+e_{-\alpha})_{U_i}(\phi_\ell\otimes\textup{id}_{\mathbf{U}})\mathbf{\Psi}(
\textup{proj}_{M_{\lambda_N}}^\nu(e_\alpha)_{M_{\lambda_N}}\phi_r)\frac{\xi_{\nu}}{(\xi_\alpha-\xi_{-\alpha})}
\end{split}
\end{equation*}
and use that 
\[
\sum_{\alpha\in R}\frac{(e_\alpha+e_{-\alpha})\otimes e_\alpha}{\xi_\alpha-\xi_{-\alpha}}=
\sum_{\alpha\in R}\frac{e_\alpha\otimes y_\alpha}{\xi_\alpha-\xi_{-\alpha}}
\]
to obtain
\begin{equation*}
\begin{split}
-\sum_{\mu\in\Lambda}\sum_{\alpha\in R}&(e_\alpha+e_{-\alpha})_{U_i}
(\phi_\ell\otimes\textup{id}_{\mathbf{U}})(\mathbf{\Psi}(e_\alpha)_{M_{\lambda_N}}\phi_r^\mu)
\frac{\xi_\mu}{(1-\xi_{-2\alpha})}=\\
=&-\sum_{\alpha\in R}\sum_{\nu\in\Lambda}(e_\alpha)_{U_i}(\phi_\ell\otimes\textup{id}_{\mathbf{U}})\mathbf{\Psi}(
\textup{proj}_{M_{\lambda_N}}^\nu(y_\alpha)_{M_{\lambda_N}}\phi_r)
\frac{\xi_{\nu}}{(\xi_\alpha-\xi_{-\alpha})}\\
=&\widetilde{r}_{U_iV_r^*}^+F_{M_{\underline{\lambda}}}^{\phi_\ell,\mathbf{\Psi},\phi_r},
\end{split}
\end{equation*}
where we used that $\phi_r$ is a $\mathfrak{k}$-intertwiner, as well as the explicit formula \eqref{rtildeplus} for $\widetilde{r}^+$. Substituting this identity
in \eqref{notrunk}, we obtain \eqref{todo100}.
\end{proof}

\subsection{Boundary fusion operators}\label{sectionBFO}

In Subsection \ref{S62} we introduced the expectation value of (differential) vertex operators.
Recall the parametrisation of the vertex operators introduced in Definition \ref{notVO}.
The expectation value of products \eqref{Psi} of vertex operators gives rise to the fusion operator:
\begin{defi}\textup{(}\cite[Prop. 3.7]{E}\textup{)}.
Let $\lambda\in\mathfrak{h}_{\textup{reg}}^*$. The fusion operator 
$\mathbf{J}_{\mathbf{U}}(\lambda)$ is the $\mathfrak{h}$-linear automorphism of $\mathbf{U}$ defined by
\[
u_1\otimes\cdots\otimes u_N\mapsto
(m_{\lambda_0}^*\otimes\textup{id}_{\mathbf{U}})
(\Psi_{\lambda_1}^{u_1}\otimes\textup{id}_{U_2\otimes\cdots\otimes U_N})
\cdots (\Psi_{\lambda_{N-1}}^{u_{N-1}}\otimes\textup{id}_{U_N})\Psi_{\lambda_N}^{u_N}(m_\lambda)
\]
for $u_i\in U_i[\mu_i]$ ($\mu_i\in P$), where $\lambda_i:=\lambda-\mu_{i+1}\cdots -\mu_N$
for $i=0,\ldots,N-1$ and $\lambda_N=\lambda$. 
\end{defi}
We suppress the dependence on $\mathbf{U}$ and denote $\mathbf{J}_{\mathbf{U}}(\lambda)$ by
$\mathbf{J}(\lambda)$ if no confusion is possible (in fact, a universal fusion operator 
exists, in the sense that its action on $\mathbf{U}$ reproduces
$\mathbf{J}_{\mathbf{U}}(\lambda)$ for all finite dimensional 
$\mathfrak{g}^{\oplus N}$-modules, see, e.g., \cite[Prop. 3.19]{E} and references therein).

Lemma \ref{evlem} shows that for $\mathbf{u}:=u_1\otimes\cdots\otimes u_N$ with $u_i\in U[\mu_i]$  and $\lambda\in\mathfrak{h}_{\textup{reg}}^*$ we have
\begin{equation}\label{fusionmeaning}
\Psi_\lambda^{\mathbf{J}(\lambda)\mathbf{u}}=
(\Psi_{\lambda_1}^{u_1}\otimes\textup{id}_{U_2\otimes\cdots\otimes U_N})
\cdots (\Psi_{\lambda_{N-1}}^{u_{N-1}}\otimes\textup{id}_{U_N})\Psi_{\lambda_N}^{u_N}
\end{equation}
in $\textup{Hom}_{\mathfrak{g}}(M_\lambda,M_{\lambda_0}\otimes\mathbf{U})$.

Fix from now on a finite dimensional semisimple $\mathfrak{k}$-module $V_\ell$. Recall the parametrisation of $\mathfrak{k}$-intertwiners $\phi_{\ell,\lambda}^v\in\textup{Hom}_{\mathfrak{k}}(M_\lambda,V_\ell)$ by their expectation value $v\in V_\ell$ as introduced in Definition \ref{intertwinerparametrization}\,{\bf a}.
\begin{lem}\label{boundaryfusion}
Let $\lambda\in\mathfrak{h}_{\textup{reg}}^*$. The linear operator $\mathbf{J}_{\ell,\mathbf{U}}(\lambda)\in\textup{End}(V_\ell\otimes\mathbf{U})$,
defined by
\[
\mathbf{J}_{\ell,\mathbf{U}}(\lambda)(v\otimes\mathbf{u}):=
(\phi_{\ell,\lambda-\mu}^v\otimes\textup{id}_{\mathbf{U}})\Psi_\lambda^{\mathbf{J}(\lambda)\mathbf{u}}(m_\lambda)
\]
for $v\in V_\ell$ and $\mathbf{u}\in \mathbf{U}[\mu]$ \textup{(}$\mu\in P$\textup{)},
is a linear automorphism. 
 \end{lem}
\begin{proof}
For $v\in V_\ell$ and $\mathbf{u}\in\mathbf{U}[\mu]$ we have
\[
\Psi_\lambda^{\mathbf{J}(\lambda)\mathbf{u}}(m_\lambda)\in m_{\lambda-\mu}\otimes
\mathbf{J}(\lambda)\mathbf{u}+\bigoplus_{\nu>\mu}M_{\lambda-\mu}\otimes
\mathbf{U}[\nu],
\]
and consequently we get
\[
\mathbf{J}_{\ell,\mathbf{U}}(\lambda)(v\otimes\mathbf{u})\in v\otimes \mathbf{J}(\lambda)\mathbf{u}
+\bigoplus_{\nu>\mu}V_\ell\otimes \mathbf{U}[\nu].
\]
Choose an ordered tensor product basis of $V_\ell\otimes\mathbf{U}$ 
in which the $\mathbf{U}$-components consist of weight vectors. Order the tensor product basis 
in such a way that is compatible with the dominance order on the weights of the $\mathbf{U}$-components of the basis elements. With respect to such a basis, 
$\mathbf{J}_{\ell,\mathbf{U}}(\lambda)(\textup{id}_{V_\ell}\otimes\mathbf{J}(\lambda)^{-1})$ is represented by a triangular operator with ones on the diagonal, hence it is invertible.
\end{proof}
We call $\mathbf{J}_{\ell,\mathbf{U}}(\lambda)$ the (left) boundary fusion operator on $V_\ell\otimes\mathbf{U}$ (we denote $\mathbf{J}_{\ell,\mathbf{U}}(\lambda)$ by $\mathbf{J}_\ell(\lambda)$ if no confusion is possible). A right version $\mathbf{J}_r(\lambda)$ of the
left boundary fusion operator $\mathbf{J}_\ell(\lambda)$ can be constructed in an analogous
manner. We leave the straightforward details to the reader. 

Let $\lambda\in\mathfrak{h}_{\textup{reg}}^*$. By Proposition \ref{relEisPrinalg}\,{\bf a} and Lemma \ref{boundaryfusion}, the map
\[
V_\ell\otimes \mathbf{U}\rightarrow\textup{Hom}_{\mathfrak{k}}(M_\lambda,V_\ell\otimes\mathbf{U}),
\qquad
v\otimes \mathbf{u}\mapsto \phi_{\ell,\lambda}^{\mathbf{J}_\ell(\lambda)(v\otimes\mathbf{u})}
\]
is a linear isomorphism. The $\mathfrak{k}$-intertwiner $\phi_{\ell,\lambda}^{\mathbf{J}_\ell(\lambda)(v\otimes\mathbf{u})}$ admits the following alternative description.
\begin{cor}\label{vertexint}
Let $\lambda\in\mathfrak{h}_{\textup{reg}}^*$, $\mu\in P$
and $v\in V_\ell$. For $\mathbf{u}\in\mathbf{U}[\mu]$ we have
\begin{equation}\label{boundaryfusionformula}
\phi_{\ell,\lambda}^{\mathbf{J}_\ell(\lambda)(v\otimes\mathbf{u})}=
(\phi_{\ell,\lambda-\mu}^v\otimes\textup{id}_{\mathbf{U}})
\Psi_\lambda^{\mathbf{J}(\lambda)\mathbf{u}}.
\end{equation}
For $\mathbf{u}=u_1\otimes\cdots\otimes u_N$ with $u_i\in U_i[\mu_i]$ 
\textup{(}$1\leq i\leq N$\textup{)} we furhermore have
\begin{equation}\label{boundaryfusionformula2}
\phi_{\ell,\lambda}^{\mathbf{J}_\ell(\lambda)(v\otimes\mathbf{u})}=
(\phi_{\ell,\lambda_0}^v\otimes\textup{id}_{\mathbf{U}})(\Psi_{\lambda_1}^{u_1}\otimes\textup{id}_{U_2\otimes\cdots\otimes U_N})
\cdots (\Psi_{\lambda_{N-1}}^{u_{N-1}}\otimes\textup{id}_{U_N})\Psi_{\lambda_N}^{u_N},
\end{equation}
with $\lambda_i:=\lambda-\mu_{i+1}\cdots-\mu_N$ \textup{(}$i=0,\ldots,N-1$\textup{)}, 
and $\lambda_N:=\lambda$.
\end{cor}
\begin{proof}
The result follows immediately from Proposition \ref{relEisPrinalg}\,{\bf a}, 
Lemma \ref{boundaryfusion} and \eqref{fusionmeaning}.
\end{proof}

\begin{defi}
Let $V_\ell,V_r$ be finite dimensional semisimple $\mathfrak{k}$-modules.
Let $\underline{\lambda}=(\lambda_0,\ldots,\lambda_N)$ with $\lambda_N\in\mathfrak{h}_{\textup{reg}}^*$ and with $\mu_i:=\lambda_i-\lambda_{i-1}\in P$ for $i=1,\ldots,N$.
Let $v\in V_\ell$, $f\in V_r^*$ and $\mathbf{u}=u_1\otimes\cdots\otimes
u_N$ with $u_i\in U_i[\mu_i]$. We write
\[
F_{M_{\underline{\lambda}}}^{v,\mathbf{u},f}:=
F_{M_{\underline{\lambda}}}^{\mathbf{J}_\ell(\lambda_N)(v\otimes\mathbf{u})\otimes f}
\]
for the formal $N$-point spherical function with leading coefficient 
$\mathbf{J}_\ell(\lambda_N)(v\otimes\mathbf{u})\otimes f$, and 
\[
\mathbf{F}_{\underline{\lambda}}^{v,u,f}:=
\delta F_{M_{\underline{\lambda}-\rho}}^{v,\mathbf{u},f}=
\mathbf{F}_{\underline{\lambda}}^{\mathbf{J}_\ell(\lambda_N-\rho)(v\otimes\mathbf{u})\otimes f}
\]
for its normalised version.
\end{defi}
Note that
\[
F_{M_{\underline{\lambda}}}^{v,\mathbf{u},f}=
F_{M_{\underline{\lambda}}}^{\phi_{\ell,\lambda_0}^v,\mathbf{\Psi}_{\lambda_N}^{\mathbf{J}(\lambda_N)\mathbf{u}},\phi_{r,\lambda_N}^f}
\]
by Lemma \ref{vertexint}.
Written out as formal power series we thus have
the following three expressions for $F_{M_{\underline{\lambda}}}^{v,\mathbf{u},f}$,
\begin{equation*}
\begin{split}
& F_{M_{\underline{\lambda}}}^{v,\mathbf{u},f}=\sum_{\mu\leq\lambda_N}
\phi_{\ell,\lambda_N}^{\mathbf{J}_\ell(\lambda_N)\mathbf{u}}(\textup{proj}_{M_{\lambda_N}}^\mu\phi_{r,\lambda_N}^f)\xi_\mu\\
&\qquad\quad=\sum_{\mu\leq\lambda_N}(\phi_{\ell,\lambda_0}^v\otimes\textup{id}_{\mathbf{U}})
\Psi_{\lambda_N}^{\mathbf{J}(\lambda_N)\mathbf{u}}(\textup{proj}_{M_{\lambda_N}}^\mu\phi_{r,\lambda_N}^f)\xi_\mu\\
&\quad=\sum_{\mu\leq\lambda_N}(\phi_{\ell,\lambda_0}^v\otimes\textup{id}_{\mathbf{U}})(\Psi_{\lambda_1}^{u_1}\otimes\textup{id}_{U_2\otimes\cdots\otimes U_N})
\cdots (\Psi_{\lambda_{N-1}}^{u_{N-1}}\otimes\textup{id}_{U_N})\Psi_{\lambda_N}^{u_N}
(\textup{proj}_{M_{\lambda_N}}^\mu\phi_{r,\lambda_N}^f)\xi_\mu.
\end{split}
\end{equation*}
The main results of the previous subsection for $\lambda\in\mathfrak{h}_{\textup{reg}}^*$ can now be
reworded as follows.
\begin{cor}\label{corMAINkzb}
Let $\lambda\in\mathfrak{h}_{\textup{reg}}^*$. 
Let $\underline{\lambda}=(\lambda_0,\ldots,\lambda_N)$ with $\lambda_N\in\mathfrak{h}_{\textup{reg}}^*$ and with $\mu_i:=\lambda_i-\lambda_{i-1}\in P$ for $i=1,\ldots,N$.
Let $v\in V_\ell$, $f\in V_r^*$ and $\mathbf{u}=u_1\otimes\cdots\otimes
u_N$ with $u_i\in U_i[\mu_i]$ \textup{(}$i=1,\ldots,N$\textup{)}. Then we have for $i=1,\ldots,N$, 
\begin{equation}\label{eigenvalueeqn}
\begin{split}
\mathcal{D}_i\bigl(\mathbf{F}_{\underline{\lambda}}^{v,\mathbf{u},f}\bigr)&=\Bigl(\frac{(\lambda_i,\lambda_i)}{2}-
\frac{(\lambda_{i-1},\lambda_{i-1})}{2}\Bigr)\mathbf{F}_{\underline{\lambda}}^{v,\mathbf{u},f},\\
\mathbf{H}^{(N)}\bigl(\mathbf{F}_{\underline{\lambda}}^{v,\mathbf{u},f}\bigr)&=
-\frac{(\lambda_N,\lambda_N)}{2}\mathbf{F}_{\underline{\lambda}}^{v,\mathbf{u},f}
\end{split}
\end{equation}
and $H_z^{(N)}\bigl(\mathbf{F}_{\underline{\lambda}}^{v,\mathbf{u},f}\bigr)=\zeta_{\lambda_N-\rho}(z)
\mathbf{F}_{\underline{\lambda}}^{v,\mathbf{u},f}$ for $z\in Z(\mathfrak{g})$.
This holds true 
as $V_\ell\otimes\mathbf{U}\otimes V_r^*$-valued analytic functions on $A_+$ when 
$\lambda_N\in\mathfrak{h}_{\textup{HC}}^*\cap\mathfrak{h}_{\textup{reg}}^*$.
\end{cor}

\subsection{Commutativity of the asymptotic boundary KZB operators}\label{S66}
In this subsection we show that the asymptotic boundary KZB operators $\mathcal{D}_i$ ($1\leq i\leq N$)
pairwise commute in $\mathbb{D}_{\mathcal{R}}\otimes U(\mathfrak{k})\otimes
U(\mathfrak{g})^{\otimes N}\otimes U(\mathfrak{k})$, and that they also commute with the
quantum Hamiltonians $H_z^{(N)}\in \mathbb{D}_{\mathcal{R}}\otimes U(\mathfrak{k})^{\otimes (N+2)}$ for $z\in Z(\mathfrak{g})$ (and hence also with $\mathbf{H}^{(N)}$).
We begin with the following lemma.
\begin{lem}
Let $V$ be a finite dimensional semisimple $U(\mathfrak{k})\otimes U(\mathfrak{g})^{\otimes N}\otimes
U(\mathfrak{k})$-module and 
suppose that $\lambda\in\mathfrak{h}_{\textup{reg}}^*$. 
Then the asymptotic boundary KZB operators 
$\mathcal{D}_i$ \textup{(}$1\leq i\leq N$\textup{)} and the quantum Hamiltonians $H_z^{(N)}$ \textup{(}$z\in Z(\mathfrak{g})$\textup{)}
pairwise commute as linear operators on 
$V[[\xi_{-\alpha_1},\ldots,\xi_{-\alpha_\rr}]]\xi_\lambda$.
\end{lem}
\begin{proof}
It suffices to prove the lemma for $V=V_\ell\otimes\mathbf{U}\otimes V_r^*$ with
$V_\ell, V_r$ finite dimensional semisimple $\mathfrak{k}$-modules and $\mathbf{U}=U_1\otimes\cdots\otimes U_N$ with $U_1,\ldots,U_N$ finite dimensional $\mathfrak{g}$-modules.

Define an ultrametric $d$ on 
$(V_\ell\otimes\mathbf{U}\otimes V_r^*)[[\xi_{-\alpha_1},\ldots,\xi_{-\alpha_\rr}]]\xi_\lambda$
by the formula $d(f,g):=2^{-\varpi(f-g)}$ with, for $\sum_{\mu\leq\lambda}e_\mu\xi_\mu\in
(V_\ell\otimes\mathbf{U}\otimes V_r^*)[[\xi_{-\alpha_1},\ldots,\xi_{-\alpha_\rr}]]\xi_\lambda$ nonzero,
\[
\varpi\bigl(\sum_{\mu\leq\lambda}e_\mu\xi_\mu\bigr):=
\textup{min}\{ (\lambda-\mu,\rho^\vee) \,\, | \,\, \mu\leq\lambda: e_\mu\not=0\},
\]
and $\varpi(0)=\infty$.
Consider $(V_\ell\otimes\mathbf{U}\otimes V_r^*)[[\xi_{-\alpha_1},\ldots,
\xi_{-\alpha_\rr}]]\xi_\lambda$ as topological space with respect to the resulting metric topology. 
Note that $\mathcal{D}_i$ ($1\leq i\leq N$) and $H_z^{(N)}$ ($z\in Z(\mathfrak{g})$) are continuous linear operators on
$(V_\ell\otimes\mathbf{U}\otimes V_r^*)[[\xi_{-\alpha_1},\ldots,
\xi_{-\alpha_\rr}]]\xi_\lambda$ since their scalar components lie
in the subring
$\mathcal{R}
\subseteq\mathbb{C}[[\xi_{-\alpha_1},\ldots,\xi_{-\alpha_\rr}]]$. It thus suffices to show that 
$(V_\ell\otimes\mathbf{U}\otimes V_r^*)[[\xi_{-\alpha_1},\ldots,
\xi_{-\alpha_\rr}]]\xi_\lambda$ has a topological linear basis consisting of common eigenfunctions
for the differential operators $\mathcal{D}_i$ ($1\leq i\leq N$) and $H_z^{(N)}$ ($z\in Z(\mathfrak{g})$).

Fix linear basis $\{v_i\}_{i\in I}$, $\{\mathbf{b}_j\}_{j\in J}$ and $\{f_s\}_{s\in S}$ of $V_\ell$, $\mathbf{U}$ and $V_r^*$
respectively. Take the basis elements $\mathbf{b}_j$ of the form $\mathbf{b}_j=u_{1,j}\otimes\cdots\otimes u_{N,j}$ with
$u_{k,j}$ a weight vector in $U_k$ of weight $\mu_k(\mathbf{b}_j)$ ($1\leq k\leq N$). For 
$q\in \sum_{k=1}^r\mathbb{Z}_{\geq 0}\alpha_k$ write 
\[
\underline{\lambda}(\mathbf{b}_j)-q:=(\lambda-q-\mu_1(\mathbf{b}_j)-\cdots-\mu_N(\mathbf{b}_j),
\ldots,\lambda-q-\mu_N(\mathbf{b}_j),\lambda-q).
\]
We then have
\[
\mathbf{F}_{\underline{\lambda}(\mathbf{b}_j)-q}^{v_i,\mathbf{b}_j,f_s}=
\bigl(\mathbf{J}_\ell(\lambda-q)(v_i\otimes\mathbf{b}_j)\otimes f_s\bigr)\xi_{\lambda-q}+
\sum_{\mu<\lambda-q}e_{i,j,s;q}(\mu)\xi_\mu
\]
for certain vectors $e_{i,j,s;q}(\mu)\in V_\ell\otimes\mathbf{U}\otimes V_r^*$. Lemma \ref{boundaryfusion} then implies that
\[
\bigl\{\mathbf{F}_{\underline{\lambda}(\mathbf{b}_j)-q}^{v_i,\mathbf{b}_j,f_s}\,\,\, | \,\,\, (i,j,s)\in I\times J\times S,\,\,  q\in \sum_{k=1}^r\mathbb{Z}_{\geq 0}\alpha_k \bigr\}
\]
is a topological linear basis of $(V_\ell\otimes\mathbf{U}\otimes V_r^*)[[\xi_{-\alpha_1},\ldots,
\xi_{-\alpha_\rr}]]\xi_\lambda$. Finally Corollary \ref{corMAINkzb} shows that the basis elements
$\mathbf{F}_{\underline{\lambda}(\mathbf{b}_j)-q}^{v_i,\mathbf{b}_j,f_s}$ are simultaneous eigenfunctions of
$\mathcal{D}_k$ ($1\leq k\leq N$) and $H_z^{(N)}$ ($z\in Z(\mathfrak{g})$).
\end{proof}
We can now show the universal integrability of the asymptotic boundary KZB operators, as well as their compatibility with the quantum Hamiltonians $H_z^{(N)}$ ($z\in Z(\mathfrak{g})$).
\begin{thm}\label{consistentoperators}
In $\mathbb{D}_{\mathcal{R}}\otimes U(\mathfrak{k})\otimes U(\mathfrak{g})^{\otimes N}
\otimes U(\mathfrak{k})$ we have
\[
[\mathcal{D}_i,\mathcal{D}_j]=0,\qquad [\mathcal{D}_i,H_z^{(N)}]=0,\qquad [H_z^{(N)},H_{z^\prime}^{(N)}]=0
\]
for $i,j=1,\ldots,N$ and $z,z^\prime\in Z(\mathfrak{g})$.
\end{thm}
\begin{proof}
By the previous lemma, it suffices to show that if the differential operator
\[
L\in\mathbb{D}_{\mathcal{R}}\otimes U(\mathfrak{k})\otimes U(\mathfrak{g})^{\otimes N}\otimes U(\mathfrak{k})\]
acts as zero
on $V[[\xi_{-\alpha_1},\ldots,\xi_{-\alpha_\rr}]]\xi_\lambda$ for all finite dimensional semisimple
$U(\mathfrak{k})\otimes U(\mathfrak{g})^{\otimes N}\otimes U(\mathfrak{k})$-modules $V$
and all $\lambda\in\mathfrak{h}_{\textup{reg}}^*$, then $L=0$.

We identify the algebra $\mathbb{D}(A)^A$ of constant coefficient differential operators
on $A$ with the algebra $S(\mathfrak{h}^*)$ of complex polynomials on $\mathfrak{h}^*$,
by associating $\partial_h$ ($h\in \mathfrak{h}_0$) with the linear polynomial
$\lambda\mapsto \lambda(h)$. For $p\in S(\mathfrak{h}^*)$ we write $p(\partial)$
for the corresponding constant coefficient differential operator on $A$.

Write
\[
L=\sum_if_iL_i
\]
with $\{f_i\}_i\subset\mathcal{R}$ linear independent 
and $L_i\in \mathbb{D}(A)^A\otimes U(\mathfrak{k})\otimes U(\mathfrak{g})^{\otimes N}\otimes U(\mathfrak{k})$. Expand 
\[
L_i=\sum_{j}a_{ij}p_{ij}(\partial)
\]
with $p_{ij}\in S(\mathfrak{h}^*)$ and $\{a_{ij}\}_{j}\subset U(\mathfrak{k})\otimes U(\mathfrak{g})^{\otimes N}\otimes U(\mathfrak{k})$ linear independent for all $i$. Then 
\[
0=L(v\xi_\lambda)=\sum_i\bigl(\sum_jp_{ij}(\lambda)a_{ij}(v)\bigr)f_i\xi_\lambda
\]
in $V[[\xi_{-\alpha_1},\ldots,\xi_{-\alpha_\rr}]]\xi_\lambda$ for $v\in V$ and $\lambda\in\mathfrak{h}_{\textup{reg}}^*$, where $V$ is an arbitrary finite dimensional semisimple $U(\mathfrak{k})\otimes
U(\mathfrak{g})^{\otimes N}\otimes U(\mathfrak{k})$-module. Since the $\{f_i\}_i$ are linear independent, we get
\[
\sum_jp_{ij}(\lambda)a_{ij}(v)=0
\]
in $V$ for all $i$, for all $\lambda\in\mathfrak{h}_{\textup{reg}}^*$ and $v\in V$, with $V$ any finite dimensional semisimple $U(\mathfrak{k})\otimes
U(\mathfrak{g})^{\otimes N}\otimes U(\mathfrak{k})$-module. 
Since $\mathfrak{l}=\mathfrak{k}\oplus\mathfrak{g}^{\oplus N}\oplus\mathfrak{k}$ is a reductive Lie algebra, there exists for each 
$0\not=x\in U(\mathfrak{l})$ a finite dimensional semisimple representation $\pi: \mathfrak{l}\rightarrow\mathfrak{gl}(V)$ 
such that $\pi(x)\not=0$ (this follows by noting that in the proof of \cite[Thm. 2.5.7]{Di} the representation may be chosen to be 
semisimple when the Lie algebra is reductive). 
Hence
\[
\sum_jp_{ij}(\lambda)a_{ij}=0
\]
in $U(\mathfrak{k})\otimes U(\mathfrak{g})^{\otimes N}\otimes U(\mathfrak{k})$ for all
$i$ and all $\lambda\in\mathfrak{h}_{\textup{reg}}^*$.
By the linear independence of $\{a_{ij}\}_{j}$,
we get $p_{ij}(\lambda)=0$ for all $i,j$ and all $\lambda\in \mathfrak{h}_{\textup{reg}}^*$,
hence $p_{ij}=0$ for all $i,j$. This completes the proof of the theorem.

\end{proof}

\subsection{Folded dynamical trigonometric $r$-and $k$-matrices}\label{inteq}

We end this section by discussing the reformulation of the commutator relations
\[
[\mathcal{D}_i,\mathcal{D}_j]=0
\]
in 
$\mathbb{D}_{\mathcal{R}}\otimes U(\mathfrak{k})\otimes U(\mathfrak{g})^{\otimes N}
\otimes U(\mathfrak{k})$ for $1\leq i,j\leq N$
in terms of explicit consistency conditions for the constituents $r^{\pm}$ and $\kappa$
of the asymptotic boundary KZB operators $\mathcal{D}_i$ (see \eqref{bKZBoper}).

Before doing so, we first discuss as a warm-up the situation for the usual asymptotic KZB equations (see \cite{ES} and references therein), which we will construct from an appropriate "universal" version of the operators that are no longer integrable. Recall that $\Delta^{(N-1)}: U(\mathfrak{g})\rightarrow U(\mathfrak{g})^{\otimes N}$ is the $(N-1)$th iterated comultiplication of $U(\mathfrak{g})$.

\begin{prop}
Fix $\widehat{r}\in \mathcal{R}\otimes U(\mathfrak{g})^{\otimes 2}$ satisfying the invariance
property
\begin{equation}\label{invprop}
\lbrack\Delta(h), \widehat{r}\,\rbrack=0\qquad \forall\, h\in\mathfrak{h}.
\end{equation}
For $N\geq 2$ and $1\leq i\leq N$ write,
\[
\widehat{\mathcal{D}}_i^{(N)}:=E_i-
\sum_{s=1}^{i-1}\widehat{r}_{si}+\sum_{s=i+1}^N\widehat{r}_{is}\in
\mathbb{D}_{\mathcal{R}}\otimes U(\mathfrak{g})^{\otimes N}
\]
with $E=\sum_{k=1}^\rr\partial_{x_k}\otimes x_k$, see \eqref{E}.
The following two statements are equivalent.
\begin{enumerate}
\item[{\bf a.}] For $N\geq 2$ and $1\leq i\not=j\leq N$,
\begin{equation}\label{gencomm}
\lbrack \widehat{\mathcal{D}}_i^{(N)},\widehat{\mathcal{D}}_j^{(N)}\rbrack=
-\sum_{k=1}^r\partial_{x_k}(\widehat{r}_{ij})\Delta^{(N-1)}(x_k)
\end{equation}
in $\mathbb{D}_{\mathcal{R}}\otimes U(\mathfrak{g})^{\otimes N}$.
\item[{\bf b.}] $\widehat{r}$ is a solution of the classical dynamical Yang-Baxter equation,
\begin{equation}\label{cdYBe}
\begin{split}
\sum_{k=1}^\rr\bigl((x_k)_3\partial_{x_k}(\widehat{r}_{12})-
&(x_k)_2\partial_{x_k}(\widehat{r}_{13})+(x_k)_1\partial_{x_k}(\widehat{r}_{23})\bigr)\\
&\qquad+\lbrack \widehat{r}_{12},\widehat{r}_{13}\rbrack
+\lbrack \widehat{r}_{12},\widehat{r}_{23}\rbrack+\lbrack \widehat{r}_{13},\widehat{r}_{23}
\rbrack=0
\end{split}
\end{equation}
in $\mathcal{R}\otimes U(\mathfrak{g})^{\otimes 3}$.
\end{enumerate}
\end{prop}
\begin{proof}
By direct computations,
\[
\lbrack \widehat{\mathcal{D}}_1^{(2)},\widehat{\mathcal{D}}_2^{(2)}\rbrack=
-\sum_{k=1}^\rr\partial_{x_k}(\widehat{r})\Delta(x_k)
\]
in $\mathbb{D}_{\mathcal{R}}\otimes U(\mathfrak{g})^{\otimes 2}$ and
\begin{equation*}
\begin{split}
\lbrack \widehat{\mathcal{D}}_1^{(3)},\widehat{\mathcal{D}}_2^{(3)}\rbrack=
&-\sum_{k=1}^\rr\partial_{x_k}(\widehat{r}_{12})\Delta^{(2)}(x_k)
+ \lbrack \widehat{r}_{12},\widehat{r}_{13}\rbrack
+\lbrack \widehat{r}_{12},\widehat{r}_{23}\rbrack+\lbrack \widehat{r}_{13},\widehat{r}_{23}
\rbrack\\
&+
\sum_{k=1}^\rr\bigl((x_k)_3\partial_{x_k}(\widehat{r}_{12})-
(x_k)_2\partial_{x_k}(\widehat{r}_{13})+(x_k)_1\partial_{x_k}(\widehat{r}_{23})\bigr)
\end{split}
\end{equation*}
in $\mathbb{D}_{\mathcal{R}}\otimes U(\mathfrak{g})^{\otimes 3}$.
Hence {\bf a} implies {\bf b}.

It is a straightforward but tedious computation to show that classical dynamical Yang-Baxter equation
\eqref{cdYBe} implies \eqref{gencomm} for all $N\geq 2$ and all $1\leq i\not=j\leq N$.
\end{proof}
For instance, $\widehat{r}(h):=r(h/2)$ ($h\in\mathfrak{h}$) with $r$ Felder's $r$-matrix \eqref{Felderr} satisfies the classical dynamical Yang-Baxter equation \eqref{cdYBe} as well as the invariance condition \eqref{invprop}. The same holds true for $\widehat{r}=2r$.
\begin{cor}[KZB operators]
Let $N\geq 2$. Let $U_1,\ldots,U_N$ be finite dimensional $\mathfrak{g}$-modules and write
$\mathbf{U}:=U_1\otimes\cdots\otimes U_N$ as before. Let 
$\widehat{r}\in \mathcal{R}\otimes U(\mathfrak{g})^{\otimes 2}$ be a solution of the
classical dynamical Yang-Baxter equation \eqref{cdYBe} satisfying the invariance property
\eqref{invprop}. Define differential operators $\widehat{\mathcal{D}}_i^{\mathbf{U}}\in
\mathbb{D}_{\mathcal{R}}\otimes \textup{End}\bigl(\mathbf{U}[0]\bigr)$ for $i=1,\ldots,N$ by
\[
\widehat{\mathcal{D}}_i^{\mathbf{U}}:=\widehat{\mathcal{D}}_i^{(N)}|_{\mathbf{U}[0]}.
\]
Then $\lbrack \widehat{\mathcal{D}}_i^{\mathbf{U}},\widehat{\mathcal{D}}_j^{\mathbf{U}}\rbrack=0$
in $\mathbb{D}_{\mathcal{R}}\otimes \textup{End}\bigl(\mathbf{U}[0]\bigr)$ for $i,j=1,\ldots,N$.
\end{cor}
\begin{rema}\label{ESremark}
Let $\lambda\in\mathfrak{h}_{\textup{reg}}^*$. Let $\mathbf{u}=u_1\otimes\cdots\otimes u_N\in\mathbf{U}[0]$ with $u_i\in U_i[\mu_i]$ and $\sum_{j=1}^N\mu_j=0$, and write
$\lambda_i:=\lambda-\mu_{i+1}-\cdots-\mu_N$ ($i=1,\ldots,N-1$) and $\lambda_N:=\lambda$.
By \cite{ES,E}, the weighted trace of the product $\Psi_\lambda^{\mathbf{J}(\lambda)\mathbf{u}}$
of the $N$ vertex operators $\Psi_{\lambda_i}^{u_i}\in\textup{Hom}_{\mathfrak{g}}(M_{\lambda_i}, M_{\lambda_{i-1}}\otimes U_i)$ are common eigenfunctions of the asymptotic KZB operators 
$\widehat{\mathcal{D}}_i^{\mathbf{U}}$ \textup{(}$1\leq i\leq N$\textup{)} with $\widehat{r}(h)=r(h/2)$ and $r$ Felder's $r$-matrix \eqref{Felderr}.
\end{rema}

Now we prove the analogous result for asymptotic boundary KZB type operators. This time the universal versions of the asymptotic boundary KZB operators themselves will already be integrable. This is because we are considering asymptotic boundary KZB operators associated to {\it split} Riemannian symmetric pairs $G/K$ (note that the representation theoretic context from Remark \ref{ESremark} relates the asymptotic KZB operators to the group $G$ viewed as the symmetric space
$G\times G/\textup{diag}(G)$, with $\textup{diag}(G)$ the group $G$ diagonally embedding into 
$G\times G$). 
\begin{prop}[general asymptotic boundary KZB operators]\label{bKZBopergeneral}
Let $A_\ell$ and $A_r$ be two complex unital associative algebras.
Let $\widetilde{r}^{\,\pm}\in 
\mathcal{R}\otimes U(\mathfrak{g})^{\otimes 2}$ and $\widetilde{\kappa}\in
\mathcal{R}\otimes A_\ell\otimes U(\mathfrak{g})\otimes A_r$ and suppose that
\begin{equation}\label{invpropb}
\lbrack h\otimes 1, \widetilde{r}^{\,+}\rbrack=\lbrack 1\otimes h, \widetilde{r}^{\,-}\rbrack\qquad
\forall\, h\in\mathfrak{h}.
\end{equation}
Write for $N\geq 2$ and $1\leq i\leq N$,
\[
\widetilde{\mathcal{D}}_i^{(N)}:=E_i-
\sum_{s=1}^{i-1}\widetilde{r}_{si}^{\,+}-\widetilde{\kappa}_i
-\sum_{s=i+1}^N\widetilde{r}_{is}^{\,-}
\in\mathbb{D}_{\mathcal{R}}\otimes A_\ell\otimes U(\mathfrak{g})^{\otimes N}\otimes A_r
\]
with $E$ given by \eqref{E} and with the indices indicating in which tensor components
of $U(\mathfrak{g})^{\otimes N}$ the $U(\mathfrak{g})$-components of $\widetilde{r}^{\,\pm}$ and
$\widetilde{\kappa}$ are placed. The following statements are equivalent.
\begin{enumerate}
\item[{\bf a.}] For all $N\geq 2$ and all $1\leq i,j\leq N$,
\[
\lbrack \widetilde{\mathcal{D}}_i^{(N)},\widetilde{\mathcal{D}}_j^{(N)}\rbrack=0
\]
in $\mathbb{D}_{\mathcal{R}}\otimes A_\ell\otimes U(\mathfrak{g})^{\otimes N}\otimes A_r$.
\item[{\bf b.}] $\widetilde{r}^+$ and $\widetilde{r}^-$ are solutions of the following three
coupled classical dynamical Yang-Baxter equations,
\begin{equation}\label{mixedcdYBe}
\begin{split}
\sum_{k=1}^\rr\bigl((x_k)_1\partial_{x_k}(\widetilde{r}_{23}^{\,-})-
(x_k)_2\partial_{x_k}(\widetilde{r}_{13}^{\,-})\bigr)&=
\lbrack \widetilde{r}_{13}^{\,-},\widetilde{r}_{12}^{\,+}\rbrack
+\lbrack \widetilde{r}_{12}^{\,-},\widetilde{r}_{23}^{\,-}\rbrack+\lbrack \widetilde{r}_{13}^{\,-},
\widetilde{r}_{23}^{\,-}\rbrack,\\
\sum_{k=1}^\rr\bigl((x_k)_1\partial_{x_k}(\widetilde{r}_{23}^{\,+})-
(x_k)_3\partial_{x_k}(\widetilde{r}_{12}^{\,-})\bigr)&=\lbrack \widetilde{r}_{12}^{\,-},\widetilde{r}_{13}^{\,+}\rbrack
+\lbrack \widetilde{r}_{12}^{\,-},\widetilde{r}_{23}^{\,+}\rbrack+\lbrack \widetilde{r}_{13}^{\,-},
\widetilde{r}_{23}^{\,+}\rbrack,\\
\sum_{k=1}^\rr\bigl((x_k)_2\partial_{x_k}(\widetilde{r}_{13}^{\,+})-
(x_k)_3\partial_{x_k}(\widetilde{r}_{12}^{\,+})\bigr)&=
\lbrack \widetilde{r}_{12}^{\,+},\widetilde{r}_{13}^{\,+}\rbrack
+\lbrack \widetilde{r}_{12}^{\,+},\widetilde{r}_{23}^{\,+}\rbrack+\lbrack \widetilde{r}_{23}^{\,-},
\widetilde{r}_{13}^{\,+}\rbrack
\end{split}
\end{equation}
in $\mathcal{R}\otimes U(\mathfrak{g})^{\otimes 3}$,
and $\widetilde{\kappa}$ is a solution of the associated classical dynamical reflection equation
\begin{equation}\label{mixedcdRe}
\sum_{k=1}^\rr\bigl((x_k)_1\partial_{x_k}(\widetilde{\kappa}_2+\widetilde{r}^{+})-
(x_k)_2\partial_{x_k}(\widetilde{\kappa}_1+\widetilde{r}^{-})\bigr)
=
\lbrack\widetilde{\kappa}_1+\widetilde{r}^{-}, \widetilde{\kappa}_2+\widetilde{r}^{+}\rbrack
\end{equation}
in $\mathcal{R}\otimes A_\ell\otimes U(\mathfrak{g})^{\otimes 2}\otimes A_r$.
\end{enumerate}
\end{prop}
\begin{proof}
By direct computations, $\lbrack \widetilde{\mathcal{D}}_1^{(2)},\widetilde{\mathcal{D}}_2^{(2)}\rbrack=0$ is equivalent to the dynamical reflection equation \eqref{mixedcdRe} and
$\lbrack \widetilde{\mathcal{D}}_i^{(3)},\widetilde{\mathcal{D}}_j^{(3)}\rbrack=0$ for
$(i,j)=(1,2), (1,3), (2,3)$ is equivalent to the three coupled classical dynamical Yang-Baxter equations.
Hence {\bf a} implies {\bf b}. Conversely, a direct but tedious computation shows that the three coupled classical dynamical Yang-Baxter equations and the associated classical dynamical reflection equation
imply $\lbrack \widetilde{\mathcal{D}}_i^{(N)},\widetilde{\mathcal{D}}_j^{(N)}\rbrack=0$ for 
$N\geq 2$ and $1\leq i,j\leq N$. 
\end{proof}
Applied to the asymptotic boundary KZB operators $\mathcal{D}_i$ ($1\leq i\leq N$) given by \eqref{bKZBoper},
we obtain from Theorem \ref{consistentoperators} with $A_\ell=U(\mathfrak{k})=A_r$  
the following main result of this subsection.
\begin{thm}\label{thmDEF}
The folded dynamical $r$-matrices $r^{\pm}\in\mathcal{R}\otimes \mathfrak{g}^{\otimes 2}$ \textup{(}see
\eqref{explicitsigmatauexpl}\textup{)} and the dynamical $k$-matrix $\kappa\in\mathcal{R}\otimes
\mathfrak{k}\otimes U(\mathfrak{g})\otimes\mathfrak{k}$ \textup{(}see \eqref{kdef}\textup{)}
satisfy the coupled classical dynamical Yang-Baxter equations \eqref{mixedcdYBe} in
$\mathcal{R}\otimes U(\mathfrak{g})^{\otimes 3}$ and the associated classical dynamical reflection equation \eqref{mixedcdRe} in $\mathcal{R}\otimes U(\mathfrak{k})\otimes U(\mathfrak{g})^{\otimes 2}\otimes U(\mathfrak{k})$.
\end{thm}
A direct algebraic proof of Theorem \ref{thmDEF}, which does not resorting to the commutativity of the asymptotic boundary KZB operators,  is given in \cite{St2}.

\end{document}